\documentclass[a4paper]{amsart}

\usepackage[dvipsnames,svgnames,x11names,hyperref]{xcolor}
\usepackage[colorlinks,citecolor=Mahogany,linkcolor=Mahogany,urlcolor=Mahogany,filecolor=Mahogany]{hyperref}

\usepackage{amsmath,amstext,amssymb,mathrsfs,amscd,amsthm,indentfirst}
\usepackage{amsfonts}
\usepackage{mathtools}
\usepackage{stmaryrd}
\usepackage{bbm}
\usepackage{units}

\usepackage{booktabs}
\usepackage{verbatim}
\usepackage{enumerate}

\usepackage{graphicx}

\usepackage{tikz,tikz-cd}
\usetikzlibrary{matrix,calc,positioning,arrows,decorations.pathreplacing,fit,patterns}

\numberwithin{equation}{section}

\newcommand{\cat}[1]{\mathsf{#1}}

\newcommand{\mr}[1]{{\rm #1}}
\newcommand{\bunit}{\mathbbm{1}}
\newcommand{\fgr}[1]{\| {#1} \|}

\newcommand{\grr}{\mathrm{gr}}

\newcommand{\CircNum}[1]{\ooalign{\hfil\raise .00ex\hbox{\scriptsize #1}\hfil\crcr\mathhexbox20D}}

\newcommand{\bA}{\mathbb{A}}

\newcommand{\bC}{\mathbb{C}}

\newcommand{\bF}{\mathbb{F}}

\newcommand{\bH}{\mathbb{H}}

\newcommand{\bL}{\mathbb{L}}

\newcommand{\bN}{\mathbb{N}}

\newcommand{\bP}{\mathbb{P}}
\newcommand{\bQ}{\mathbb{Q}}
\newcommand{\bR}{\mathbb{R}}

\newcommand{\bZ}{\mathbb{Z}}

\newcommand{\gA}{\bold{A}}
\newcommand{\gB}{\bold{B}}
\newcommand{\gC}{\bold{C}}
\newcommand{\gD}{\bold{D}}
\newcommand{\gE}{\bold{E}}

\newcommand{\gH}{\bold{H}}

\newcommand{\gK}{\bold{K}}

\newcommand{\gM}{\bold{M}}

\newcommand{\gR}{\bold{R}}

\newcommand{\gT}{\bold{T}}
\newcommand{\gU}{\bold{U}}

\newcommand{\gX}{\bold{X}}

\newcommand{\gZ}{\bold{Z}}

\newcommand{\cA}{\mathcal{A}}
\newcommand{\cB}{\mathcal{B}}
\newcommand{\cC}{\mathcal{C}}

\newcommand{\cO}{\mathcal{O}}

\newcommand{\cS}{\mathcal{S}}

\newcommand{\cX}{\mathcal{X}}
\newcommand{\cY}{\mathcal{Y}}

\newcommand{\bk}{\mathbbm{k}}

\newcommand\lra{\longrightarrow}

\newcommand\Diff{\mathrm{Diff}}

\newcommand\colim{\operatorname*{colim}}

\newcommand{\fS}{\mathfrak{S}}

\newcommand{\R}{\bR}

\newcommand{\Z}{\mathbbm{Z}}

\renewcommand{\epsilon}{\varepsilon}

\newcommand{\suspsplit}{T^{E_1}}
\renewcommand{\split}{S^{E_1}}
\newcommand{\Steinb}{\mathrm{St}}
\newcommand{\Alg}{\cat{Alg}}

\usepackage{amsthm}
\theoremstyle{plain}
\newtheorem{MainThm}{Theorem}

\newtheorem{theorem}{Theorem}[section]

\newtheorem{proposition}[theorem]{Proposition}
\newtheorem{lemma}[theorem]{Lemma}

\newtheorem{corollary}[theorem]{Corollary}

\theoremstyle{definition}
\newtheorem{definition}[theorem]{Definition}

\newtheorem{example}[theorem]{Example}

\theoremstyle{remark}

\newtheorem{remark}[theorem]{Remark}
\newtheorem*{remark*}{Remark}

\title{$E_2$-cells and mapping class groups}

\author{S{\o}ren Galatius}
\email{galatius@math.ku.dk}
\address{Department of Mathematics\\
	University of Copenhagen\\
	Universitetsparken 5\\
	2100 Copenhagen \O \\
	Denmark}

\author{Alexander Kupers}
\email{kupers@math.harvard.edu}
\address{Department of Mathematics \\
	Harvard University\\
	One Oxford Street \\
	Cambridge MA, 02138 \\USA}

\author{Oscar Randal-Williams}
\email{o.randal-williams@dpmms.cam.ac.uk}
\address{Department of Pure Mathematics and Mathematical Statistics\\
	University of Cambridge\\
	Wilberforce Road\\
	Cambridge CB3 0WB\\
	UK}

\date{\today}

\begin{document}
	
\begin{abstract}
We prove a new kind of stabilisation result, ``secondary homological stability,'' for the homology of mapping class groups of orientable surfaces with one boundary component. These results are obtained by constructing CW approximations to the classifying spaces of these groups, in the category of $E_2$-algebras, which have no $E_2$-cells below a certain vanishing line.
\end{abstract}

\maketitle

\section{Introduction} 

The mapping class group $\Gamma_{g,1}$ of a surface $\Sigma_{g,1}$ of genus $g$ with one boundary component is the group of diffeomorphisms of $\Sigma_{g,1}$ which fix the boundary, taken up to isotopy. As well as its evident position in geometric topology, this group arises as the fundamental group of the associated moduli space $\mathcal{M}(\Sigma_{g,1})$ of Riemann surfaces having the topological type of $\Sigma_{g,1}$. In fact, this moduli space is an Eilenberg--Mac~Lane space for $\Gamma_{g,1}$, so the group cohomology of $\Gamma_{g,1}$ coincides with the cohomology of the space $\mathcal{M}(\Sigma_{g,1})$. It is a motivating problem in both topology and algebraic geometry to understand the (co)homology of such spaces.

By considering $\Sigma_{g-1,1}$ as lying inside $\Sigma_{g,1}$, we may extend diffeomorphisms of $\Sigma_{g-1,1}$ to $\Sigma_{g,1}$, defining group homomorphisms
\[s \colon \Gamma_{g-1,1} \lra \Gamma_{g,1}\]
called \emph{stabilisation maps}. One may ask about the effect of these maps on homology: they were shown by Harer \cite{HarerStab} to induce an isomorphism on homology in a range of degrees increasing with $g$. This stability range was later improved, and the best known stability range has been established by Boldsen \cite{Boldsen} (see also \cite{R-WResolution,WahlSurvey}): the induced map $s_*$ on homology is an epimorphism in homological degrees $* \leq \frac{2g-2}{3}$ and an isomorphism in degrees $* \leq \frac{2g-4}{3}$. In particular, it follows that the relative homology groups satisfy $H_d(\Gamma_{g,1}, \Gamma_{g-1,1};\bZ)=0$ for $d \leq \tfrac{2g-2}{3}$.

\begin{comment}

Pick a compact connected oriented surface of genus $g$ with one boundary component and denote it $\Sigma_{g,1}$, and write $\mr{Diff}_\partial(\Sigma_{g,1})$ for the topological group of diffeomorphisms of $\Sigma_{g,1}$ fixing a neighborhood of the boundary pointwise. Let us write
\[\Gamma_{g,1} \coloneqq \pi_0(\Diff_\partial(\Sigma_{g,1}))\]
for the mapping class group of $\Sigma_{g,1}$, and also pick embeddings $\Sigma_{g-1,1} \subset \Sigma_{g,1}$; by extending diffeomorphisms by the identity, these induce injective group homomorphisms
\[s \colon \Gamma_{g-1,1} \lra \Gamma_{g,1},\]
called \emph{stabilisation maps}. These maps were first shown to induce an isomorphism on homology in a range of degrees increasing with $g$ by Harer \cite{HarerStab}. The best previously known stability range has been established by Boldsen \cite{Boldsen} (see also \cite{R-WResolution,WahlSurvey}): the induced map $s_*$ on homology is an epimorphism in homological degrees $* \leq \frac{2g-2}{3}$ and an isomorphism in degrees $* \leq \frac{2g-4}{3}$. In particular, it follows that the relative homology groups satisfy $H_d(\Gamma_{g,1}, \Gamma_{g-1,1};\bZ)=0$ for $d \leq \tfrac{2g-2}{3}$.

\end{comment}

Our main result is a new type of stability theorem for these relative homology groups, in a larger range than that in which ordinary stability makes them vanish.

\begin{MainThm}\label{thm:A}
There are maps
\[\varphi_* \colon H_{d-2}(\Gamma_{g-3,1}, \Gamma_{g-4,1};\bZ) \lra H_d(\Gamma_{g,1}, \Gamma_{g-1,1};\bZ),\]
constructed below, which are
epimorphisms for $d \leq \tfrac{3g-1}{4}$ and isomorphisms for $d
\leq \tfrac{3g-5}{4}$. With rational coefficients they are epimorphisms for $d \leq \tfrac{4g-1}{5}$ and isomorphisms for $d \leq \tfrac{4g-6}{5}$.
\end{MainThm}

We call this result \emph{secondary homological stability}; qualitatively it can be expressed as ``the failure of homological stability is itself stable.'' We will show that such secondary homological stability extends to surfaces with several boundaries and with punctures, and to homology with twisted coefficients satisfying a strong polynomiality condition. 

In light of Harer's stability theorem it became reasonable to ask for a computation of the limiting homology groups $\colim_{g \to \infty}H_d(\Gamma_{g,1};\bZ)$, which was accomplished by Madsen and Weiss \cite{MadsenWeiss}; in light of Theorem \ref{thm:A} it becomes reasonable and extremely interesting to ask for a computation of the limiting homology groups
\[\colim_{k \to \infty}H_{d+2k}(\Gamma_{g+3k,1}, \Gamma_{g-1+3k,1};\bZ)\]
formed with respect to these new stabilisation maps.

We can say a little bit about these relative groups, although not far beyond their vanishing range.
\begin{MainThm}\label{thm:B}\mbox{} 
\begin{enumerate}[(i)]
\item\label{it:B1} $H_d(\Gamma_{g,1}, \Gamma_{g-1,1};\bZ)=0$ for $d \leq \tfrac{2g-1}{3}$.

\item\label{it:B2} $H_{2k}(\Gamma_{3k,1}, \Gamma_{3k-1,1};\bZ) \cong \bZ$ for $k \geq 0$.

\item\label{it:B3} $H_{2k+1}(\Gamma_{3k+1,1}, \Gamma_{3k,1};\bQ)=0$ for $k \geq 1$ (but $H_1(\Gamma_{1,1},\Gamma_{0,1};\bQ) = \bQ$).

\item\label{it:B4} $H_{2k+2}(\Gamma_{3k+2,1}, \Gamma_{3k+1,1};\bQ) \neq 0$ for $k \geq 0$.

\item\label{it:B5} $H_{2k+2}(\Gamma_{3k+1,1}, \Gamma_{3k,1};\bQ) \neq 0$ for $k \geq 3$ (but we have $H_{2}(\Gamma_{1,1}, \Gamma_{0,1};\bQ) = 0$ and $H_4(\Gamma_{4,1},\Gamma_{3,1};\bQ) \neq 0$).
\end{enumerate}
\end{MainThm}

The vanishing range for $H_d(\Gamma_{g,1}, \Gamma_{g-1,1};\bZ)$ in Theorem~\ref{thm:B} (\ref{it:B1}) is slightly better than that of \cite{Boldsen}. Theorem~\ref{thm:B} (\ref{it:B2}), (\ref{it:B4}) and (\ref{it:B5}) determine the first non-zero rational relative homology group in every genus except $g=7$.  We remark that a rational homology version of (\ref{it:B2}) was obtained by Harer in the unpublished preprint \cite{harer1993improved}.
In Section \ref{sec:table} we will describe the concrete impact of our results on the rational homology of $\Gamma_{g,1}$ for $g \leq 9$, which is summarised in Figure \ref{fig:rat}. 

The maps $\varphi_*$ which feature in Theorem \ref{thm:A} will be defined in Section~\ref{sec:integr-second-homol}. They will be constructed by obstruction theory, and as usual such constructions do not typically produce a single map but rather a set of choices of maps: in this case the choices form an $H_3(\Gamma_{4,1}, \Gamma_{3,1};\bZ)$-torsor. We will not attempt to pin down a preferred choice, though in principle this could be done by picking a null-homotopy of a certain map which we know to be null-homotopic. %After making one choice in low dimensions a particular map $\phi_*$ is completely determined; however, we will not attempt to pin down a preferred choice. 
In this introduction we shall be content with describing $\varphi_*$ after inverting 10, in which case there is a preferred choice. Pick an element $\lambda \in H_2(\Gamma_{3,1};\bZ)$ on which the Miller--Morita--Mumford class $\kappa_1 \colon H_2(\Gamma_{3,1};\bZ) \to \Z$ evaluates to $12$. This is known to be the smallest possible positive value of $\kappa_1$ %(e.g.\ using the signature theorem, \cite[Satz 2]{Meyer}, and the fact that $\Gamma_{3,1} \to \Gamma_3$ is surjective on second homology \cite[Theorem 3.8]{Boldsen}), 
and to characterise the class $\lambda$ modulo the (2-torsion) image of $H_2(\Gamma_{2,1};\bZ) \to H_2(\Gamma_{3,1};\bZ)$.  Boundary connected sum induces an operation
\begin{equation*}
  H_*(\Gamma_{3,1};\bZ) \otimes H_*(\Gamma_{g-3,1},\Gamma_{g-4,1};\bZ) \lra H_*(\Gamma_{g,1},\Gamma_{g-1,1};\bZ),
\end{equation*}
and substituting $\lambda \in H_*(\Gamma_{3,1};\bZ)$ in the first entry we obtain a homomorphism which we shall also denote $\lambda \colon H_{d-2}(\Gamma_{g-3,1},\Gamma_{g-4,1};\bZ) \to H_d(\Gamma_{g,1},\Gamma_{g-1,1};\bZ)$.  After inverting 10, we may then take
\begin{equation*}
(\varphi_*)\otimes \Z[\tfrac1{10}] = \lambda/10 \colon H_{d-2}(\Gamma_{g-3,1}, \Gamma_{g-2,1};\bZ[\tfrac{1}{10}]) \lra H_d(\Gamma_{g,1}, \Gamma_{g-1,1};\bZ[\tfrac{1}{10}]).
\end{equation*}

\setcounter{tocdepth}{1}
\tableofcontents

\section{Overview of $E_2$-cells and $E_2$-homology}
\label{sec:homology-theory-e_2}
The purpose of this section is to explain the methods used to prove Theorems \ref{thm:A} and \ref{thm:B}. We aim for an informal discussion of the ideas involved, and refer to \cite{e2cellsI} for a more formal discussion as well as for proofs; we shall refer to things labelled $X$ in \cite{e2cellsI} as $E_k$.$X$ throughout this paper. We hope this informal treatment may allow for an independent first reading of the present paper, at least for the reader with some prior exposure to the notion of an \emph{$E_2$-algebra}. We shall work in the category of simplicial $\bk$-modules. Here, $\mathbbm{k}$ is a commutative ring, but not much is lost by assuming $\mathbbm{k}$ is a field; if it is a field of characteristic zero then not much is lost by considering instead the notion of Gerstenhaber algebras in chain complexes over $\bk$.  At the end of this introduction we shall offer some remarks for the reader more familiar with that notion.

In order to keep track of genus, we shall be working in the category
of $\bN$-graded simplicial $\mathbbm{k}$-modules, i.e.\ each object $M$ consists of a simplicial $\mathbbm{k}$-module $M_\bullet(g)$ for each
$g \in \bN$ (for us $\bN$ always includes $0$). The tensor product of two objects $M$ and $N$ is given by the formula
\begin{equation*}
  (M \otimes N)_p(g) = \bigoplus_{a+b = g} M_p(a) \otimes_\mathbbm{k}
  N_p(b).
\end{equation*}

The \emph{little $2$-cubes} operad has $n$-ary operations given by rectilinear embeddings $\{1, \dots, n\} \times I^2 \to I^2$ such that the interiors of the images of the cubes are disjoint. Here we shall work with the \emph{non-unitary} version of this operad, meaning that the space of 0-ary operations is empty. Take the singular simplicial set to get an operad in simplicial sets, and take the free $\mathbbm{k}$-module to get an operad in simplicial $\mathbbm{k}$-modules. The resulting operad, made $\bN$-graded by
concentrating it in degree 0, shall be denoted $\mathcal{C}_2$.  In this section, ``$E_2$-algebra'' shall mean an algebra over this operad (i.e.\ a ``non-unital $E_2$-algebra in $\bN$-graded simplicial $\mathbbm{k}$-modules''). Note that the homotopy groups of a simplicial $\bk$-module are isomorphic to the homology groups of the chain complex associated it to by Dold--Kan.

The module of \emph{indecomposables} of an $E_2$-algebra $\gR$ is defined by the exact sequence of $\bN$-graded simplicial $\mathbbm{k}$-modules
\begin{equation*}
  \bigoplus_{n \geq 2} \mathcal{C}_2(n) \otimes \gR^{\otimes n} \lra \gR
  \lra Q^{E_2}(\gR) \lra 0,
\end{equation*}
with leftmost map induced by the $E_2$-algebra structure on $\gR$.

The functor $\gR \mapsto Q^{E_2}(\gR)$ is not homotopy invariant, but has a derived functor $Q^{E_2}_{\mathbbm{L}}$, see Section $E_k$.13 for more details on its definition and calculation.  In particular, we shall use an insight going back to at least Getzler--Jones \cite{GetzlerJones} that there is a formula for the derived indecomposables in terms of a $2$-fold iterated bar construction.

Let us now assume $\mathbbm{k}$ is a field, and define numerical
invariants of $\gR$ by
\begin{equation*}
  b_{g,d}^{E_2}(\gR) \coloneqq \dim_\mathbbm{k} \pi_d((Q^{E_2}_\mathbbm{L}(\gR))(g)) \in \bN
  \cup\{\infty\},
\end{equation*}
which we call \emph{$E_2$-Betti numbers}. 

A key result, proved in Theorem $E_k$.11.21, is that under mild conditions on $\gR$ there exists a multiplicative filtration on (an $E_2$-algebra equivalent to) $\gR$, whose associated graded is a \emph{free} $E_2$-algebra on an $\bN$-graded simplicial $\bk$-module with precisely $b_{g,d}^{E_2}(\gR)$ many non-degenerate $d$-simplices in grading $g$.  Now, the homotopy groups of a free $E_2$-algebra are completely known: if $\mathbbm{k}$ has characteristic 0, then the homotopy groups of a free $E_2$-algebra are a free Gerstenhaber algebra on the set of generators; if $\mathbbm{k}$ has characteristic $p > 0$, then the homotopy groups are explicitly described in F.\ Cohen's thesis (published in \cite{CLM}) in terms of iterated Dyer--Lashof operations applied to iterated Browder brackets of generators. Therefore, if we can calculate the numbers $b_{g,d}^{E_2}(\gR)$ for a given $E_2$-algebra $\gR$, then we
obtain a multiplicative spectral sequence converging to $\pi_*(\gR)$
with a known $E^1$-page.  For many purposes an estimate of the
numbers $b_{g,d}^{E_2}(\gR)$, or even just their vanishing in a certain
region, will suffice.

A slightly more refined statement uses the notion of a \emph{CW
$E_2$-algebra}, built out of free $E_2$-algebras by iteratively attaching cells in the category of $E_2$-algebras in order of dimension. The data for a cell attachment to an $E_2$-algebra $\gR$ is an ``attaching map'' $e\colon S^{g,d-1}_\bk \to \gR$, which by an adjunction is the same as a map $\bk[\partial \Delta^{d}] \to \gR(g)$ of simplicial $\bk$-modules. The inclusion $\partial \Delta^{d} \hookrightarrow \Delta^d$ induces a map $S^{g,d-1}_\bk \to D^{g,d}_\bk$, and to attach a cell we form the pushout in $\Alg_{E_2}(\cat{sMod}_\bk^\bN)$
\[\begin{tikzcd} \gE_2(S^{g,d-1}_\bk) \rar \dar & \gR \dar \\
\gE_2(D^{g,d}_\bk) \rar & \gR \cup^{E_2}_e \gD^{g,d},\end{tikzcd}\]
where $\gE_2(-)$ denotes the free $E_2$-algebra construction. 

An equivalence from a CW $E_2$-algebra to $\gR$ is called a CW approximation to $\gR$, and a key result is that if $\gR(0) \simeq 0$, then $\gR$ admits a CW approximation built using precisely $b_{g,d}^{E_2}$ many $d$-cells in grading $g$.  The above-mentioned filtration and the resulting spectral sequence is then generated multiplicatively by
``giving the $d$-cells filtration $d$'' (see Section $E_k$.11
%\ref{I.sec:additive-case}
for a more precise discussion of what this means).

\vspace{1em}

Theorems \ref{thm:A} and \ref{thm:B} are obtained by applying these ideas to a model of the space $\bigsqcup_{g \geq 1} B\Gamma_{g,1}$ which is an algebra over the (non-unital) little 2-cubes operad, so that its singular chains (regarded as a simplicial $\mathbbm{k}$-module) will be an $E_2$-algebra $\gR$ in the above sense when graded by genus. An application of the same ideas to general linear groups is given in \cite{e2cellsIII}.

The starting point is an important estimate: $\smash{b_{g,d}^{E_2}}(\gR) = 0$ for
$d < g-1$, established in Proposition~\ref{prop:MCGE2vanish}
below. We first prove a vanishing line for the homotopy groups of $\smash{Q^{E_1}_\bL(\gR)}$ by describing it in terms of a bar construction (see Theorem $E_k$.13.7), interpreting this as a certain complex of separating arcs, and finally showing that this is highly-connected using geometric methods. Since $\smash{Q^{E_2}_\bL(\gR)}$ may be similarly described as a double bar construction, we can transfer the vanishing line from $Q^{E_1}_\bL(\gR)$ to $Q^{E_2}_\bL(\gR)$ using a bar spectral sequence (see Theorem $E_k$.14.2), and the dimensions of homotopy groups of $Q^{E_2}_\bL(\gR)$ are exactly the $E_2$-Betti numbers.

The only $(g,d) \in \bN^2$ satisfying both $d \geq g-1$ and
$\frac{d}{g} \leq \frac{3}{4}$ are $(1,0)$, $(2,1)$, $(3,2)$, and $(4,3)$; therefore we conclude that there must exist a CW approximation with no cells in bidegrees $(g,d)$ satisfying $\frac{d}{g} \leq \frac{3}{4}$, except possibly those bidegrees.  We then analyse known low-dimensional calculations of the homology of mapping class groups, and draw finer conclusions about how these and a small number of additional cells are attached to each other, in order to deduce our main results. The main observation is that while a computation of $\pi_{g,d}(\gR)$ requires knowledge of all cells of dimension $\leq d+1$ and grading $\leq g$, proving results related to stability only requires knowledge about those cells in bidegree $(g',d')$ satisfying $\frac{d'}{g'} \leq \frac{d}{g}$. We call the number $\smash{\frac{d}{g}}$ the \emph{slope} of a bidegree $(g,d)$. This connection between $E_2$-cells and homological stability was first indicated in \cite{KupersMiller}.

Finally, let us mention that in order to prove something about
$Q^{E_2}_\mathbbm{L}(\gR)$ for the relevant $E_2$-algebra $\gR$, for example
its vanishing in the $d < g-1$ range, it is helpful to calculate derived indecomposables in more refined categories, such as the category of $E_2$-algebras in $\bN$-graded simplicial sets.  In
fact it is even more convenient to work in a category in which our
$E_2$-algebra is represented by the terminal object, viz.\ the functor
category defined in Section~\ref{sec:mcg-as-e2} below. Vanishing
estimates in that category may be transferred to the more useful
category of $\bN$-graded simplicial $\mathbbm{k}$-modules for formal
reasons, as in Corollaries~\ref{cor:MCGVanishingLine} and \ref{cor:MCG21vanish}.

\begin{remark} As promised, let us briefly explain how some of these notions may be
  reinterpreted in the case $\mathbbm{k}$ is a field of characteristic
  0.  The reader will not find a more formal version of this
  discussion in \cite{e2cellsI}, so it should perhaps be regarded as
  more of a heuristic picture for now.  In the characteristic zero
  case, the category of $\bN$-graded simplicial $\mathbbm{k}$-modules
  may be replaced with $\bN\times \bN$-graded chain complex whose
  boundary map has degree $(0,-1)$, and ``$E_2$-algebra'' may be
  replaced by (non-unital) ``Gerstenhaber algebra''.

  The notion of ``cellular $E_2$-algebra'' may then be replaced by
  ``quasi-free Gerstenhaber algebra'': one that is free as a
  Gerstenhaber algebra in $\bN\times \bN$-graded vector spaces (i.e.\
  after forgetting the boundary map).  The boundary map applied to a
  generator is then a sum of products of iterated brackets of other
  generators, and the quasi-free Gerstenhaber algebra is called 
  \emph{minimal} if the boundary of each generator contains no terms
  which are a non-zero scalar times another generator (i.e.\ only
  non-trivial products or brackets are allowed).  The Betti number
  $b_{g,d}^{E_2}(\gR)$ may then be interpreted as the number of
  generators of bidegree $(g,d)$ in a minimal quasi-free Gerstenhaber
  algebra equivalent to $\gR$.
\end{remark}

\section{Mapping class groups as an $E_2$-algebra}
\label{sec:mcg-as-e2}

Let us choose once and for all a compact oriented smooth surface $\Sigma_{1,1} \subset [0,1]^3$ of genus 1 with one boundary component, equal as a subset to $[0,1]^2 \times \{0\}$ near $\partial [0,1]^3$. We write $\Sigma_{g,1} \subset [0,g] \times [0,1]^2$ for the surface of genus $g$ obtained as the union of $g$ translates of $\Sigma_{1,1}$, as shown in Figure \ref{fig:standardsurface}. 

\begin{figure}[h]
	\begin{tikzpicture}
	\draw (2.35,1.5) -- (4,1.5) -- (2.5,0) -- (0,0) -- (1.5,1.5) -- (1.65,1.5);
	\draw [dotted] (1.65,1.5) -- (2.35,1.5);
	\begin{scope}[xshift=2cm,yshift=1.25cm]
	\draw (-.5,-.5) to[out=0,in=-90] (-.4,0.8) to[out=90,in=180] (0,1.2) to[out=0,in=90] (.4,0.8) to[out=-90,in=180] (.5,-.5);	
	\draw (-.2,0.8) to[out=-90,in=-90] (.2,0.8);
	\draw (-.15,0.75) to[out=90,in=90] (.15,0.75);
	\end{scope}
	\node at (2,0) [below] {$\Sigma_{1,1}$};
	
	\begin{scope}[xshift=4cm]
	\draw (2.35,1.5) -- (4,1.5);
	\draw [dashed] (4,1.5) -- (2.5,0);
	\draw (2.5,0) -- (0,0) -- (1.5,1.5) -- (1.65,1.5);
	\draw [dotted] (1.65,1.5) -- (2.35,1.5);
	\begin{scope}[xshift=2cm,yshift=1.25cm]
	\draw (-.5,-.5) to[out=0,in=-90] (-.4,0.8) to[out=90,in=180] (0,1.2) to[out=0,in=90] (.4,0.8) to[out=-90,in=180] (.5,-.5);	
	\draw (-.2,0.8) to[out=-90,in=-90] (.2,0.8);
	\draw (-.15,0.75) to[out=90,in=90] (.15,0.75);
	\end{scope}
	\end{scope}
	\begin{scope}[xshift=6.5cm]
	\draw (2.35,1.5) -- (4,1.5) -- (2.5,0) -- (0,0);
	\draw (1.5,1.5) -- (1.65,1.5);
	\draw [dotted] (1.65,1.5) -- (2.35,1.5);
	\begin{scope}[xshift=2cm,yshift=1.25cm]
	\draw (-.5,-.5) to[out=0,in=-90] (-.4,0.8) to[out=90,in=180] (0,1.2) to[out=0,in=90] (.4,0.8) to[out=-90,in=180] (.5,-.5);	
	\draw (-.2,0.8) to[out=-90,in=-90] (.2,0.8);
	\draw (-.15,0.75) to[out=90,in=90] (.15,0.75);
	\end{scope}
	\end{scope}
	\node at (6.5,0) [below] {$\Sigma_{2,1}$};
	\end{tikzpicture}
	\caption{The standard surface $\Sigma_{1,1}$, and the surface $\Sigma_{g,1}$ for $g=2$ obtained from two copies of $\Sigma_{1,1}$.}
	\label{fig:standardsurface}
\end{figure}
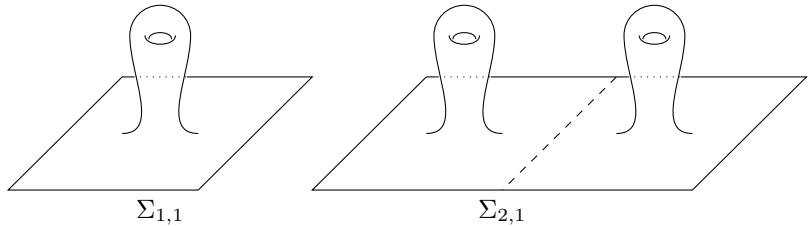

In certain formulae we shall write $e_1 = (1,0,0) \in \R^3$ for the first basis vector and $(\Sigma_{1,1} + n\cdot e_1) \subset [n,n+1] \times [0,1]^2$ for the translate of $\Sigma_{1,1}$; in this notation we may identify $\Sigma_{g,1}$ as
\begin{equation*}
  \Sigma_{g,1} = \bigcup_{n=0}^{g-1} (\Sigma_{1,1} + n\cdot e_1).
\end{equation*}

Let $\Diff_\partial(\Sigma_{g,1})$ denote the group of smooth diffeomorphisms of $\Sigma_{g,1}$ which are equal to the identity on a neighbourhood of $\partial \Sigma_{g,1}$. We shall give this group the $C^\infty$-topology, which makes into a topological group. We then define the \emph{mapping class group} of $\Sigma_{g,1}$ to be the discrete group
\[\Gamma_{g,1} \coloneqq \pi_0(\Diff_\partial(\Sigma_{g,1})),\]
so that in particular $\Gamma_{0,1} = \{1\}$ as a consequence of Smale's theorem \cite[Theorem B]{SmaleSphere}. Note that passing to $\pi_0$ does not lose any homotopical information, as the path components of $\mr{Diff}_\partial(\Sigma_{g,1})$ are contractible by \cite[Th\'eor\`eme 1]{Gramain}.

Let us define a groupoid $\cat{MCG}$ having objects the natural numbers and morphisms
\begin{equation*}
  \cat{MCG}(g,h) \coloneqq
  \begin{cases}
    \Gamma_{g,1} & \text{if $g = h$,}\\
    \varnothing & \text{otherwise}.
  \end{cases}
\end{equation*}
We define a monoidal structure $\oplus$ on $\cat{MCG}$ by $g \oplus h \coloneqq g+h$ on objects, and, for morphisms $\varphi \in \Gamma_{g,1}$ and $\psi \in \Gamma_{h,1}$, by $\varphi \oplus \psi \coloneqq \varphi \cup (\psi + g \cdot e_1)$ as a diffeomorphism of 
\[\Sigma_{g+h, 1} = \Sigma_{g,1} \cup (\Sigma_{h,1} + g \cdot e_1) \subset [0,g+h] \times [0,1]^2.\]
This defines a strict monoidal category, which admits a braiding using the half right-handed Dehn-twist along the boundary. In other words, the braiding is induced by the element in $\pi_1(\cC_2(2)/\Sigma_2)$ switching the two little cubes by moving the second over the first, as in Figure \ref{fig:braiding} (and see Figure \ref{fig:dehn-twist} for a picture of a right-handed Dehn twist). 

\begin{lemma}\label{lem:MonoidalInj}
The homomorphism $-\oplus - \colon \Gamma_{g,1} \times \Gamma_{h,1} \to \Gamma_{g+h,1}$ is injective.
\end{lemma}
\begin{proof}
Let $e \colon [0,1] \to \Sigma_{g+h,1}$ be the embedding $e(t) = (g, t, 0)$, which decomposes $\Sigma_{g+h,1}$ into $\Sigma_{g,1}$ and $\Sigma_{h,1}+g \cdot e_1$. The map $\Diff_\partial(\Sigma_{g+h,1}) \to \mr{Emb}_\partial([0,1], \Sigma_{g+h,1})$ given by $\varphi \mapsto \varphi \circ e$ is a Serre fibration \cite{Lima}, and the fibre over $e$ is $\Diff_\partial(\Sigma_{g,1}) \times \Diff_\partial(\Sigma_{h,1})$. The claim then follows from the long exact sequence on homotopy groups using the fact that $\mr{Emb}_\partial([0,1], \Sigma_{g+h,1})$ has contractible path components \cite[Th\'eorem\`e 5]{Gramain}.
\end{proof}

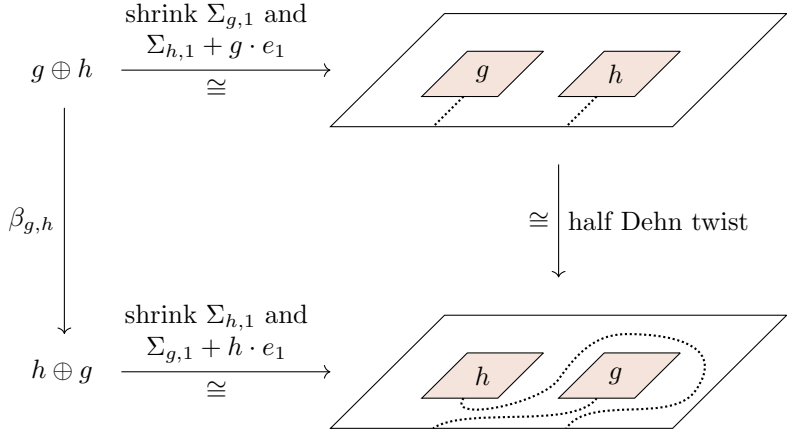
\begin{figure}[h]
	\begin{tikzpicture}
	\begin{scope}
	\node at (-3,.75) [left] {$g \oplus h$};
	\draw [->] (-2.75,.75) -- (0,.75);
	\node at (-1.5,.75) [above,align=center] {shrink $\Sigma_{g,1}$ and\\ $\Sigma_{h,1}+g \cdot e_1$};
	\node at (-1.5,.75) [below] {$\cong$};
	\draw (6,1.5) -- (4.5,0) -- (0,0) -- (1.5,1.5) -- (6,1.5);
	\draw [xscale=.4,yscale=.4, xshift=3cm,yshift=1cm,fill=Mahogany!10!white] (4,1.5) -- (2.5,0) -- (0,0) -- (1.5,1.5) -- (4,1.5);
	\node at (2,.7) {$g$};
	\draw [densely dotted,thick] (1.75,.4) -- (1.35,0);
	\draw [xscale=.4,yscale=.4, xshift=7.5cm,yshift=1cm,fill=Mahogany!10!white] (4,1.5) -- (2.5,0) -- (0,0) -- (1.5,1.5) -- (4,1.5);
	\node at (3.75,.7) {$h$};
	\draw [densely dotted,thick] (3.5,.4) -- (3.1,0);
	\end{scope}
	
	\draw [->] (3,-.5) -- (3,-2);
	\node at (3,-1.25) [left] {$\cong$};
	\node at (3,-1.25) [right] {half Dehn twist};
	
	\begin{scope}[yshift=-4cm]
	\node at (-3,.75) [left] {$h \oplus g$};
	\draw [->] (-2.75,.75) -- (0,.75);
	\node at (-1.5,.75) [above,align=center] {shrink $\Sigma_{h,1}$ and\\ $\Sigma_{g,1}+h \cdot e_1$};
	\node at (-1.5,.75) [below] {$\cong$};
	\draw (6,1.5) -- (4.5,0) -- (0,0) -- (1.5,1.5) -- (6,1.5);
	\draw [xscale=.4,yscale=.4, xshift=3cm,yshift=1cm,fill=Mahogany!10!white] (4,1.5) -- (2.5,0) -- (0,0) -- (1.5,1.5) -- (4,1.5);
	\node at (2,.7) {$h$};
	\draw [densely dotted,thick] (3.5,.4) to[in=35,out=-125,looseness=0.6] (1.35,0);
	\draw [xscale=.4,yscale=.4, xshift=7.5cm,yshift=1cm,fill=Mahogany!10!white] (4,1.5) -- (2.5,0) -- (0,0) -- (1.5,1.5) -- (4,1.5);
	\node at (3.75,.7) {$g$};
	\draw [densely dotted,thick] (1.75,.4) to[out=-125,in=-125,looseness=.8] (3,.65) to[out=55,in=180] (3.9,1.25) to[out=0,in=55,looseness=1.5] (4.8,.5) to[out=-125,in=55,looseness=.8] (3.1,0);
	\end{scope}
	
	\draw [->] (-3.5,.25) -- (-3.5,-2.75);
	\node at (-3.5,-1.25) [left] {$\beta_{g,h}$};
	\end{tikzpicture}
	\caption{The braiding of $\cat{MCG}$.}
	\label{fig:braiding}
\end{figure}

The braided monoidal structure on $\cat{MCG}$ endows the disjoint union $\bigsqcup_{g \geq 1} B\Gamma_{g,1}$ of the classifying spaces of the automorphism groups in $\cat{MCG}$ with the structure of a non-unital $E_2$-algebra, up to weak equivalence. To see this, we consider the braided monoidal category $\cat{MCG}$ in the framework set up in Part 4
%\ref{I.part:framework}
of \cite{e2cellsI}, and shall use some notation from there too.

That framework yields a braided monoidal and simplicially enriched category $\cat{sSet}^\cat{MCG}$ whose objects are functors $\cat{MCG} \to \cat{sSet}$.  An object $X$ in this category has bigraded homology groups given by $H_{g,d}(X;\mathbbm{k}) \coloneqq H_d(X(g);\mathbbm{k})$ for any commutative ring $\mathbbm{k}$. 

The category $\cat{sSet}^\cat{MCG}$ contains a non-unital $E_2$-algebra given by $\varnothing$ if $g=0$ and a point $\ast$ otherwise. As in Section $E_k$.17.1, we can construct a weakly equivalent cofibrant non-unital $E_2$-algebra $\gT \in \Alg_{E_2}(\cat{sSet}^\cat{MCG})$ by cofibrant replacement, which satisfies
\[\gT(g) \simeq \begin{cases}
\varnothing & \text{ if $g=0$,}\\
* & \text{ if $g > 0$}.
\end{cases}\]

Writing $r \colon \cat{MCG} \to \bN$  for the unique functor which is the identity map on objects, the left Kan extension of $\gT$ along $r$ inherits an $E_2$-algebra structure by Lemma~ $E_k$.2.12, and gives rise to an object $\gR \coloneqq r_*(\gT) \in \Alg_{E_2}(\cat{sSet}^\bN)$, having 
\[\gR(g) \simeq \begin{cases} \varnothing & \text{if $g=0$,} \\
B\Gamma_{g,1} & \text{if $g > 0$}.\end{cases}\]
This is the promised $E_2$-algebra structure on an $\bN$-graded simplicial set weakly equivalent to $\bigsqcup_{g \geq 0} B\Gamma_{g,1}$. Such an $E_2$-algebra was first vaguely described by Miller \cite{Miller}; a more precise description, making use of the same braided monoidal groupoid $\cat{MCG}$, was given by Fiedorowicz--Song \cite{FiedorowiczSong}.

For any commutative ring $\mathbbm{k}$ we have $H_{g,d}(\gR;\mathbbm{k}) \cong H_d(\Gamma_{g,1};\mathbbm{k})$ for $g>0$, so the homology of all mapping class groups is encoded in the object $\gR$. As we are interested in the homology of $\gR$ with coefficients in $\mathbbm{k}$, we apply the symmetric monoidal functor $\mathbbm{k}[-] \colon \cat{sSet} \to \cat{sMod}_\mathbbm{k}$ sending a (simplicial) set to the free (simplicial) $\mathbbm{k}$-module on it, to obtain an $E_2$-algebra
\[\gR_\mathbbm{k} \coloneqq \mathbbm{k}[\gR] \in \Alg_{E_2}(\cat{sMod}_\mathbbm{k}^\bN)\]
in the category of $\bN$-graded simplicial $\bk$-modules.

Using a construction reminiscent of Moore loops, in Section $E_k$.12.2 we construct a unital strictly associative algebra $\overline{\gR}_\bk$ which as an $E_1$-algebra is equivalent to the free unitalisation $\gR^+_\bk$ by Proposition $E_k$.12.9. We will work with these objects, making the identifications $H_{g,d}(\gR_\bk) = H_{g,d}(\gR^+_\bk)$ for $g>0$ and $ H_{g,d}(\gR^+_\bk) = H_{g,d}(\overline{\gR}_\bk)$ for all $g$ without further comment. We shall write 
\begin{equation}\label{eqn:sigma}
\sigma \in H_{1,0}(\gR_\bk) = H_0(\Gamma_{1,1};\bk)
\end{equation} for the canonical generator. 

There are objects $S^{p,q}_\bk \in \cat{sMod}_\bk^\bN$ defined by
\begin{equation}\label{eqn:bigraded-spheres} S^{p,q}_\bk(g) \coloneqq \begin{cases}
    \bk[\Delta^q]/\bk[\partial \Delta^q]
    % \bk[\Delta^q/\partial \Delta^q]
    &\text{if $g=p$,}\\
0 & \text{otherwise},
\end{cases}\end{equation}
and we may choose a map $\sigma \colon S^{1,0}_\mathbbm{k} \to \overline{\mathbf{R}}_\mathbbm{k}$ representing the class $\sigma \in H_{1,0}(\overline{\mathbf{R}}_\mathbbm{k})$. Using the unital and strictly associative multiplication $- \cdot - \colon \overline{\gR}_\bk \otimes \overline{\gR}_\bk \to \overline{\gR}_\bk$ we can form the map $\sigma \cdot - \colon S^{1,0}_\bk \otimes \overline{\gR}_\bk \to \overline{\gR}_\bk$ in $\cat{sMod}_\bk^\bN$, and we write $\overline{\gR}_\bk/\sigma$ for its homotopy cofibre. The operation $S^{1,0} \otimes -$ corresponds to shifting the $\bN$-grading. That is, $(S^{1,0}_\bk \otimes X)(g) \cong X(g-1)$. The map $(\sigma \cdot -)(g)$ is therefore identified up to homotopy with the map on simplicial $\bk$-chains induced by the inclusion
\[\Gamma_{g-1,1} \xrightarrow{\{\mr{id}_1\} \times \mathrm{Id}} \Gamma_{1,1} \times \Gamma_{g-1,1} \overset{\oplus}\lra \Gamma_{g,1},\]
so there is an isomorphism
\[H_{g,d}(\overline{\gR}_\bk/\sigma) \cong H_d(\Gamma_{g,1}, \Gamma_{g-1,1};\bk).\]

In Section $E_k$.12.2.3 we have explained how the map $\sigma \cdot -$ may be represented up to homotopy by a map $\phi(\sigma)$ of left $\overline{\gR}_\bk$-modules, which endows $\overline{\gR}_\bk/\sigma$ up to homotopy with the structure of a left $\overline{\gR}_\bk$-module and yields a cofibre sequence
\begin{equation}\label{eq:RModSigmaCofSeq}
S^{1,0} \otimes \overline{\gR}_\bk \overset{\phi(\sigma)}\lra \overline{\gR}_\bk \overset{q}\lra \overline{\gR}_\bk/\sigma \overset{\partial}\lra S^{1,1} \otimes \overline{\gR}_\bk
\end{equation}
in the homotopy category of left $\overline{\gR}_\bk$-modules.

Our first step is to establish the \emph{standard connectivity estimate} of Definition $E_k$.17.6 for $\gR$, and hence for $\gR_\mathbbm{k}$ since $\mathbbm{k}[-]$ preserves connectivity. This estimate says that a certain semi-simplicial pointed set $\suspsplit_\bullet(g)$ associated to each $g \in \cat{MCG}$ is $(g-1)$-connected, but it is equivalent to show that a certain semi-simplicial set $\split_\bullet(g)$, described in Definition $E_k$.17.9, is $(g-3)$-connected. We now give the definition of this semi-simplicial set, in terms adapted to the situation at hand. 

By Remark $E_k$.17.11, it may be defined using certain ``Young-type'' subgroups of $\Gamma_{g,1}$ associated to an ordered $(p+2)$-tuple $(g_0,\ldots,g_{p+1})$ of positive integers summing to $g$. Given such a $(p+2)$-tuple, we can decompose $\Sigma_{g,1}$ into $(p+2)$ subsurfaces as
\[\Sigma_{g,1} = \bigsqcup_{i=0}^{p+1} (\Sigma_{g_i,1}+(g_0+\cdots+g_{i-1}) \cdot e_1).\]
Such a decomposition induces a homomorphism $\iota_{(g_0,\ldots,g_{p+1})} \colon \Gamma_{g_0,1} \times \cdots \times \Gamma_{g_{p+1},1} \to \Gamma_{g,1}$. This is injective by Lemma \ref{lem:MonoidalInj}.  In Section \ref{sec:HighConn} below we shall identify its image as the stabiliser of a $p$-tuple of arcs.

\begin{definition}
Given a $(p+2)$-tuple $(g_0,\ldots,g_{p+1})$ of positive integers summing to $g$, the Young-type subgroup $\Gamma_{(g_0,\ldots,g_{p+1}),1} \subset \Gamma_{g,1}$ is defined to be the image of $\iota_{(g_0,\ldots,g_{p+1})}$.
\end{definition}

The subgroup $\Gamma_{(g_0,\ldots,g_{p+1}),1}$ is isomorphic to $\Gamma_{g_0,1} \times \cdots \times \Gamma_{g_{p+1},1}$, and is functorial in $(g_0,\ldots,g_{p+1})$ in the following sense. We say that $(g'_0,\ldots,g'_{p'+1})$ is a \emph{refinement} of $(g_0,\ldots,g_{p+1})$ if there is a surjective order-preserving map $\phi \colon [p'+1] \to [p+1]$ such that $g_i = \sum_{j \in \phi^{-1}(i)} g'_j$. Whenever $(g'_0,\ldots,g'_{p'+1})$ is a refinement of $(g_0,\ldots,g_{p+1})$ the subgroup $\Gamma_{(g'_0,\ldots,g'_{p'+1}),1}$ is contained in $\Gamma_{(g_0,\ldots,g_{p'+1}),1}$. Upon using the given identifications
\[\Gamma_{(g'_0,\ldots,g'_{p'+1}),1} \cong \Gamma_{g'_0,1} \times \cdots \times \Gamma_{g'_{p'+1},1} \quad \text{and} \quad \Gamma_{(g_0,\ldots,g_{p'+1}),1} \cong \Gamma_{g_0,1} \times \cdots \times \Gamma_{g_{p+1},1},\]
this inclusion is induced by the monoidal structure of $\cat{MCG}$: when adjacent integers are added, we apply the homomorphism $- \oplus - \colon \Gamma_{g'_i,1} \times \Gamma_{g'_{i+1},1} \to \Gamma_{g'_i+g'_{i+1},1}$.

\begin{definition}\label{defn:SplittingCx1} The semi-simplicial set $\split_\bullet(g)$ has $p$-simplices
\[\split_p(g) = \bigsqcup_{\substack{g_0, \ldots, g_{p+1} > 0\\\sum g_i=g}} \frac{\Gamma_{g,1}}{\Gamma_{(g_0,\ldots,g_{p+1}),1}}\]
and face maps induced by inclusions of Young-type subgroups. It is called the \emph{$E_1$-splitting complex} and its geometric realisation $\fgr{\split_\bullet(g)}$ is denoted $\split(g)$. 
\end{definition}
Section \ref{sec:HighConn} will be dedicated to the proof of the following theorem.

\begin{theorem}\label{thm:MCGSplittingConn}
The space $\split(g)$ is $(g-3)$-connected.
\end{theorem}

Because $\split(g)$ is $(g-2)$-dimensional and $(g-3)$-connected, by the equivalence $\Sigma^2 \split(g) \simeq \suspsplit(g)$ of Lemma $E_k$.17.10 it follows that the pointed space $\suspsplit(g)$ has the homotopy type of a wedge of $g$-dimensional spheres; we write
\[\Steinb^{E_1}(g) \coloneqq \tilde{H}_{g}(\suspsplit(g);\bZ)\]
for the $\bZ[\Gamma_{g,1}]$-module given by its top homology, which is a free $\bZ$-module called the \emph{$E_1$-Steinberg module}. This moniker refers to the fact that the $E_1$-Steinberg module arising in the related context of general linear groups is the top homology of Charney's split building \cite{Charney}, which is closely related to the Tits building whose top homology is the classical Steinberg module. 

The equivalence $S^1 \wedge Q^{E_1}_\bL(\gR)(g) \simeq \suspsplit(g) \wedge_{\Gamma_{g,1}} (E\Gamma_{g,1})_+$ of Corollary $E_k$.17.5 implies that
\begin{equation}\label{eq:MCGE1Homology}
H_{g,d}^{E_1}(\gR_\mathbbm{k}) = H_{d-(g-1)}(\Gamma_{g,1} ; \Steinb^{E_1}(g) \otimes_\bZ \mathbbm{k})
\end{equation}
and in particular we deduce the following.

\begin{corollary}\label{cor:MCGVanishingLine}
$H_{g,d}^{E_1}(\gR_\mathbbm{k})=0$ for $d < g-1$.
\end{corollary}

\subsection{Low genus computations}

In this section we shall recall some calculations of the low-dimensional homology of $\Gamma_{g,1}$'s. Many of these calculations are not new---we deduce most of them from Korkmaz's survey \cite{Korkmaz} and Korkmaz--Stipsicz \cite{KorkmazStipsicz}, see Remark \ref{rem:mcg-comp} for additional references---but we take care to describe the answer in terms of operations on the homology of an $E_2$-algebra when possible.

As explained in Section $E_k$.16, for an $E_2$-algebra $\gX$ in $\cat{sSet}^\bN$ there is a \emph{product}
\[-\cdot - \colon H_{g,q}(\gX) \otimes H_{g', q'}(\gX) \lra H_{g+g', q+q'}(\gX)\]
and a \emph{Browder bracket}
\[[-,-] \colon H_{g,q}(\gX) \otimes H_{g',q'}(\gX) \lra H_{g+g', q+q'+1}(\gX)\]
in homology with arbitrary coefficients, as well as \emph{Araki--Kudo--Dyer--Lashof operations}
\[Q^s_2 \colon H_{g,q}(\gX;\bF_2) \lra H_{2g, q+s}(\gX;\bF_2)\]
for $s \leq q$, and a \emph{top operation}
\[\xi \colon H_{g,q}(\gX;\bF_2) \lra H_{2g, 2q+1}(\gX;\bF_2).\]
%In general the homology of an $E_k$-algebra with $\bF_p$-coefficients comes equipped with Dyer--Lashof operations $Q^s_p$, but for $k=2$ and $p=2$ all of these either vanish or can be written in terms of the product. However, i
The basic properties of Dyer--Lashof operations mean that $Q^s_2(x)=0$ for $s < |x|$, and $\smash{Q_2^{|x|}}(x)= x \cdot x$, and so (in this case of $E_2$-algebras) they carry no more information than the product. But it is convenient to have them available to simplify the naming of certain homology classes, and it is furthermore convenient to write $Q^{q+1}_2(x) \coloneqq \xi(x)$, and treat this as an honorary Dyer--Lashof operation. As in the proof of Theorem $E_k$.18.1 we shall also make use of an integral refinement of one of these operations, 
\[Q^1_\bZ \colon H_{g, 0}(\gX;\bZ) \lra H_{2g, 1}(\gX;\bZ),\]
which interacts with $\rho_2$, reduction modulo 2, as $\rho_2(Q^1_\bZ(x)) = Q^1_2(\rho_2(x))$ and with the Browder bracket as $2 Q^1_\bZ(x) = -[x,x]$ (our sign is chosen to counteract Cohen's sign convention in \cite{CLM}), and is characterised by these properties. Indeed, the free $E_2$-algebra on a generator $x$ of bidegree $(g,0)$ is $\gE_2(S^{g,0})$, and $E_2(S^{g,0})(2g) \simeq \mathrm{Conf}_2(\R^2)$ has first homology $\bZ$; the class $Q^1_\bZ(x)$ is the generator represented by a half-rotation in the anticlockwise direction.

Recall that to any submanifold of an oriented surface diffeomorphic to a circle there is an associated \emph{right-handed Dehn twist} diffeomorphism, whose isotopy class depends only on the submanifold and the orientation of the surface, see Figure \ref{fig:dehn-twist}.

\begin{figure}[h]
	\begin{tikzpicture}
	\begin{scope}[scale=0.95]
	\draw (0,0) circle (.7cm);
	\draw (0,0) circle (2cm);
	\draw [thick,Mahogany] (0,-2) -- (0,-.7);
	\node at (0,-2) [Mahogany] {$\bullet$};
	\node at (0,-.7) [Mahogany] {$\bullet$};
	
	\draw [->] (2.5,0) -- (6.5,0);
	\node at (4.5,-.5) {right-handed Dehn twist};
	
	\begin{scope}[xshift=9cm]
	\draw (0,0) circle (.7cm);
	\draw (0,0) circle (2cm);
	\draw [Mahogany, thick,  domain=-2:-.7, samples=80, smooth] 
	plot (xy polar cs:angle={360/1.3 * (\x+.7)+90}, radius={\x});
	\node at (0,-2) [Mahogany] {$\bullet$};
	\node at (0,-.7) [Mahogany] {$\bullet$};
	\end{scope}
	\end{scope}
	\end{tikzpicture}
	\caption{A right-handed Dehn twist.}
	\label{fig:dehn-twist}
\end{figure}
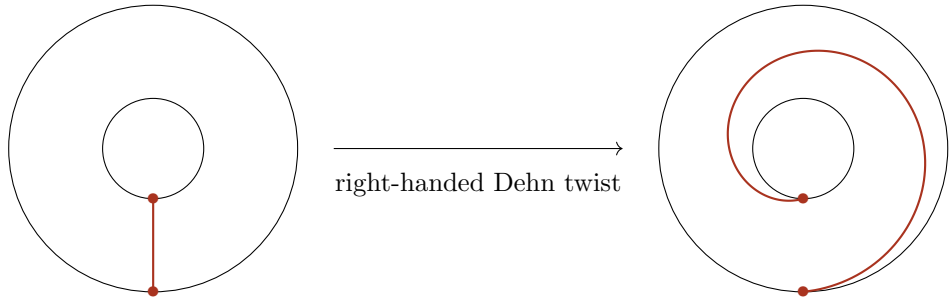

Finally, we wish to recall something about certain low-dimensional cohomology classes of $\Gamma_{g,1}$'s. A diffeomorphism of $\Sigma_{g,1}$ induces an isomorphism of $H_1(\Sigma_{g,1};\bZ)$, which respects the (skew-symmetric and non-singular) intersection form: this determines a group homomorphism $\Gamma_{g,1} \to \mathrm{Sp}_{2g}(\bZ)$, which naturally factors through the mapping class group $\Gamma_g$ of a closed genus $g$ surface. We may form the composition
\[B\Gamma_{g,1} \lra B\Gamma_{g} \lra B\mathrm{Sp}_{2g}(\bZ) \lra B\mathrm{Sp}_{2g}(\bR) \overset{\sim}\longleftarrow BU(g),\]
where $\mathrm{Sp}_{2g}(\bR)$ denotes the topological, not discrete, group, which is homotopy equivalent to its maximal compact subgroup $U(g)$. The pullback of the first Chern class under this composition is denoted $\lambda_1 \in H^2(\Gamma_{g,1};\bZ)$ (or in the cohomology of $\Gamma_g$), and is often called the Hodge class. It is natural with respect to the inclusions $\Gamma_{g-1,1} \to \Gamma_{g,1}$. It follows from \cite[Theorems 3.9 and 3.10]{KorkmazStipsicz} that $H^2(\Gamma_{g,1};\bZ) \cong \bZ$ for all $g \geq 3$; combining this with the work of Meyer \cite{Meyer} shows that $\lambda_1$ generates $H^2(\Gamma_{g,1};\bZ)$ in these cases. Perhaps more familiar is the first Miller--Morita--Mumford class $\kappa_1 \in H^2(\Gamma_{g,1};\bZ)$ (or in the cohomology of $\Gamma_g$), which we will describe in more detail in Section \ref{sec:proofthmh43}: this satisfies $\kappa_1 = 12 \lambda_1$ (because both $4\lambda_1$ and $\tfrac{1}{3}\kappa_1$ compute the signature of total spaces of surface bundles over surfaces, by Meyer's theorem for the first and Hirzebruch's signature theorem for the second).

\begin{lemma}\label{lem:MCGFacts}
	\mbox{}
	\begin{enumerate}[(i)]
		\item\label{it:MCGFacts11} $H_{1,1}(\gR_\bZ) = \bZ$ generated by a right-handed Dehn twist $\tau$ along any non-separating circle in $\Sigma_{1,1}$.
		
		\item\label{it:MCGFacts21} $H_{2,1}(\gR_\bZ) = \bZ/10$ generated by $\sigma \tau$, where as usual $\sigma \in H_0(\Gamma_{1,1};\Z)$ is the class of a point.
		
		\item\label{it:MCGFactsg1} $H_{g,1}(\gR_\bZ) = 0$ for $g \geq 3$.
		
		\item\label{it:MCGFactsQ1} $Q^1_{\bZ}(\sigma) = 3 \sigma \tau \in H_{2,1}(\gR_\bZ) = \bZ/10\{\sigma \tau\}$.
		
		\item\label{it:MCGFacts42} $H_{4,2}(\gR_{\bZ}) = \bZ$ generated by $\sigma  \lambda$, where $\lambda \in H_{3,2}(\gR_\bZ)$ satisfies $\langle \lambda,\lambda_1 \rangle=1$ and is well-defined up to the image of $\sigma \cdot - \colon H_{2,2}(\gR_{\bZ}) \to H_{3,2}(\gR_{\bZ})$.

		\item\label{it:MCGFacts32rel} $H_{3,2}(\overline{\gR}_\bZ/\sigma) = \bZ$, generated by a class $\mu$ such that $\partial(\mu) = \sigma \tau \in H_{2,1}(\overline{\gR}_\bZ)$ and $q(\lambda)=10\mu$, where $\partial$ and $q$ are homomorphisms in the long exact sequence associated to the cofibre sequence \eqref{eq:RModSigmaCofSeq}.

	\end{enumerate}
	Furthermore, we have
	\begin{enumerate}[(i)]
		\setcounter{enumi}{6}
		\item\label{it:MCGFactsg2} $H_{g,2}(\gR_{\bZ[\frac{1}{2}]})  = 0$ for $g<3$.
		
		\item\label{it:MCGFacts32abs} $H_{3,2}(\gR_{\bZ[\frac{1}{10}]}) = \bZ[\tfrac{1}{10}]$, generated by a class $\lambda$ which maps to $10\mu$ under the isomorphism $H_{3,2}(\gR_{\bZ[\frac{1}{10}]}) \to H_{3,2}(\overline{\gR}_{\bZ[\frac{1}{10}]}/\sigma)= \bZ[\tfrac{1}{10}]\{\mu\}$.	
	\end{enumerate}
\end{lemma}
\begin{proof}
	For the first four parts we rely on Wajnryb's presentation \cite{WajnrybPres} of the mapping class groups $\Gamma_{g,1}$. Our reference for this is Korkmaz's survey \cite[Section 4]{Korkmaz}. 
	
	The generators in this presentation are all Dehn twists around non-separating curves; by the classification of surfaces such curves are permuted transitively by the mapping class group, so $\Gamma_{g,1}$ is normally generated by a single element, so has cyclic abelianisation generated by $\sigma^{g-1} \tau$. It is proved on page 108 of \cite{Korkmaz} that $\Gamma_{1,1}$ is isomorphic to the braid group on three strands, and so has abelianisation $\bZ$ which must be generated by $\tau$. In $\Gamma_{2,1}$ one finds the first instance of the ``two-holed torus relation,'' which implies that $10\sigma \tau=0 \in H_1(\Gamma_{2,1};\bZ)$, and it is shown on page 107 of \cite{Korkmaz} that this group is $\bZ/10$ generated by $\sigma \tau$. In $\Gamma_{3,1}$ one finds the first instance of the ``lantern relation,'' which implies that $\sigma^2 \tau=0$. This completes the proof of (\ref{it:MCGFacts11}), (\ref{it:MCGFacts21}) and (\ref{it:MCGFactsg1}).
	
	For item (\ref{it:MCGFactsQ1}), we first show that $2 Q_\bZ^1(\sigma) = -[\sigma, \sigma]=6\sigma \tau \in H_1(\Gamma_{2,1};\bZ)$, so that $Q_\bZ^1(\sigma)$ is either $3\sigma \tau$ or $8\sigma \tau$, and then exclude the latter case by showing that $\rho_2(Q^1_\bZ(\sigma)) = Q^1_2(\sigma)  \neq 0 \in H_1(\Gamma_{2,1};\bF_2)$.
	
	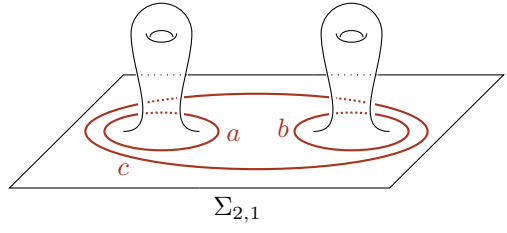
\begin{figure}[h]
		\begin{tikzpicture}
		\node at (3,0) [below] {$\Sigma_{2,1}$};
		
		\draw [thick,Mahogany] (3.25,0.75) ellipse (2.25cm and 0.5cm);
		\draw [thick,Mahogany] (2,0.75) ellipse (0.75cm and 0.25cm);
		\draw [thick,Mahogany] (4.5,0.75) ellipse (0.75cm and 0.25cm);
		\fill [white] (1.7,0.75) rectangle (2.3,2);
		\fill [white] (4.2,0.75) rectangle (4.8,2);
		
		\begin{scope}
		\clip (1.8,0.75) rectangle (2.2,2);
		\draw [thick,Mahogany,densely dotted] (3.25,0.75) ellipse (2.25cm and 0.5cm);
		\draw [thick,Mahogany,densely dotted] (2,0.75) ellipse (0.75cm and 0.25cm);
		\end{scope}
		
		\begin{scope}
		\clip (4.3,0.75) rectangle (4.7,2);
		\draw [thick,Mahogany,densely dotted] (3.25,0.75) ellipse (2.25cm and 0.5cm);
		\draw [thick,Mahogany,densely dotted] (4.5,0.75) ellipse (0.75cm and 0.25cm);
		\end{scope}

		\begin{scope}[xshift=0cm]
		\draw (2.35,1.5) -- (4,1.5);
		\draw (2.5,0) -- (0,0) -- (1.5,1.5) -- (1.65,1.5);
		\draw [dotted] (1.65,1.5) -- (2.35,1.5);
		\begin{scope}[xshift=2cm,yshift=1.25cm]
		\draw (-.5,-.5) to[out=0,in=-90] (-.4,0.8) to[out=90,in=180] (0,1.2) to[out=0,in=90] (.4,0.8) to[out=-90,in=180] (.5,-.5);	
		\draw (-.2,0.8) to[out=-90,in=-90] (.2,0.8);
		\draw (-.15,0.75) to[out=90,in=90] (.15,0.75);
		\end{scope}
		\end{scope}
		
		\begin{scope}[xshift=2.5cm]
		\draw (2.35,1.5) -- (4,1.5) -- (2.5,0) -- (0,0);
		\draw (1.5,1.5) -- (1.65,1.5);
		\draw [dotted] (1.65,1.5) -- (2.35,1.5);
		\begin{scope}[xshift=2cm,yshift=1.25cm]
		\draw (-.5,-.5) to[out=0,in=-90] (-.4,0.8) to[out=90,in=180] (0,1.2) to[out=0,in=90] (.4,0.8) to[out=-90,in=180] (.5,-.5);	
		\draw (-.2,0.8) to[out=-90,in=-90] (.2,0.8);
		\draw (-.15,0.75) to[out=90,in=90] (.15,0.75);
		\end{scope}
		\end{scope}
		
		\node [Mahogany] at (2.95,0.7) {$a$};
		\node [Mahogany] at (3.6,0.8) {$b$};
		\node [Mahogany] at (1.5,0.25) {$c$};
		
		\end{tikzpicture}
		\caption{Curves such that $-[\sigma,\sigma] = t_c t_a^{-1} t_b^{-1}$.}
		\label{fig:browder}
	\end{figure}

	For the first of these, recall from Section $E_k$.16.1.1 that $-[\sigma, \sigma] \in H_{2,1}(\gR_\bZ)$ is represented by the diffeomorphism of $\Sigma_{2,1}$ given by the square of the braiding $\beta_{1,1}$: this is the diffeomorphism $t_c t_a^{-1} t_b^{-1}$ for the curves $a$, $b$, and $c$ in Figure \ref{fig:browder}. The two-holed torus relation applied to the complement of the genus $1$ surface with one boundary component bounded by $a$ implies that
	\[[t_c] + [t_a] = 12 \sigma \tau = 2 \sigma \tau \in H_1(\Gamma_{2,1};\bZ) = \bZ/10\{\sigma \tau\},\]
	and there is a diffeomorphism interchanging the curves $a$ and $b$ so $t_a$ and $t_b$ are conjugate. Thus $[t_c t_a^{-1} t_b^{-1}] = 2\sigma \tau -3[t_a]$. Finally, the central extension
	\begin{equation}\label{eq:Genus1Ext}
	0 \lra \bZ \lra \Gamma_{1,1} \lra \Gamma_1^1 \cong SL_2(\bZ) \lra 0
	\end{equation}
	gives an exact sequence
	\[0 \lra \bZ \lra H_1(\Gamma_{1,1};\bZ) = \bZ\{\tau\} \lra \bZ/12 \lra 0\]
	which shows that the Dehn twist around the boundary is $\pm 12 \tau \in H_1(\Gamma_{1,1};\bZ) = \bZ\{\tau\}$; in fact it is $12 \tau$ \cite[page 108]{Korkmaz}. (Alternatively, the expression of a Dehn twist along the boundary as a sixth power of a product of two positive Dehn twists \cite[proof of Lemma 1.3]{minsky}.) Hence $[t_a] = 12 \sigma \tau$ and so $-[\sigma, \sigma] = [t_c t_a^{-1} t_b^{-1}] = 2 \sigma \tau - 3(12\sigma \tau) = 6 \sigma \tau \in \bZ/10\{\sigma \tau\}$ as required.
	
	Let us now show that $Q^1_2(\sigma) \neq 0 \in H_1(\Gamma_{2,1};\bF_2)$. Under the homomorphism $\Gamma_{2,1} \to \mathrm{Sp}_4(\bZ) \to \mathrm{Sp}_4(\bZ/2)$, the homology class $Q^1_2(\sigma)$ is represented by the matrix
	\begin{equation}\label{eq:Transposition}
\left(\begin{array}{cc}
	0 & I_2 \\
	I_2 & 0 \\
	\end{array}\right) \in \mathrm{Sp}_4(\bZ/2).
\end{equation}
	Recall that $\mathrm{Sp}_4(\bZ/2) \cong \fS_6$. Following \cite[\S 3.1.5]{OMearaSymp} this isomorphism may be given by the action of $\mathrm{Sp}_4(\bZ/2)$ on the set $\Omega$ of 5 element subsets $\{v_1, \ldots, v_5\} \subset \bF_2^4$ which are totally non-orthogonal, i.e.\ $\langle v_i, v_j\rangle=1$ for $i \neq j$; there are 6 such subsets, and the induced homomorphism $\phi \colon \mathrm{Sp}_4(\bZ/2) \to \mathrm{Sym}(\Omega)$ is an isomorphism. In terms of the standard symplectic basis $e_1, f_1, e_2, f_2$ for $(\bZ/2)^4$ these 6 subsets may be enumerated as
	\begin{align*}
	1 &= \{e_1, f_1, e_1+f_1+e_2, e_1+f_1+f_2, e_1+f_1+e_2+f_2\}\\
	2 &= \{e_2, f_2, e_2+f_2+e_1, e_2+f_2+f_1, e_2+f_2+e_1+f_1\}\\
	3 &= \{e_1, e_1+f_1, f_1+e_2, f_1+f_2, f_1+e_2+f_2\}\\
	4 &= \{e_2, e_2+f_2, f_2+e_1, f_2+f_1, f_2+e_1+f_1\}\\
	5 &= \{f_2, e_1+e_2, e_1+f_1+e_2, e_2+f_2, f_1+e_2\}\\
	6 &= \{f_1, e_1+e_2, e_2+f_2+e_1, e_1+f_1, f_2+e_1\}
	\end{align*}
	and one sees that $\phi\left(\begin{smallmatrix}
	0 & I_2 \\
	I_2 & 0 \\
	\end{smallmatrix}\right)$ is given by the permutation $(12)(34)(56) \in \fS_6$, so is odd. Thus it remains nontrivial in $H_1(\mathrm{Sp}_4(\bZ/2);\bF_2) = \bF_2$, so $Q^1_2(\sigma) \neq 0$ as required.
	
	Item (\ref{it:MCGFacts42}) follows from the discussion just before the statement of this lemma: by \cite[Theorem 3.9]{KorkmazStipsicz} we have $H_2(\Gamma_{4,1};\bZ) \cong \bZ$ and by the discussion the cohomology class $\lambda_1$ takes the value 1 on this group; by the proof of \cite[Theorem 3.10]{KorkmazStipsicz} (or the stability range of \cite[Theorem 1]{Boldsen}) the map $H_{2}(\Gamma_{3,1};\bZ) \to H_{2}(\Gamma_{4,1};\bZ)$ is surjective. Combining these, there exists a $\lambda \in H_2(\Gamma_{3,1};\bZ)$ such that $\langle \lambda, \lambda_1 \rangle = 1$. We shall address the well-definedness of $\lambda$ below.
	
	For item (\ref{it:MCGFacts32rel}), as explained in the proof of \cite[Theorem 3.10]{KorkmazStipsicz}, calculating with Hopf's formula for second homology and Wajnryb's presentation of $\Gamma_{3,1}$ is not completely conclusive, and shows that $H_2(\Gamma_{3,1};\bZ)$ is isomorphic to either $\bZ$ or $\bZ \oplus \bZ/2$ (the latter is in fact the case by Remark \ref{rem:mcg-comp}). The (potential) homology class of order 2 is represented by a product of commutators $[a_i, a_j]$ arising in part (A) of \cite[Theorem 2.1]{KorkmazStipsicz}. Each such commutator is supported in a genus 2 subsurface, and hence the (potential) homology class of order 2 lies in the image of the stabilisation map $\sigma \cdot - \colon H_2(\Gamma_{2,1};\bZ) \to H_2(\Gamma_{3,1};\bZ)$. Thus the cokernel of this stabilisation map is isomorphic to $\bZ$, generated by any of the classes $\lambda$ described above (this also shows that $\lambda$ is well-defined modulo the image of $\sigma \cdot -$, finishing (\ref{it:MCGFacts42})), and so there is a short exact sequence
\begin{equation}\label{eq:Ext}
0 \lra \bZ\{\lambda\} \overset{q}\lra H_2(\Gamma_{3,1}, \Gamma_{2,1};\bZ) \overset{\partial}\lra \bZ/10\{\sigma \tau\} \lra 0.
\end{equation}
We wish to determine this extension. Recall that for a finite cyclic group $G$ and a generator $g \in G$ there is an extension
\begin{equation*}
\begin{tikzcd}
0 \rar & \bZ \rar{1 \mapsto |G|}& \bZ \rar{1 \mapsto g} & G  \rar & 0
\end{tikzcd}
\end{equation*}
defining a class $\epsilon_g \in \mr{Ext}_\bZ^1(G, \bZ)$, such that $\bZ/|G|\{\epsilon_g\} \overset{\sim}\to \mr{Ext}_\bZ^1(G, \bZ)$. If $\mu_n \subset \bC^\times$ is the group of $n$th roots of unity, generated by $\xi_n = \exp(\tfrac{2\pi i}{n})$, then the class $\epsilon_{\xi_n} \in \mr{Ext}_\bZ^1(\mu_n, \bZ) \cong H^2(\mu_n;\bZ)$ is the same as the first Chern class of the representation $L_n$ of $\mu_n$ on $\bC$ by multiplication.

Describing the extension \eqref{eq:Ext} is the same as describing the image of $\lambda_1$ under the map
\[ H^2(\Gamma_{3,1};\bZ) = \bZ\{\lambda_1\} \lra H^2(\Gamma_{2,1};\bZ) \cong \mr{Ext}_\bZ^1(\bZ/10\{\sigma\tau\}, \bZ) = \bZ/10\{\epsilon_{\sigma\tau}\}.\]
We claim that $\lambda_1 = \epsilon_{\sigma\tau}$ in $H^2(\Gamma_{2,1};\bZ)$: this means that there is a $\mu \in H_2(\Gamma_{3,1}, \Gamma_{2,1};\bZ)$ such that $\partial(\mu) = \sigma\tau$ and $10\mu = q(\lambda)$, which is precisely what (\ref{it:MCGFacts32rel}) claims. The first homology of the mapping class group $\Gamma_2$ of a closed genus 2 surface is also $\bZ/10$, as described on page 107 of \cite{Korkmaz}, and its second homology is torsion by \cite[Theorem 3.10]{KorkmazStipsicz}, so the map $\delta \colon \Gamma_{2,1} \to \Gamma_2$ is an isomorphism on second cohomology: thus we are required to describe the class $\lambda_1 \in H^2(\Gamma_{2};\bZ) \cong \mr{Ext}_\bZ^1(\bZ/10\{\delta_*\sigma\tau\}, \bZ) = \bZ/10\{\epsilon_{\delta_*\sigma\tau}\}$. We shall use that $H_1(\Gamma_2;\bZ) \cong \bZ/10\{\delta_*\sigma\tau\}$ can be generated by elements of $\Gamma_2$ of order 2 and 5.

Firstly consider the genus 2 Riemann surface $\Sigma$ arising as the double cover of an elliptic curve branched over two points. Interchanging the sheets of the cover gives an action of $\mu_2$ of $\Sigma$, for which the action of $\xi_2$ on $H_1(\Sigma;\bZ)$ is given in a symplectic basis (of $a$- and $b$-curves) by \eqref{eq:Transposition}, considered as a matrix with integer entries. By the proof of (\ref{it:MCGFactsQ1}) the associated map $i \colon \mu_2 \to \Gamma_2$ is non-trivial on first homology, so must be given by $i_*[\xi_2] = 5\delta_*\sigma\tau$. Thus $i^*\epsilon_{\delta_*\sigma\tau} = \epsilon_{\xi_2}$. When we view the matrix \eqref{eq:Transposition} as having real coefficients it actually lies in the subgroup $U(2) \subset \mr{Sp}_4(\bR)$. Thus the composition
\[B\mu_2 \overset{Bi}\lra B\Gamma_{2} \lra B\mathrm{Sp}_{4}(\bZ) \lra B\mathrm{Sp}_{4}(\bR) \overset{\sim}\longleftarrow BU(2)\]
is seen to classify the 2-dimensional complex representation of $\mu_2$ given by the sum of the trivial representation and $L_2$, whose first Chern class is therefore $\epsilon_{\xi_2} \in H^2(\mu_2;\bZ)$. Thus $i^*\lambda_1 = \epsilon_{\xi_2} = i^*\epsilon_{\delta_*\sigma\tau}$.

Secondly, consider the genus 2 Riemann surface $\Sigma'$ arising as the double cover of $\mathbb{CP}^1$ branched over 0 and the $5$th roots of unity. The involution given by interchanging the sheets is the hyperelliptic involution. Consider the action of $\mu_5 \subset \bC^\times$ on $\mathbb{CP}^1 = \bC \cup \{\infty\}$. As $\xi_5$ fixes the branch points it lifts, in two different ways which differ by the hyperelliptic involution, to a diffeomorphism of $\Sigma'$. One of these lifts will have order 5 and one will have order 10 so there is a preferred choice of lift of $\xi_5$ which defines an action of $\mu_5$ on $\Sigma'$. (The order 10 diffeomorphism can not be used to study the 2- and 5-local situations simultaneously because the hyperelliptic involution is given by an even number of Dehn twists.) We need to understand the effect of the associated map $j \colon \mu_5 \to \Gamma_2$ on first homology. The relevant elements of $\Gamma_2$ arise from lifts of diffeomorphisms of the complex plane punctured at 0 and at the 5th roots of unity.  We shall identify the mapping class group of this punctured plane with the six-strand braid group in the usual way. Given an arc in $\bC\bP^1$ between two of the branch points and not passing through any other branch point, a positive braid of the endpoints along it lifts to a right-handed Dehn twist of $\Sigma'$ along a non-separating curve \cite[p.\ 54]{AGLV}. A moment's sketching shows that the braid corresponding to the diffeomorphism of $\bC\setminus \{0, 1, \xi_5, \xi_5^2, \xi_5^3, \xi_5^4\}$ given by multiplication by $\xi_5$ is given by 6 such positive braids, so $j_*[\xi_5] = 6 \delta_*\sigma \tau \in H_1(\Gamma_2;\bZ)$ and hence $j^*\epsilon_{\delta_*\sigma\tau} = 3\epsilon_{\xi_5}$. The action of $\mu_5$ on $\Sigma'$ has three fixed points: one lying over 0, and two lying over $\infty$. At the two fixed points over $\infty$ the action of $\mu_5$ on the tangent space of $\Sigma'$ is $\overline{L}_5$. At the fixed point over 0 the action on the tangent space is such that its tensor square is $L_5$: thus it is $L^3_5$. It therefore follows from the parametrised Riemann--Hurwitz formula of Kawazumi--Uemura \cite[Theorem B]{KU} for the Miller--Morita--Mumford classes that
\[12 j^*\lambda_1 = j^*\kappa_1 = (3-1-1)c_1(L_5) = \epsilon_{\xi_5} \in H^2(\mu_5 ; \bZ) \cong \mr{Ext}_\bZ^1(\mu_5, \bZ),\]
and so, dividing by 12, we have $j^*\lambda_1 = 3 \epsilon_{\xi_5} = j^*\epsilon_{\delta_*\sigma\tau}$. As $i^*$ and $j^*$ are jointly injective, it follows that $\lambda_1 = \epsilon_{\delta_*\sigma\tau}$, as claimed.
	
	For item (\ref{it:MCGFactsg2}), if $g=0$, there is nothing to show. If $g=1$, the claim follows from the Lyndon--Hochschild--Serre spectral sequence for the central extension \eqref{eq:Genus1Ext}, which actually shows that $H_2(\Gamma_{1,1};\bZ)=0$. If $g=2$, the claim follows from \cite[Theorem 1.3]{KorkmazStipsicz}.
	
	For item (\ref{it:MCGFacts32abs}) we consider the exact sequence
	\[\cdots H_{2,2}(\gR_{\bZ[\frac{1}{10}]}) \overset{\sigma \cdot -}\lra H_{3,2}(\gR_{\bZ[\frac{1}{10}]}) \lra H_{3,2}(\overline{\gR}_{\bZ[\frac{1}{10}]}/\sigma) \overset{\partial}\lra H_{2,1}(\gR_{\bZ[\frac{1}{10}]}) \cdots\]
	associated to the homotopy cofibre sequence \eqref{eq:RModSigmaCofSeq}, and the fact that the outer two terms are zero by items (\ref{it:MCGFactsg2}) and (\ref{it:MCGFacts21}).
\end{proof}

\begin{figure}[h]
	\begin{tikzpicture}
	\begin{scope}
	\draw (-1,0)--(5.5,0);
	\draw (0,-1) -- (0,2.5);
	
	\foreach \s in {0,...,2}
	{
		\draw [dotted] (-.5,\s)--(5.5,\s);
		\node [fill=white] at (-.25,\s) [left] {\tiny $\s$};
	}
	
	\foreach \s in {1,...,4}
	{
		\draw [dotted] ({1.25*\s},-0.5)--({1.25*\s},2.5);
		\node [fill=white] at ({1.25*\s},-.5) {\tiny $\s$};
		\node [fill=white] at ({1.25*\s},0) {$\bZ \sigma^\s$};
	}
	\node [fill=white] at (0,0) {$\bZ 1$};
	\node [fill=white] at (0,-.5) {\tiny 0};
	\node [fill=white] at ({1.25*1},1) {$\bZ \tau$};
	\node [fill=white] at ({1.25*2},1) {$\bZ/10 \sigma \tau$};
	\node [fill=white] at ({1.25*2},2) {$A$};
	\node [fill=white] at ({1.25*3},2) {$\bZ \lambda \oplus B$};
	\node [fill=white] at ({1.25*4},2) {$\bZ \sigma \lambda$};

	\node [fill=white] at (-.5,-.5) {$\nicefrac{d}{g}$};
	\end{scope}
	\end{tikzpicture}
	\caption{Summary of $H_{g,d}(\gR_\bZ)$ as described in Lemma \ref{lem:MCGFacts}, where $A$ and $B$ are abelian groups with $A[\tfrac{1}{2}] = 0 = B[\tfrac{1}{10}]$, and $\lambda$ is only well-defined modulo torsion. Empty entries are $0$. Compare to Remark \ref{rem:mcg-comp} and Figure \ref{fig:rat}.}
	\label{fig:table-mcg-rat}
\end{figure}

\begin{remark}\label{rem:mcg-comp} Though we will not need it, the table in Figure \ref{fig:table-mcg-rat} may be extended. Godin \cite[Proposition 18]{GodinMCG} and Abhau--B\"odigheimer--Ehrenfried \cite[Section 2.1]{ABE} have independently computed $H_*(\Gamma_{2,1};\bZ)$ and showed $A \cong \bZ/2\bZ$. It is a consequence of Godin's computations that $\tau^2$ is not $0$ but generates $A$ \cite[Example 4]{GodinMCG}. Wang \cite[Conjecture 4.2.1]{WangThesis} and Boes--Hermann \cite[Theorem]{BoesHermannThesis} have computed $H_*(\Gamma_{3,1};\bZ)$ up to torsion at primes $> 23$, concluding that $B \cong \bZ/2\bZ$, which has been confirmed independently by Sakasai \cite[Theorem 4.9]{Sakasai}. As a consequence of (\ref{it:MCGFacts32rel}) the map $H_2(\Gamma_{2,1};\bZ) \to H_2(\Gamma_{3,1};\bZ)$ is injective. Additional sources for computations are \cite{HarerH2,HarerH3,BensonCohen,Pitsch}. The non-vanishing of the Browder bracket $[\sigma,\sigma]$ is due to Fiedorowicz--Song \cite[Theorem 2.5]{FiedorowiczSong}, and up to a convention about signs (\ref{it:MCGFactsQ1}) may be deduced from their Lemma 2.1.\end{remark}

\section{The $E_1$-splitting complex is highly-connected}\label{sec:HighConn}

In addition to the machinery developed in \cite{e2cellsI} and a study of low-dimensional low-genus homology groups, we shall use that the $E_1$-splitting complex $\split_\bullet(g)$ is highly-connected. Of course, most proofs of stability involve proving the high-connectivity of certain simplicial complexes or semi-simplicial sets, but those which arise through the machinery of \cite{e2cellsI} are of a different nature to those which arise in the ``usual'' homological stability argument as for example in \cite{RWW}. Examples of such complexes have appeared before, such as in \cite{Charney,hepworth}, but without the connection to $E_1$-cells. The main result of this section is a proof of Theorem \ref{thm:MCGSplittingConn}; the reader willing to take this at face value on a first reading may proceed to Section~\ref{sec:stability}.

\subsection{Poset techniques}\label{sec:posets}

The semi-simplicial sets in the proof of Theorem \ref{thm:MCGSplittingConn} can be described as nerves of posets. To analyse the connectivity of posets and of maps between posets, we use a slight generalisation of \cite[Theorem 9.1]{quillenposet} which we learnt from Looijenga--van der Kallen \cite[Corollary 2.2]{LvdK}. Recall that if $\cY$ is a poset with partial order $\leq$ and $y \in \cY$ then $\cY_{< y}$ denotes the subset $\{y' \in \cY \setminus \{y\} \mid y' \leq y\}$ with the induced ordering, and similarly for $\cY_{> y}$.  If $f \colon \cX \to \cY$ is a map of posets then $f_{\leq y} \coloneqq \{x \in \cX \mid f(x) \leq y\}$ and similarly for $f_{\geq y}$.

\begin{theorem}\label{thm.posetmapconnectedness} 
	Let $f \colon \cX \to \cY$ be a map of posets, $n \in \bZ$, and $t_\cY \colon \cY \to \bZ$ be a function. Suppose that one of the following holds
	\begin{enumerate}[(i)]
		\item for every $y \in \cY$, $f_{\leq y}$ is $(t_\cY(y)-2)$-connected and $\cY_{>y}$ is $(n-t_\cY(y)-1)$-connected, or
		\item for every $y \in \cY$, $f_{\geq y}$ is $(n-t_\cY(y)-1)$-connected and $\cY_{<y}$ is $(t_\cY(y)-2)$-connected.
	\end{enumerate}
	Then the map $f$ is $n$-connected.
\end{theorem}

We shall also use the following consequence of Theorem \ref{thm.posetmapconnectedness}, which is a version of the Nerve Theorem of Borsuk adapted to posets and is a slight generalisation of \cite[Theorem 2.3]{LvdK}. Recall that if $\cX$ is a poset, then a subposet $\cY \subset \cX$ is said to be \emph{closed} if $x \leq y$ for $x \in \cX$ and $y \in \cY$ implies that $x \in \cY$. The closed subposets of $\cX$ form a poset ordered by inclusion.

\begin{corollary}[Nerve Theorem]\label{cor.nervelemma} 
	Let $\cX$ be a poset, $\cA$ another poset, and 
	\[F \colon \cA^\mr{op} \lra \{\text{closed subposets of $\cX$}\}\]
	be a functor. Let $n \in \bZ$, and $t_\cX \colon \cX \to \bZ$, $t_\cA \colon \cA \to \bZ$ be functions such that
	\begin{enumerate}[(i)]
		\item $\cA$ is $(n-1)$-connected, and
		
		\item for every $a \in \cA$, $\cA_{<a}$ is $(t_\cA(a)-2)$-connected and $F(a)$ is $(n-t_\cA(a)-1)$-connected, and
		
		\item for every $x \in \cX$, $\cX_{<x}$ is $(t_\cX(x)-2)$-connected and the subposet
		\[\cA_x \coloneqq \{a \in \cA \mid x \in F(a)\} \subset \cA\]
		is $((n-1)-t_\cX(x)-1)$-connected.
	\end{enumerate}
	Then $\cX$ is $(n-1)$-connected.
\end{corollary}

\begin{proof}Form a new poset $\cA \wr F$ whose objects are pairs $(a,x) \in \cA^\mr{op} \times \cX$ such that $x \in F(a)$, with partial order given by $(a,x) \preceq (a',x')$ if $a \leq a'$ and $x \leq x'$. There are two maps of posets
	\[\cA^\mr{op} \overset{\pi_1}\longleftarrow \cA \wr F \overset{\pi_2}\lra \cX\]
	given on objects by $\pi_1(a,x) = a$ and $\pi_2(a,x) = x$.
	
	Apply Theorem \ref{thm.posetmapconnectedness} (i) to the map $\pi_1 \colon \cA \wr F \to \cA^\mr{op}$ using the function $t_\cA$. For each $a \in \cA^\mr{op}$, the poset
	\[(\pi_1)_{\leq a} \coloneqq \{(b, x) \in \cA^\mr{op} \times \cX \mid x \in F(b), b \leq a\}\]
	contains $\{a\} \times F(a)$ as a subposet, and $(b,x) \mapsto (a,x) \colon (\pi_1)_{\leq a} \to \{a\} \times F(a)$ is right adjoint to the inclusion so these posets are homotopy equivalent. Thus $(\pi_1)_{\leq a}$ is $(n-t_\cA(a)-1)$-connected, and $(\cA^\mr{op})_{> a} \cong (\cA_{< a})^\mr{op}$ is $(t_\cA(a)-2)$-connected, so by Theorem \ref{thm.posetmapconnectedness} (i) the map $\pi_1$ is $n$-connected. As $\cA$ is $(n-1)$-connected, it follows that $\cA \wr F$ is $(n-1)$-connected.
	
	Now apply Theorem \ref{thm.posetmapconnectedness} (ii) to the map $\pi_2 \colon \cA \wr F \to \cX$ using the function $t_\cX$. For each $x \in \cX$, the poset
	\[(\pi_2)_{\geq x} \coloneqq \{(a, y) \in \cA^\mr{op} \times \cX \mid y \in F(a), x \leq y\}\]
	contains $(\cA_x)^\mr{op} \times \{x\}$ as a subposet, and $(a, y) \mapsto (a,x) \colon (\pi_2)_{\geq x} \to (\cA_x)^\mr{op} \times \{x\}$ is right adjoint to the inclusion so these posets are homotopy equivalent. Thus $(\pi_2)_{\geq x}$ is $((n-1)-t_\cX(x)-1)$-connected, and $\cX_{<x}$ is $(t_\cX(x)-2)$-connected, so by Theorem \ref{thm.posetmapconnectedness} (ii) the map $\pi_2$ is $(n-1)$-connected, and hence $\cX$ is $(n-1)$-connected.
\end{proof}

\subsection{Proof of Theorem \ref{thm:MCGSplittingConn}}

\label{sec:PfMCGCxSpherical} In this subsection we will study the connectivity of the $E_1$-splitting complex $\split(g)$ of Definition~\ref{defn:SplittingCx1}.  We remark that Looijenga \cite{looijenga2013connectivity} has proved a similar result, but for a complex of closed curves instead of arcs with boundary conditions.

The following semi-simplicial set is similar to the one denoted $\mathcal{O}(S,b_0,b_1)$ in \cite[Definition 2.1]{WahlSurvey}, except ``non-separating'' is replaced by the extreme opposite.

\begin{definition}\label{defn:SplittingCx2}
	Let $\Sigma$ be a connected oriented surface with one boundary component, and let $b_0, b_1 \in \partial \Sigma$ be distinct marked points. Let $S(\Sigma, b_0, b_1)_0$ denote the set of isotopy classes of smoothly embedded arcs in $\Sigma$ from $b_0$ to $b_1$ which separate $\Sigma$ into two path components, each having positive genus.
	
	Let $S(\Sigma, b_0, b_1)_p$ denote the set of $(p+1)$-tuples $(a_0, \ldots, a_p)$ of distinct elements of $S(\Sigma, b_0, b_1)_0$ where the collection $(a_0,\ldots,a_p)$ may be represented by a collection of embedded arcs such that
	\begin{enumerate}[(i)]
		\item they are disjoint except at their endpoints,
		\item the order $a_0, \ldots, a_p$ coincides with that obtained by ordering the arcs in clockwise fashion as they approach $b_0$ (and that obtained by ordering the arcs in counterclockwise fashion as they approach $b_1$),
		\item the region of $\Sigma$ between $a_i$ and $a_{i+1}$ has positive genus, for each $i$.
	\end{enumerate}
	These form a semi-simplicial set $[p] \mapsto S(\Sigma, b_0, b_1)_p$, where the $i$th face map is given by forgetting $a_i$.
\end{definition}

By a theorem of Gramain \cite[Th\'eor\`eme 5]{Gramain}, the space of disjoint arcs representing a $(p+1)$-tuple $(a_0, \ldots, a_p)$ of isotopy classes is contractible. In particular, the last two conditions above do not depend on the choice of such representatives. Furthermore, the parenthesised statement in (ii) follows from the unparenthesised one. We next relate it to the $E_1$-splitting complex $\split_\bullet(g)$. The model surface $\Sigma_{g,1} \subset [0,g] \times [0,1]^2$ has boundary $\partial ([0,g] \times [0,1]) \times \{0\}$ and so the points $(0,0,0)$ and $(0,1,0)$ lie on the boundary of $\Sigma_{g,1}$.

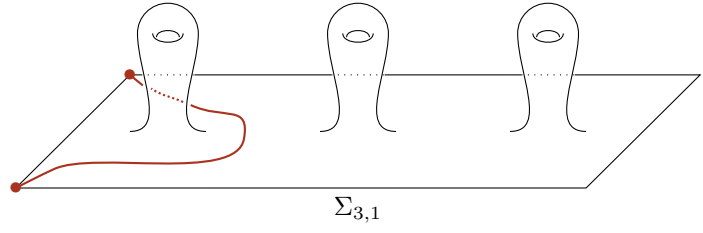
\begin{figure}[h]
	\begin{tikzpicture}
	\begin{scope}[xshift=0cm]
	\draw (2.35,1.5) -- (4,1.5);
	\draw (2.5,0) -- (0,0) -- (1.5,1.5) -- (1.65,1.5);
	\draw [dotted] (1.65,1.5) -- (2.35,1.5);
	\begin{scope}[xshift=2cm,yshift=1.25cm]
	\draw (-.5,-.5) to[out=0,in=-90] (-.4,0.8) to[out=90,in=180] (0,1.2) to[out=0,in=90] (.4,0.8) to[out=-90,in=180] (.5,-.5);	
	\draw (-.2,0.8) to[out=-90,in=-90] (.2,0.8);
	\draw (-.15,0.75) to[out=90,in=90] (.15,0.75);
	\end{scope}
	\end{scope}
	
	\begin{scope}[xshift=2.5cm]
	\draw (2.35,1.5) -- (4,1.5);
	\draw (2.5,0) -- (0,0);
	\draw (1.5,1.5) -- (1.65,1.5);
	\draw [dotted] (1.65,1.5) -- (2.35,1.5);
	\begin{scope}[xshift=2cm,yshift=1.25cm]
	\draw (-.5,-.5) to[out=0,in=-90] (-.4,0.8) to[out=90,in=180] (0,1.2) to[out=0,in=90] (.4,0.8) to[out=-90,in=180] (.5,-.5);	
	\draw (-.2,0.8) to[out=-90,in=-90] (.2,0.8);
	\draw (-.15,0.75) to[out=90,in=90] (.15,0.75);
	\end{scope}
	\end{scope}
	
	\begin{scope}[xshift=5cm]
	\draw (2.35,1.5) -- (4,1.5) -- (2.5,0) -- (0,0);
	\draw (1.5,1.5) -- (1.65,1.5);
	\draw [dotted] (1.65,1.5) -- (2.35,1.5);
	\begin{scope}[xshift=2cm,yshift=1.25cm]
	\draw (-.5,-.5) to[out=0,in=-90] (-.4,0.8) to[out=90,in=180] (0,1.2) to[out=0,in=90] (.4,0.8) to[out=-90,in=180] (.5,-.5);	
	\draw (-.2,0.8) to[out=-90,in=-90] (.2,0.8);
	\draw (-.15,0.75) to[out=90,in=90] (.15,0.75);
	\end{scope}
	\end{scope}
	
	\node at (4.5,0) [below] {$\Sigma_{3,1}$};
	\draw [thick,Mahogany] (0,0) to[out=27.5,in=-155] (.5,.25) to[out=25,in=-100,looseness=.6] (3,0.65) to[out=80,in=-25,looseness=1.5] (2.3,1.1);
	\draw [thick,Mahogany,densely dotted] (2.2,1.14) to[out=155,in=-35] (1.78,1.32);
	\draw [thick,Mahogany] (1.67,1.36) -- (1.5,1.5);
	
	\node [Mahogany] at (0,0) {$\bullet$};
	\node [Mahogany] at (1.5,1.5) {$\bullet$};
	\end{tikzpicture}
	\caption{A $0$-simplex of $S(\Sigma_{3,1}, (0,0,0), (0,1,0))_\bullet$.}
	\label{fig:s0simplex}
\end{figure}

\begin{proposition}\label{prop:SplittingModel}
There is a $\Gamma_{g,1}$-equivariant isomorphism between $\split_\bullet(g)$ and $S(\Sigma_{g,1}, (0,0,0), (0,1,0))_\bullet$.
\end{proposition}
\begin{proof}
Consider a slightly modified semi-simplicial set $S'(\Sigma_{g,1})_\bullet$ whose $p$-simplices are $(p+1)$-tuples $(a_0, \ldots, a_p)$ of isotopy classes of arcs on $\Sigma_{g,1}$ starting on the interval $[0,g] \times \{(0,0)\}$ and ending on the interval $[0,g] \times \{(1,0)\}$, such that the collection $(a_0,\ldots,a_p)$ may be represented by a collection of embedded separating arcs such that
\begin{enumerate}[(i)]
	\item the $a_i$ are disjoint,		
	\item $a_0(t) < \ldots < a_p(t) \in [0,g]$ for $t=0$ and $t=1$ (in this definition, the endpoints of the arcs are allowed to move),
	\item the region of $\Sigma_{g,1}$ between $a_i$ and $a_{i+1}$ has positive genus, for each $i$.
\end{enumerate}
Face maps are again given by forgetting arcs.
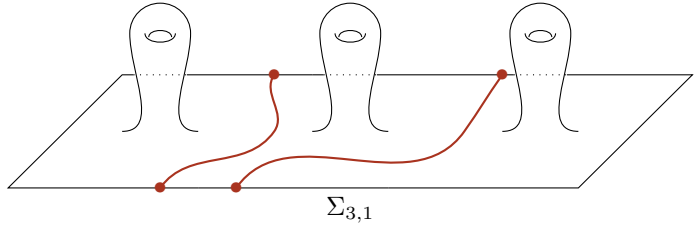
\begin{figure}[h]
	\begin{tikzpicture}
	\begin{scope}[xshift=0cm]
	\draw (2.35,1.5) -- (4,1.5);
	\draw (2.5,0) -- (0,0) -- (1.5,1.5) -- (1.65,1.5);
	\draw [dotted] (1.65,1.5) -- (2.35,1.5);
	\begin{scope}[xshift=2cm,yshift=1.25cm]
	\draw (-.5,-.5) to[out=0,in=-90] (-.4,0.8) to[out=90,in=180] (0,1.2) to[out=0,in=90] (.4,0.8) to[out=-90,in=180] (.5,-.5);	
	\draw (-.2,0.8) to[out=-90,in=-90] (.2,0.8);
	\draw (-.15,0.75) to[out=90,in=90] (.15,0.75);
	\end{scope}
	\end{scope}
	
	\begin{scope}[xshift=2.5cm]
	\draw (2.35,1.5) -- (4,1.5);
	\draw (2.5,0) -- (0,0);
	\draw (1.5,1.5) -- (1.65,1.5);
	\draw [dotted] (1.65,1.5) -- (2.35,1.5);
	\begin{scope}[xshift=2cm,yshift=1.25cm]
	\draw (-.5,-.5) to[out=0,in=-90] (-.4,0.8) to[out=90,in=180] (0,1.2) to[out=0,in=90] (.4,0.8) to[out=-90,in=180] (.5,-.5);	
	\draw (-.2,0.8) to[out=-90,in=-90] (.2,0.8);
	\draw (-.15,0.75) to[out=90,in=90] (.15,0.75);
	\end{scope}
	\end{scope}
	
	\begin{scope}[xshift=5cm]
	\draw (2.35,1.5) -- (4,1.5) -- (2.5,0) -- (0,0);
	\draw (1.5,1.5) -- (1.65,1.5);
	\draw [dotted] (1.65,1.5) -- (2.35,1.5);
	\begin{scope}[xshift=2cm,yshift=1.25cm]
	\draw (-.5,-.5) to[out=0,in=-90] (-.4,0.8) to[out=90,in=180] (0,1.2) to[out=0,in=90] (.4,0.8) to[out=-90,in=180] (.5,-.5);	
	\draw (-.2,0.8) to[out=-90,in=-90] (.2,0.8);
	\draw (-.15,0.75) to[out=90,in=90] (.15,0.75);
	\end{scope}
	\end{scope}
	
	\node at (4.5,0) [below] {$\Sigma_{3,1}$};
	
	\draw [thick,Mahogany] (2,0) to[out=55,in=-125] (3.5,0.75) to[out=55,in=-125] (3.5,1.5);
	\node [Mahogany] at (2,0) {$\bullet$};
	\node [Mahogany] at (3.5,1.5) {$\bullet$};
	
	\draw [thick,Mahogany] (3,0) to[out=55,in=-125] (6,0.75) to[out=55,in=-125] (6.5,1.5);
	\node [Mahogany] at (3,0) {$\bullet$};
	\node [Mahogany] at (6.5,1.5) {$\bullet$};
	
	\end{tikzpicture}
	\caption{A $1$-simplex of $S'(\Sigma_{3,1})_\bullet$.}
	\label{fig:sprimesimplex}
\end{figure}

There is a semi-simplicial map 
\[S'(\Sigma_{g,1})_\bullet \lra S(\Sigma_{g,1}, (0,0,0), (0,1,0))_\bullet\]
given by dragging the endpoints of the arcs to $(0,i,0) \subset [0,g]\times \{(i, 0)\}$ for $i \in \{0,1\}$, and this is an isomorphism of semi-simplicial sets. It is $S'(\Sigma_{g,1})_\bullet$ that we compare to $\split_\bullet(g)$.

For $0< n < g$, let $\gamma_{n} \colon [0,1] \to \Sigma_{g,1}$ be the arc given by $t \mapsto (n, t, 0) \in \Sigma_{g,1}$, a particular 0-simplex of $S'(\Sigma_{g,1})_\bullet$ separating a genus $n$ piece from a genus $g-n$ piece. For $n < m$ the arcs $\gamma_n$ and $\gamma_m$ are disjoint, and satisfy $\gamma_n(t) < \gamma_m(t)$ for $t=0$ and $t=1$. For each tuple $(g_0, \ldots, g_{p+1})$ with $g_i > 0$ such that $g=\sum g_i$, there is $p$-simplex
\[\sigma_{g_0, g_1, \ldots, g_{p+1}} = (\gamma_{g_0}, \gamma_{g_0+g_1}, \ldots, \gamma_{g_0+ \cdots + g_p}) \in S'(\Sigma_{g,1})_p.\]
These arcs cut the surface into  $p+1$ subsurfaces of genus $g_0$, \dots, $g_{p+1}$. The stabiliser of $\sigma_{g_0, \ldots, g_{p+1}}$ under the action of $\Gamma_{g,1}$ on $S'(\Sigma_{g,1})_p$ is the subgroup of $\Gamma_{g,1}$ which fixes all of these arcs up to isotopy: by the isotopy extension theorem this consists of those diffeomorphisms which fix these arcs pointwise, i.e.\ the Young-type subgroup $\Gamma_{(g_0,\ldots,g_{p+1}),1}$ of $\Gamma_{g,1}$ isomorphic to $\Gamma_{g_0,1} \times \cdots \times \Gamma_{g_{p+1},1}$. Acting on the $\sigma_{g_0, \ldots, g_{p+1}}$ therefore gives a map
\begin{equation}\label{eqn:sprimemap} \bigsqcup_{\substack{g_0, \ldots, g_{p+1} > 0\\\sum g_i=g}} \frac{\Gamma_{g,1}}{\Gamma_{(g_0,\ldots,g_{p+1}),1}} \lra S'(\Sigma_{g,1})_p\end{equation}
from the $p$-simplices of $\split_\bullet(g)$, and this defines a semi-simplicial map.

We claim that this semi-simplicial map is an isomorphism. It is surjective on simplices by the classification of surfaces: we may act upon any $(a_0, \ldots, a_p) \in S'(\Sigma_{g,1})_p$ to transform it into $\sigma_{g_0, \ldots, g_{p+1}}$, where $g_i>0$ is the genus of the piece of surface between $a_{i-1}$ and $a_i$ ($a_{-1}$ and $a_{p+1}$ are to interpreted in the obvious way). To see that it is injective, we look at the genus of the pieces of the decomposition to see that all of the terms of left hand of (\ref{eqn:sprimemap}) have disjoint image. Injectivity then follows from the fact $\Gamma_{(g_0,\ldots,g_{p+1}),1}$ is the stabiliser of $\sigma_{g_0, \ldots, g_{p+1}}$.\end{proof}

In view of this proposition, Theorem \ref{thm:MCGSplittingConn} follows when we show that $\fgr{S(\Sigma, b_0, b_1)_\bullet}$ is $(g-3)$-connected if $\Sigma$ has genus $g$.

\subsubsection{The poset model}

In order to employ the techniques described in Section \ref{sec:posets} we will describe $S(\Sigma,b_0,b_1)_\bullet$ as the nerve of a poset, as follows.

\begin{definition}
	Let $\mathcal{S}(\Sigma, b_0, b_1)$ be the poset with underlying set $S(\Sigma, b_0, b_1)_0$, and where $a < a'$ if these isotopy classes may be represented by disjoint arcs so that the ordering $(a, a')$ coincides with the clockwise ordering of the arcs at $b_0$, and the region of $\Sigma$ between $a$ and $a'$ has positive genus.
\end{definition}

The semisimplicial set $S(\Sigma, b_0, b_1)_\bullet$ may be identified with the nerve of the poset $\mathcal{S}(\Sigma, b_0, b_1)$. For an inductive argument, we need a version of the above for (connected) surfaces with more than one boundary. (This will also be used when we discuss homology of mapping class groups of surfaces with several boundaries.)

\begin{definition}\label{defn:PreparedSurface}
A \emph{prepared surface} $(\Sigma, \partial_0 \Sigma, b_0, b_1)$ consists of a connected oriented surface $\Sigma$, a distinguished boundary component $\partial_0 \Sigma \subset \partial \Sigma$, and a pair of distinct points $b_0, b_1 \in \partial_0 \Sigma$. The boundary components which are not distinguished will be called \emph{additional}.

For an arc $a$ from $b_0$ to $b_1$ which splits $\Sigma$ into two path components, the component containing $\{0\} \times (0,1) \times \{0\} \subset \Sigma$ shall be called the \emph{left side} of $a$ and the other one the \emph{right side}, in consistency with the illustrations in this paper.

Let the poset $\mathcal{S}(\Sigma, \partial_0 \Sigma, b_0, b_1)$ consist of isotopy classes of arcs in $\Sigma$ from $b_0$ to $b_1$ such that
\begin{enumerate}[(i)]
	\item $a$ splits $\Sigma$ into two path components,
		
	\item the left side of $a$ has positive genus,
		
	\item the right side of $a$ contains all the additional boundaries, and
		
	\item if there are no additional boundaries, the right side of $a$ also has positive genus.
\end{enumerate}
	
We say $a < a'$ if these $a$ and $a'$ may be represented by disjoint arcs so that the ordering $a, a'$ coincides with the clockwise ordering of the arcs at $b_0$, and the region of $\Sigma$ between $a$ and $a'$ has positive genus.
\end{definition}

\begin{figure}[h]
	\begin{tikzpicture}
	\begin{scope}[xshift=0cm]
	\draw (2.35,1.5) -- (4,1.5);
	\draw (2.5,0) -- (0,0) -- (1.5,1.5) -- (1.65,1.5);
	\draw [dotted] (1.65,1.5) -- (2.35,1.5);
	\begin{scope}[xshift=2cm,yshift=1.25cm]
	\draw (-.5,-.5) to[out=0,in=-90] (-.4,0.8) to[out=90,in=180] (0,1.2) to[out=0,in=90] (.4,0.8) to[out=-90,in=180] (.5,-.5);	
	\draw (-.2,0.8) to[out=-90,in=-90] (.2,0.8);
	\draw (-.15,0.75) to[out=90,in=90] (.15,0.75);
	\end{scope}
	\end{scope}
	
	\begin{scope}[xshift=2.5cm]
	\draw (2.35,1.5) -- (4,1.5);
	\draw (2.5,0) -- (0,0);
	\draw (1.5,1.5) -- (1.65,1.5);
	\draw [dotted] (1.65,1.5) -- (2.35,1.5);
	\begin{scope}[xshift=2cm,yshift=1.25cm]
	\draw (-.5,-.5) to[out=0,in=-90] (-.4,0.8) to[out=90,in=180] (0,1.2) to[out=0,in=90] (.4,0.8) to[out=-90,in=180] (.5,-.5);	
	\draw (-.2,0.8) to[out=-90,in=-90] (.2,0.8);
	\draw (-.15,0.75) to[out=90,in=90] (.15,0.75);
	\end{scope}
	\end{scope}
	
	\begin{scope}[xshift=5cm]
	\draw (2.35,1.5) -- (4,1.5) -- (2.5,0) -- (0,0);
	\draw (1.5,1.5) -- (1.65,1.5);
	\draw (1.65,1.5) -- (2.35,1.5);
	\draw (2,1) ellipse (.4cm and .15cm);
	\draw (1.5,.5) ellipse (.4cm and .15cm);
	\end{scope}

	\node at (4.5,0) [below] {$\Sigma_{2,3}$};
	\draw [thick,Mahogany] (0,0) to[out=27.5,in=-155] (.5,.25) to[out=25,in=-100,looseness=.5] (5.8,0.65) to[out=80,in=-5,looseness=.8] (4.8,1.1);
	\draw [thick,Mahogany,densely dotted] (2.2,1.14) to[out=175,in=-35] (1.78,1.3);
	\draw [thick,Mahogany] (1.67,1.36) -- (1.5,1.5);
	\draw [thick,Mahogany] (2.3,1.12) -- (4.2,1.11);
	\draw [thick,Mahogany,densely dotted] (4.3,1.11) -- (4.7,1.1);
	
	\node [Mahogany] at (0,0) {$\bullet$};
	\node at (0,0) [left,Mahogany] {$b_0$};
	\node [Mahogany] at (1.5,1.5) {$\bullet$};
	\node at (1.5,1.5) [left,Mahogany] {$b_1$};
	\end{tikzpicture}
	\caption{An object of the poset $\mathcal{S}(\Sigma_{2,3}, \partial_0 \Sigma_{2,3}, b_0, b_1)$.}
	\label{fig:sobject}
\end{figure}
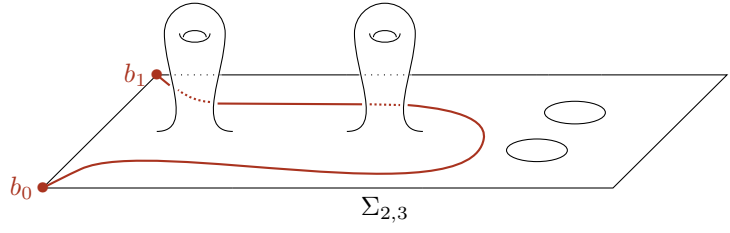

Note that if $\Sigma$ has a single boundary component, which must necessarily be the distinguished one, then this coincides with $\mathcal{S}(\Sigma, b_0, b_1)$. We shall make use of the following auxiliary simplicial complexes (or posets) of systems of non-separating arcs. These first arose in the work of Harer \cite{HarerStab} on homological stability for mapping class groups, but play a rather different role for us: rather than acting on them with the mapping class group, we will use these posets to index a covering of $\mathcal{S}(\Sigma, \partial_0 \Sigma, b_0, b_1)$ by subposets, to which we will apply the Nerve Theorem.

\begin{definition}
Let $\Sigma$ be a connected oriented surface, and $b_2, b_3 \in \partial \Sigma$ be distinct points. Let $B_0(\Sigma, b_2, b_3)$ be the simplicial complex with vertices the isotopy classes of arcs from $b_2$ to $b_3$ which are non-separating. A $(p+1)$-tuple of such isotopy classes is a $p$-simplex if they can be represented by arcs which are disjoint apart from their endpoints and are jointly non-separating.
	
Let $\mathcal{B}_0(\Sigma, b_2, b_3)$ denote the poset of simplices of $B_0(\Sigma, b_2, b_3)$.
\end{definition}

\begin{figure}[h]
	\begin{tikzpicture}
	\begin{scope}[xshift=0cm]
	\draw (2.35,1.5) -- (4,1.5);
	\draw (2.5,0) -- (0,0) -- (1.5,1.5) -- (1.65,1.5);
	\draw [dotted] (1.65,1.5) -- (2.35,1.5);
	\begin{scope}[xshift=2cm,yshift=1.25cm]
	\draw (-.5,-.5) to[out=0,in=-90] (-.4,0.8) to[out=90,in=180] (0,1.2) to[out=0,in=90] (.4,0.8) to[out=-90,in=180] (.5,-.5);	
	\draw (-.2,0.8) to[out=-90,in=-90] (.2,0.8);
	\draw (-.15,0.75) to[out=90,in=90] (.15,0.75);
	\end{scope}
	\end{scope}
	
	\begin{scope}[xshift=2.5cm]
	\draw (2.35,1.5) -- (4,1.5);
	\draw (2.5,0) -- (0,0);
	\draw (1.5,1.5) -- (1.65,1.5);
	\draw [dotted] (1.65,1.5) -- (2.35,1.5);
	\begin{scope}[xshift=2cm,yshift=1.25cm]
	\draw (-.5,-.5) to[out=0,in=-90] (-.4,0.8) to[out=90,in=180] (0,1.2) to[out=0,in=90] (.4,0.8) to[out=-90,in=180] (.5,-.5);	
	\draw (-.2,0.8) to[out=-90,in=-90] (.2,0.8);
	\draw (-.15,0.75) to[out=90,in=90] (.15,0.75);
	\end{scope}
	\end{scope}
	
	\begin{scope}[xshift=5cm]
	\draw (2.35,1.5) -- (4,1.5) -- (2.5,0) -- (0,0);
	\draw (1.5,1.5) -- (1.65,1.5);
	\draw (1.65,1.5) -- (2.35,1.5);
	\draw (1.6,.75) ellipse (.65cm and .24cm);
	\end{scope}

	\node at (4.5,0) [below] {$\Sigma_{2,2}$};
		
	\node [Mahogany] at ({7.5+.75},.75) {$\bullet$};
	\node at ({7.5+.75+.05},.7) [right,Mahogany] {$b_2$};
	\node [Mahogany] at ({5+1.75+.5},.75) {$\bullet$};
	\node at ({5+1.75+.5},.75) [left,Mahogany] {$b_3$};
	\draw [thick,Mahogany] ({7.5+.75},.75) -- ({5+1.75+.5},.75);
	\end{tikzpicture}
	\caption{A $0$-simplex of $B_0(\Sigma_{2,2}, \partial_0 \Sigma_{2,2}, b_2, b_3)$, with $b_2$ and $b_3$ chosen as in case (ii) in the proof of Theorem \ref{thm:ArcCx}.}
	\label{figb:bsimplex}
\end{figure}
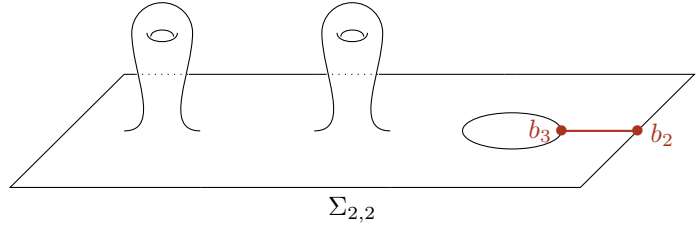

The property of these simplicial complexes we shall use is the following, originally due to Harer \cite[Theorem 1.4]{HarerStab} but see also \cite[Theorem 4.8]{WahlSurvey}.

\begin{theorem}\label{thm:ArcCx}
	Let $\Sigma$ have genus $g$.
	\begin{enumerate}[(i)]
		\item If $b_2$ and $b_3$ lie on the same boundary component, then $B_0(\Sigma, b_2, b_3)$ is $(2g-2)$-connected.
		
		\item If $b_2$ and $b_3$ lie on different boundary components, then $B_0(\Sigma, b_2, b_3)$ is $(2g-1)$-connected.
	\end{enumerate}
\end{theorem}

\subsubsection{The technical theorem}

The following implies that if $\Sigma$ is a genus $g$ surface then $\fgr{S(\Sigma, b_0, b_1)_\bullet}$ is $(g-3)$-connected. By Proposition \ref{prop:SplittingModel}, this statement implies Theorem \ref{thm:MCGSplittingConn}, but is better suited to proof by induction.

\begin{theorem}\label{thm:HighConnSCx}
Let $(\Sigma, \partial_0 \Sigma, b_0, b_1)$ be a prepared surface, and $\Sigma$ have genus $g$.
\begin{enumerate}[(i)]
	\item If there are no additional boundaries then $\mathcal{S}(\Sigma, \partial_0 \Sigma, b_0, b_1)$ is $(g-3)$-connected.
		
	\item If there are additional boundaries then $\mathcal{S}(\Sigma, \partial_0 \Sigma, b_0, b_1)$ is $(g-2)$-connected.
\end{enumerate}
\end{theorem}
\begin{proof}
We will prove the two statements simultaneously, by induction on $g$ and, in case (ii), on the number of additional boundaries. If $g \leq 1$ then both statements are clear: (i) has no content, and (ii) says that $\mathcal{S}(\Sigma, \partial_0 \Sigma, b_0, b_1)$ is non-empty, but as there are additional boundaries we can choose an arc with a genus one surface to its left and the additional boundaries to its right. Thus suppose that $g \geq 2$.
	
	\vspace{2ex}
	
	\noindent \textbf{Case (i)}. Let us suppose that (i) and (ii) hold for all genera $g'<g$, then we will prove (i) for genus $g$. Choose points $b_2, b_3 \in \partial_0 \Sigma$ between $b_0$ and $b_1$ on the right-hand side of $b_0$. There is then defined the poset $\mathcal{B}_0(\Sigma, b_2, b_3)$, and we consider the functor
	\[F \colon \mathcal{B}_0(\Sigma, b_2, b_3)^\mr{op} \lra \{\text{closed subposets of $\mathcal{S}(\Sigma, \partial_0 \Sigma, b_0, b_1)$}\}\]
	which sends a tuple of arcs $(a_0, \ldots, a_p)$ to the subposet $F(a_0, \ldots, a_p)$ of the poset $\mathcal{S}(\Sigma, \partial_0 \Sigma, b_0, b_1)$ consisting of those arcs $a$ from $b_0$ to $b_1$ for which the $a_i$ lie on the right side of $a$. This is clearly a closed subposet, and defines a functor.
	
	We wish to apply the Nerve Theorem (Corollary \ref{cor.nervelemma}). Let $n \coloneqq g-2$, define $t_{\mathcal{S}}(a)$ to be one less than the genus of the surface to the left of $a$, and define $t_{\mathcal{B}_0}(a_0, \ldots, a_p) \coloneqq p$. It remains to verify the hypotheses of the Nerve Theorem. 
	
\begin{enumerate}[(i)]
	\item  $\mathcal{B}_0(\Sigma, b_2, b_3)$ is $(2g-2)$-connected by Theorem \ref{thm:ArcCx}, so in particular it is $(g-3)=(n-1)$-connected because clearly $2g-2 \geq g-3$ if $g \geq 0$. 

	\item Let $(a_0, \ldots, a_p) \in \mathcal{B}_0(\Sigma, b_2, b_3)$. Then $\mathcal{B}_0(\Sigma, b_2, b_3)_{< (a_0, \ldots, a_p)}$ is nothing but the boundary of a $p$-simplex, so is $(p-2) = (t_{\cB_0}(a_0, \ldots, a_p)-2)$-connected, as required. On the other hand, the subposet $F(a_0, \ldots, a_p) \leq \mathcal{S}(\Sigma, \partial_0 \Sigma, b_0, b_1)$ may be considered as a poset $\mathcal{P}$ of arcs in the cut surface $\Sigma' \coloneqq \Sigma \setminus (a_0, \ldots, a_p)$. We can consider the boundary containing $b_0$ and $b_1$ as distinguished, called $\partial _0\Sigma'$, and the others as additional: this endows $\Sigma'$ with the structure of a prepared surface. In $\mathcal{P}$ arcs are still required to go from $b_0$ to $b_1$, and to have positive genus to their left, but any additional boundaries must be to their right and no genus need be to their right. This is precisely the poset $\mathcal{S}(\Sigma', \partial_0 \Sigma', b_0, b_1)$. Removing a single arc from a surface either reduces the genus by 1 and creates a new boundary, or reduces the number of boundary components and preserves the genus. As $\Sigma$ has a single boundary, $\Sigma \setminus a_0$ has two boundary components and genus $g-1$. Inductively, we see that $\Sigma'$ either has
\begin{enumerate}[(a)]
\item $g-p-1 \leq g(\Sigma') < g$ and additional boundaries, or
\item $g-p \leq g(\Sigma') < g$ and no additional boundaries.
\end{enumerate}
In either case, by inductive assumption, using (ii) or (i), $\mathcal{S}(\Sigma', \partial_0 \Sigma', b_0, b_1)$ is $((g-2)-p-1) =  (n-t_{\cB_0}(a_0,\ldots,a_p)-1)$-connected, as required.

	\item Let $a \in \mathcal{S}(\Sigma, \partial_0 \Sigma, b_0, b_1)$, and suppose it has genus $g_0$ to its left and $g_1 = g-g_0$ to its right. Then $\mathcal{S}(\Sigma, \partial_0 \Sigma, b_0, b_1)_{<a}$ is the same type of poset of splittings, but for the surface (of genus $g_0$ and having a single boundary component) to the left of $a$. By case (i) for $g_0 < g$ (as we must have $g_1 > 0$), it is $(g_0-3)$-connected, so $((g_0-1)-2) = (t_{\cS}(a)-2)$-connected, as required. On the other hand, the subposet $\mathcal{B}_0(\Sigma, b_2, b_3)_a$ is the same type of poset of non-separating arcs, but for the surface (of genus $g_1$) to the right of $a$. Thus it is $(2g_1-2)$-connected by Theorem \ref{thm:ArcCx}, so in particular $(g_1-2)=((g-2)-(g_0-1)-1) = (n-t_\cS(a)-1)$-connected, as required.
\end{enumerate}
	
Thus the Nerve Theorem applies and we conclude that $\mathcal{S}(\Sigma, \partial_0 \Sigma, b_0, b_1)$ is $(g-3)$-connected.
	
\vspace{2ex}
	
\noindent \textbf{Case (ii)}. Let us suppose that (i) holds for all genera $g' \leq g$, and that (ii) holds for all genera $g' < g$. Let $k>0$ be the number of additional boundaries, and suppose that (ii) also holds for genus $g$ and $k' < k$ additional boundaries. Choose a point $b_2 \in \partial_0 \Sigma$ between $b_0$ and $b_1$ on the right-hand side of $b_0$, and a point $b_3 \in \partial \Sigma$ on an additional boundary. There is then defined the poset $\mathcal{B}_0(\Sigma, b_2, b_3)$, and we consider the functor
	\[F \colon \mathcal{B}_0(\Sigma, b_2, b_3)^\mr{op} \lra \{\text{closed subposets of $\mathcal{S}(\Sigma, \partial_0 \Sigma, b_0, b_1)$}\}\]
	which sends a tuple of arcs $(a_0, \ldots, a_p)$ to the subposet $F(a_0, \ldots, a_p)$ of the poset $\mathcal{S}(\Sigma, \partial_0 \Sigma, b_0, b_1)$ consisting of those arcs $a$ from $b_0$ to $b_1$ for which the $a_i$ lie on the right side of $a$.
	
	We wish to apply the Nerve Theorem (Corollary \ref{cor.nervelemma}), similarly to case (i) above. Let $n \coloneqq g-1$, define $t_{\mathcal{S}}(a)$ to be one less than the genus of the surface to the left of $a$, and define $t_{\mathcal{B}_0}(a_0, \ldots, a_p) \coloneqq p$. It remains to verify the hypotheses of the Nerve Theorem. 
	
\begin{enumerate}[(i)]
	\item  $\mathcal{B}_0(\Sigma, b_2, b_3)$ is $(2g-1)$-connected by Theorem \ref{thm:ArcCx}, so in particular $(g-2)=(n-1)$-connected. 

	\item Let $(a_0, \ldots, a_p) \in \mathcal{B}_0(\Sigma, b_2, b_3)$. Then $\mathcal{B}_0(\Sigma, b_2, b_3)_{< (a_0, \ldots, a_p)}$ is nothing but the boundary of a $p$-simplex, so is $(p-2) = (t_{\cB_0}(a_0,\ldots,a_p)-2)$-connected, as required. On the other hand, the subposet $F(a_0, \ldots, a_p) \leq \mathcal{S}(\Sigma, \partial_0 \Sigma, b_0, b_1)$ may be considered as a poset $\mathcal{P}$ of arcs in the cut surface $\Sigma' \coloneqq \Sigma \setminus (a_0, \ldots, a_p)$. We can consider the boundary containing $b_0$ and $b_1$ as distinguished, called $\partial \Sigma'$, and the remaining boundaries as additional. In this cut surface, arcs in $\mathcal{P}$ are still required to go from $b_0$ to $b_1$, and to have positive genus to their left, but all the additional boundaries must be to their right, and no genus need be to their right. There are now two cases. 

If $\Sigma'$ has no additional boundaries, then in $\mathcal{P}$ we are allowed arcs which have genus zero to their right. There is one such isotopy class of arc, and it lies above every other element in $\mathcal{P}$: thus $\mathcal{P}$ is contractible, and in particular is $((g-1)-p-1) = (n-t_{\cB_0}(a_0,\ldots,a_p)-1)$-connected. 

On the other hand, if $\Sigma'$ has additional boundaries then the poset $\mathcal{P}$ is $\mathcal{S}(\Sigma', \partial_0 \Sigma', b_0, b_1)$. The surface $\Sigma'$ either has smaller genus than $\Sigma$ or fewer additional boundaries (or both), so by inductive assumption (ii) applies.  Removing the arc $a_0$ decreases the number of boundaries but preserves the genus, and hence $g(\Sigma') \geq g-p$ and so the poset is $((g-p)-2)=((g-1)-p-1)$-connected, as required.

	\item Let $a \in \mathcal{S}(\Sigma, \partial_0 \Sigma, b_0, b_1)$, and suppose it has genus $g_0$ to its left and $g_1 = g-g_0$ to its right. Then $\mathcal{S}(\Sigma, \partial_0 \Sigma, b_0, b_1)_{<a}$ is the same type of poset of splittings, but for the surface (of genus $g_0$ and having a single boundary) to the left of $a$. By case (i) and $g_0 \leq g$, it is $(g_0-3)$-connected, so $((g_0-1)-2) = (t_\cS(a)-2)$-connected, as required. On the other hand, the subposet $\mathcal{B}_0(\Sigma, b_2, b_3)_a$ is the same type of poset of non-separating arcs, but for the surface (of genus $g_1$) to the right of $a$. Thus it is $(2g_1-1)$-connected by Theorem \ref{thm:ArcCx}, so in particular $(g_1-1)=((g-1)-(g_0-1)-1) = (n-t_\cS(a)-1)$-connected, as required.
\end{enumerate} 
	
	Thus the Nerve Theorem applies and we conclude that $\mathcal{S}(\Sigma, \partial_0 \Sigma, b_0, b_1)$ is $(g-2)$-connected.
\end{proof}

\section{Homological stability and secondary homological stability} \label{sec:stability} We now apply the theory of \cite{e2cellsI} in the manner explained in Section \ref{sec:homology-theory-e_2} to the $E_2$-algebra $\gR$, using the results of the previous two sections.

\subsection{Homology stability}
We now explain the first consequences for the $E_2$-cell structure on $\gR_\bZ$ following from the standard connectivity estimate and the low-degree low-genus computations. We will later do a more refined analysis of the cell structure, and deduce from it Theorem \ref{thm:A} and several parts of Theorem \ref{thm:B}.

\begin{proposition}\label{prop:MCGE2vanish} We have $H^{E_2}_{g,d}(\gR_\bZ)=0$ for $d<g-1$.\end{proposition}

\begin{proof}
  Corollary \ref{cor:MCGVanishingLine} shows that $H^{E_1}_{g,d}(\gR_\bZ)=0$ for $d < g-1$, and by transferring vanishing lines up using Theorem $E_k$.14.2, we have $\smash{H^{E_2}_{g,d}}(\gR_\bZ)=0$ for $d < g-1$ too. 
\end{proof}

Furthermore, by Lemma \ref{lem:MCGFacts} (\ref{it:MCGFacts21}) we see that the map $\sigma \cdot - \colon H_{1,1}(\gR_\bZ) \to H_{2,1}(\gR_\bZ)$ is surjective. In particular, the general homological stability result Theorem $E_k$.18.2 applies to $\gR_\bZ$, including the second part. The first part of this theorem shows that $H_{g,d}(\overline{\gR}_\bZ/\sigma)=0$ for $3d \leq 2g-1$, giving the following vanishing range.

\begin{corollary}\label{cor:MCGStabRange}
	We have $H_d(\Gamma_{g,1}, \Gamma_{g-1,1};\bZ)=0$ for $3d \leq 2g-1$.
\end{corollary}
This is slightly better than the vanishing range which can be deduced from \cite{Boldsen, R-WResolution} (which is $3d \leq 2g-2$). In particular, this corollary recovers and slightly improves all previously known homological stability results for mapping class groups.

The second part of Theorem $E_k$.18.2 tells us the vanishing of another $E_2$-homology group, not covered by Proposition \ref{prop:MCGE2vanish}.

\begin{corollary}\label{cor:MCG21vanish} We have $H^{E_2}_{2,1}(\gR_\bZ)=0$.\end{corollary}

\begin{remark}The discussion of Koszul duality in Section $E_k$.20.3 says that the standard connectivity estimate implies the following. Consider the functor $\underline{\bk}_{>0} \colon \cat{MCG} \to \cat{Mod}_\bk$ given by $\bk$ if $g>0$ and $0$ if $g=0$. Not only does $\underline{\bk}_{>0}$ admit a presentation as a quadratic algebra generated by the trivial $\Gamma_{1,1}$-representation $\bk$ in genus 1 and relations $\ker(\bk[\Gamma_{2,1}/(\Gamma_{1,1} \times \Gamma_{1,1})] \to \bk)$ in genus 2 only, but it has $E_1$-homology groups $H^{E_1}_{g,d}(\underline{\bk}_{>0})$ concentrated along the line $d = g-1$ and given by the associated shifted Koszul dual quadratic coalgebra.\end{remark}

\subsection{Rational secondary homological stability}\label{sec:RatStab} Let us once-and-for-all choose a homotopy equivalence $S^1 \to \cC_2(2)$, so that the image of the fundamental cycle gives a $1$-simplex in $(\cC_2(2))_\bQ \in \cat{sMod}_\bQ$ representing the Browder bracket $[-,-]$. Let $\gA$ be the $E_2$-algebra in $\cat{sMod}_\bQ^\bN$ given by the presentation
\begin{equation*}
\gA \coloneqq \gE_2(S^{1,0}_\bQ\sigma  \oplus S^{3,2}_\bQ\lambda) \cup^{E_2}_{[\sigma,\sigma]} \gD^{2,2}_\bQ \rho,
\end{equation*}
where the letters $\sigma$, $\lambda$ and $\rho$ index the cells, their names serving as reminder for their eventual image in $\gR_\bQ$. The bi-graded spheres $S^{g,d}_\bQ$ are as defined in (\ref{eqn:bigraded-spheres}) and the bi-graded disks $D^{g,d}_\bQ$ are defined analogously.

We have homotopy classes $\sigma$ and $\lambda$ in $\pi_{*,*}(\gR_\bQ)$, cf.~(\ref{eqn:sigma}) and Lemma \ref{lem:MCGFacts} (\ref{it:MCGFacts32abs}), and $\sigma$ satisfies $[\sigma,\sigma] = 0$ in $H_{2,1}(\gR_\bQ)$, as this group vanishes by Lemma \ref{lem:MCGFacts} (\ref{it:MCGFactsQ1}). Let us choose maps $\sigma \colon S^{1,0}_\bQ \to \gR_\bQ$ and $\lambda \colon S^{3,2}_\bQ \to \gR_\bQ$ representing $\sigma \in H_{1,0}(\gR_\bQ)$ and $\lambda \in H_{3,2}(\gR_\bQ)$. Also choose a null-homotopy $\rho \colon D^{2,2}_\bQ \to \gR_\bQ$ of $[\sigma, \sigma] \colon \smash{S^{2,1}_\bQ} \to \gR_\bQ$. Using this data we may construct a map
\[c \colon \gA \lra \gR_\bQ\]
in $\Alg_{E_2}(\cat{sMod}_\bQ^\bN)$. That is, the formal generator $\sigma$ of $\gA$ is sent to the representative $\sigma \colon \smash{S^{1,0}_\bQ} \to \gR_\bQ$, and similarly for $\lambda$ and $\rho$.

The main virtue of the $E_2$-algebra $\gA$ is that it ``contains all the cells of small slope'' in $\gR_\bQ$: it is possible to build a CW approximation to $\gR_\bQ$ by starting from $\gA$ by attaching only cells of slope $\geq \frac34$.  We shall not work directly with such a CW approximation, instead using the following vanishing result for relative $E_2$-homology (defined as the relative homology of $Q^{E_2}_\bL(\gA) \to Q^{E_2}_\bL(\gR_\bQ)$ as in Definition $E_k$.10.7).
\begin{lemma}\label{lem:VanishingRelE2HomologyMCG}
We have $H^{E_2}_{g,d}(\gR_\bQ, \gA)=0$ for $4d \leq 3g-1$.
\end{lemma}

\begin{proof}
We already know that $H^{E_2}_{g,d}(\gR_\bZ)=0$ for $d < g-1$, and $\gA$ enjoys the same vanishing because it only has cells in bidegrees $(1,0)$, $(3,2)$ and $(2,2)$. From the long exact sequence ($E_k$.10.1) in $E_2$-homology of the map $c$ we conclude that $\smash{H^{E_2}_{g,d}}(\gR_\bQ, \gA)=0$ for $d < g-1$, so it remains to show that $\smash{H^{E_2}_{g,d}}(\gR_\bQ, \gA)$ vanishes in bidegrees $(g,d)$ given by $(1,0)$, $(2,1)$, and $(3,2)$. 

The long exact sequence for $E_2$-homology for the pair $(\gR_\bQ, \gA)$ also gives
\[H^{E_2}_{1,0}(\gA) = \bQ\{\sigma\} \overset{\cong}\lra H^{E_2}_{1,0}(\gR_\bQ) = \bQ\{\sigma\} \lra H^{E_2}_{1,0}(\gR_\bQ, \gA) \lra 0\]
so $H^{E_2}_{1,0}(\gR_\bQ, \gA)=0$, and
\[H^{E_2}_{2,1}(\gR_\bQ) \lra H^{E_2}_{2,1}(\gR_\bQ, \gA) \lra H^{E_2}_{2,0}(\gA) = 0,\]
with $H^{E_2}_{2,1}(\gR_\bQ)=0$ by Corollary \ref{cor:MCG21vanish}, so $H^{E_2}_{2,1}(\gR_\bQ, \gA)=0$. 

Finally, combining Lemma \ref{lem:MCGFacts} (\ref{it:MCGFacts11}), (\ref{it:MCGFacts21}) and (\ref{it:MCGFactsg1}) shows that $H_{g, 1}(\gR_\bQ, \gA)=0$ for all $g$, and combining Lemma \ref{lem:MCGFacts} (\ref{it:MCGFactsg2}) and (\ref{it:MCGFacts32abs}) shows that $H_{g, 2}(\gR_\bQ, \gA)=0$ for all $g \leq 3$: thus the map $c$ is $(3,3,3,3,2,2,\ldots)$-connective in the sense of Definition $E_k$.11.2, that is, the relative homology groups $H_d(\gR_\bQ(g),\gA(g))$ vanish for $d < 3$ if $g \leq 3$, and for $d<2$ if $g \geq 4$. 

The objects $\gA$ and $\gR_\bQ$ are 0-connective, so it follows from Proposition $E_k$.11.9 (applied to $c \colon \gA \to \gR_\bQ$ with connectivities $c = (0,0,0,\ldots)$, $c_f = (3,3,3,3,2,2,\ldots)$ as above, and the operad $\cO = \cC_2$ satisfying $\cO(1) \simeq \ast$) that the map between homotopy cofibres
\begin{equation}\label{eq:RelHurewicz}
\gR_\bQ/\gA \lra Q_\bL^{E_2}(\gR_\bQ)/Q_\bL^{E_2}(\gA)
\end{equation}
is $(4,4,4,4,3,3,\ldots)$-connective. In particular the Hurewicz map $H_{3,2}(\gR_\bQ, \gA) \to H_{3,2}^{E_2}(\gR_\bQ, \gA)$ is an isomorphism, and so the latter group vanishes.
\end{proof}

By iterating Definition $E_k$.12.14, there is a left $\overline{\gR}_\bQ$-module $\overline{\gR}_\bQ/(\sigma, \lambda)$ which fits into a cofibre sequence
\begin{equation}\label{eq:MultKappa}
S^{3,2} \otimes \overline{\gR}_\bQ/\sigma \overset{\phi(\lambda)}\lra \overline{\gR}_\bQ/\sigma \lra \overline{\gR}_\bQ/(\sigma, \lambda) \overset{\partial}\lra S^{3,3} \otimes \overline{\gR}_\bQ/\sigma
\end{equation}
in the category of $\overline{\gR}_\bQ$-modules, where after forgetting the $\overline{\gR}_\bQ$-module structure, $\phi(\lambda)$ is homotopic to left multiplication by $\lambda \in H_{3,2}(\overline{\gR}_\bQ)$. One can make the analogous definition for $\gA$ and since the representatives $\sigma \colon S^{1,0}_\bQ \to \gR_\bQ$ and $\lambda \colon S^{3,2}_\bQ \to \gR_\bQ$ by construction lift to $\gA$ along $c$, equation ($E_k$.12.11) in Section $E_k$.12.2.4 gives us a weak equivalence
\[\overline{\mathbf{R}}_{\mathbb{Q}}/(\sigma, \lambda) \simeq \overline{\mathbf{R}}_{\mathbb{Q}} \otimes^{\mathbb{L}}_{\overline{\mathbf{A}}} \overline{\mathbf{A}}/(\sigma, \lambda) \]
of left $\overline{\gR}_\bQ$-modules. As $\overline{\gA}/(\sigma, \lambda)$, $\overline{\gA}$, and $\overline{\gR}_\bQ$ are cofibrant in $\cat{sMod}_\bQ^\bN$, by Lemma $E_k$.9.16 we may compute the derived tensor product using the bar construction $B(\overline{\gR}_\bQ;\overline{\gA};\overline{\gA}/(\sigma, \lambda))$. This means it shall suffice to establish a vanishing line for $\overline{\gA}/(\sigma, \lambda)$.

To do so, we use the cell attachment filtration $f\gA$ on $\gA$. This is an $E_2$-algebra in the category $\bN$-graded simplicial $\bQ$-modules with an additional filtration, i.e.~$f\gA \in \Alg_{E_2}((\cat{sMod}_\bQ^\bN)^{\bZ_\leq})$ (and \emph{not} a filtration of $\gA$ by $E_2$-algebras). It is obtained as follows. Whenever we have a filtered object $X$, we may evaluate it at $d \in \bZ_\leq$ to get its $d$th filtration step. This evaluation functor $d^*$ has a left adjoint $d_*$. We may put $\gE_2(S^{1,0}_\bQ \sigma \oplus S^{3,2}_\bQ \lambda)$ in filtration degree $0$ by putting its generators in filtration degree $0$, $\gE_2(0_* S^{1,0}_\bQ \sigma \oplus 0_* S^{3,2}_\bQ \lambda)$. Since its first filtration step (like its zeroth) is equal to $\gE_2(S^{1,0}_\bQ \sigma \oplus S^{3,2}_\bQ \lambda)$, the map $[\sigma,\sigma] \colon S^{2,1}_\bQ \to \gE_2(S^{1,0}_\bQ \sigma \oplus S^{3,2}_\bQ \lambda)$ gives us a filtered map $1_* \partial D^{2,2}_\bQ \to \gE_2(0_*S^{1,0}_\bQ \sigma \oplus 0_*S^{3,2}_\bQ \lambda)$. Using this as attaching map for an $E_2$-cell in $\Alg_{E_2}((\cat{sMod}_\bQ^\bN)^{\bZ_\leq})$ we may form the filtered algebra
\[f\gA \coloneqq \left(\gE_2(0_* S^{1,0}_\bQ\sigma  \oplus 0_* S^{3,2}_\bQ\lambda)\right) \cup^{E_2}_{[\sigma, \sigma]} 1_* \mathbf{D}^{2,2}_{\mathbb{Q}} \rho) \in \Alg_{E_2}((\cat{sMod}_\bQ^\bN)^{\bZ_{\leq}}).\]
Its underlying object, defined as the colimit over $\bZ_{\leq}$, is $\gA$.

\begin{lemma}\label{lem:VanishA}
	$H_{g,d}(\overline{\gA}/(\sigma, \lambda))=0$ for $d < \tfrac{3}{4} g$.
\end{lemma}
\begin{proof}
Let us give $\gA$ the cell-attachment filtration as explained above, that is consider $\gA$ as the colimit of the filtered $E_2$-algebra $f\gA$. An important property of the cell-attachment filtration is that by passing to the associated graded the attaching maps become trivial. More precisely, Theorem $E_k$.6.4 implies that
\[\grr(f\gA) = \gE_2(S^{1,0,0}_\bQ\sigma  \oplus S^{3,2,0}_\bQ\lambda \oplus S^{2,2,1}_\bQ \rho) \in \Alg_{E_2}((\cat{sMod}_\bQ^\bN)^{\bZ_{=}}),\]
where the tri-graded spheres $S^{g,d,r}_\bQ$ are defined analogously to the bigraded spheres of (\ref{eqn:bigraded-spheres}).

We obtain a corresponding unital and strictly associative algebra $\overline{f\gA}$. Because
\[\overline{f\gA}(0) = \overline{\gE_2(S^{1,0}_\bQ\sigma  \oplus S^{3,2}_\bQ\lambda)},\]
we obtain by adjunction filtered maps
\[\sigma \colon 0_*(S^{1,0}_\bQ) \lra \overline{f\gA} \quad \text{ and }\quad \lambda 
\colon 0_*(S^{3,2}_\bQ) \lra \overline{f\gA},\]
lifting the maps of the same name into $\overline{\gA}$, and using these we may form the $\overline{f\gA}$-module $\overline{f\gA}/(\sigma, \lambda) \in (\cat{sMod}_\bQ^\bN)^{\bZ_{\leq}}$, which on taking colimits recovers the $\overline{\gA}$-module $\overline{\gA}/(\sigma, \lambda)$.

Thus the spectral sequence of the filtered object $\overline{f\gA}$ as described in Corollary $E_k$.10.17 takes the form
\[F^1_{g,p,q} = H_{g,p+q,q}\left(\overline{\gE_2(S^{1,0,0}_\bQ\sigma  \oplus S^{3,2,0}_\bQ\lambda \oplus S^{2,2,1}_\bQ \rho)};\bQ\right) \Rightarrow H_{g,p+q}(\overline{\gA}),\]
and by Section $E_k$.16.6 is a multiplicative spectral sequence whose $d^1$-differential satisfies $d^1(\rho)=[\sigma,\sigma]$. The $F^1$-page of this spectral sequence is the rational homology of a free $E_2$-algebra, and hence given by the free (1-)Gerstenhaber algebra on the generators $\{\sigma[1,0,0], \lambda[3,2,0], \rho[2,2,1]\}$. In particular, as a graded-commutative algebra it is free on the free (1-)Lie algebra $L$ on these three generators, see Figure \ref{fig:lgens}.

\begin{figure}[h]
	\begin{tikzpicture}
	\begin{scope}
	\draw (-1,0)--(5.5,0);
	\draw (0,-1) -- (0,3.5);
	
	\foreach \s in {0,...,3}
	{
		\draw [dotted] (-.5,\s)--(5.5,\s);
		\node [fill=white] at (-.25,\s) [left] {\tiny $\s$};
	}
	
	\foreach \s in {1,...,4}
	{
		\draw [dotted] ({1.25*\s},-0.5)--({1.25*\s},3.5);
		\node [fill=white] at ({1.25*\s},-.5) {\tiny $\s$};
	}
	\node [fill=white] at (0,-.5) {\tiny 0};
	\node [fill=white] at ({1.25*1},0) {$\sigma$};
	\node [fill=white] at ({1.25*2},1) {$[\sigma,\sigma]$};
	\node [fill=white] at ({1.25*2},2) {$\rho$};
	\node [fill=white] at ({1.25*3},2) {$\lambda$};
	\node [fill=white] at ({1.25*3},3) {$[\rho,\sigma]$};
	\node [fill=white] at ({1.25*4},3) {$[\lambda,\sigma]$};

	\node [fill=white] at (-.5,-.5) {$\nicefrac{d}{g}$};
	\end{scope}
	\end{tikzpicture}
	\caption{The additive generators of $L$ in the range $g \leq 4$, $d \leq 3$, with filtration degree suppressed. The Jacobi relation implies that $[\sigma,[\sigma,\sigma]]=0$ in bidegree $(3,2)$.}
	\label{fig:lgens}
\end{figure}
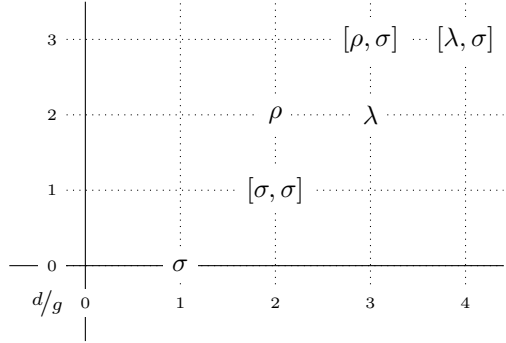

Similarly, the spectral sequence of the filtered object $\overline{f\gA}/(\sigma, \lambda)$ takes the form	
\[E^1_{g,p,q} = H_{g,p+q,q}\left(\overline{\gE_2(S^{1,0,0}_\bQ\sigma  \oplus S^{3,2,0}_\bQ\lambda \oplus S^{2,2,1}_\bQ \rho)}/(\sigma, \lambda);\bQ\right) \Rightarrow H_{g,p+q}(\overline{\gA}/(\sigma, \lambda)).\]
The indices in the trigrading $(g,p,q)$ of both spectral sequences have $g$ denoting the internal (``genus'') grading arising from starting with an $E_2$ algebra in \emph{graded} simplicial sets, $p$ denoting homological degree, and $q$ denoting the extra grading arising from passing to associated graded. The spectral sequence $E^r_{*,*,*}$ has the structure of a module over the multiplicative spectral sequence $F^r_{*,*,*}$, and the $E^1$-page is the free graded-commutative algebra $\Lambda_\bQ(L/\langle \sigma, \lambda \rangle)$ on the quotient vector space $L/\langle \sigma, \lambda \rangle$. The $d^1$-differential on $E^1_{*,*,*}$ is determined by the $d^1$-differential on $F^1_{*,*,*}$, and hence gives $E^1_{*,*,*}$ the structure of a CDGA.

\vspace{1ex}

\noindent\textbf{Claim}: $E^2_{g,p,q}$ vanishes for $\frac{p+q}{g} < \frac{3}{4}$.

\begin{proof}[Proof of claim]

The groups $E^2_{*,*,*}$ are given by the homology groups of the chain complex $(E^1_{*,*,*}, d^1) = (\Lambda_\bQ(L/\langle \sigma, \lambda \rangle),d^1)$, so to estimate these we will introduce an additional ``computational'' filtration on this chain complex which has the virtue of filtering away most of the $d^1$-differential. Let $F^\bullet E^1_{*,*,*}$ be the filtration in which $[\sigma,\sigma]$ and $\rho$ are given filtration 0, the remaining elements of a homogeneous basis of $L/\langle \sigma, \lambda \rangle$ extending these are given filtration equal to their homological degree, and this filtration is extended to $\Lambda_\bQ(L/\langle \sigma, \lambda \rangle)$ multiplicatively. The differential $d^1$ preserves this filtration.

The associated graded $\grr(F^\bullet E^1_{*,*,*})$ of this filtered chain complex can be identified with the tensor product of chain complexes 
\[(\Lambda_\bQ(\bQ\{[\sigma,\sigma],\rho\}), \delta) \otimes (\Lambda_\bQ(L/\langle \sigma,\lambda,[\sigma,\sigma],\rho \rangle), 0)\]
where $\delta([\sigma,\sigma])=0$ and $\delta(\rho) = [\sigma,\sigma]$. The first factor has homology $\bQ$, in degree $(0,0,0,0)$, so the spectral sequence for this filtration starting from its second page has the form
\[\Lambda_\bQ(L/\langle \sigma,\lambda,[\sigma,\sigma],\rho \rangle) \Rightarrow E^2_{*,*,*}.\]
All elements of $L/\langle \sigma,\lambda,[\sigma,\sigma],\rho \rangle$ lie in degrees $(g,p,q,r)$ with $\frac{p+q}{g}\geq \frac{3}{4}$, so we conclude that $E^2_{g, p, q}$ vanishes for $\frac{p+q}{g} < \frac{3}{4}$.
\end{proof}

It follows from the spectral sequence that $H_{g,d}(\overline{\gA}/(\sigma, \lambda))$ vanishes for $\frac{d}{g} < \frac{3}{4}$.\end{proof}

We have $H_{g,d}^{E_2}(\gA)=0$ for $d < g-1$, and by Lemma \ref{lem:VanishingRelE2HomologyMCG} we have $H^{E_2}_{g,d}(\gR_\bQ, \gA)=0$ for $d < \tfrac{3}{4}g$. We then apply Theorem~$E_k$.15.4 to $B(\overline{\gR}_\bQ;\overline{\gA};\overline{\gA}/(\sigma, \lambda)) \simeq \overline{\gR}_\bQ/(\sigma, \lambda)$ (with $k=2$, $\cat{S} = \cat{sMod}_\bQ$, $\cat{G} = \bN$, $f$ given by $c \colon \gA \to \gR_\bQ$, $\gM = \overline{\gA}/(\sigma, \lambda)$, $\rho(g) = g$ and $\mu(g) = \frac{3}{4}g$). The conclusion is that $H_{g,d}(\overline{\gR}_\bQ/(\sigma, \lambda))=0$ for $d < \tfrac{3}{4}g$, and the long exact sequence on homology for the cofibre sequence \eqref{eq:MultKappa} gives the following, which is our secondary homological stability theorem in the case of rational coefficients.

\begin{corollary}\label{cor:MCGSecStab}
The map
\[\lambda \cdot - \colon H_{d-2}(\Gamma_{g-3,1}, \Gamma_{g-4,1};\bQ) \lra H_{d}(\Gamma_{g,1}, \Gamma_{g-1,1};\bQ)\]
is an epimorphism for $4d \leq 3g-1$ and an isomorphism for $4d \leq 3g-5$.
\end{corollary}

Results of a similar spirit have been established by Hepworth \cite{hepworth} for certain sequences of groups (on the line $2d = g$ only), and by Miller--Wilson \cite{MillerWilson} for the homology of ordered configuration spaces in the context of representation stability. The following, along with Corollary \ref{cor:MCGStabRange}, gives complete information about the relative groups $H_d(\Gamma_{g,1}, \Gamma_{g-1,1};\bQ)$ in the range $3d \leq 2g$ and hence proves a rational version of Theorem \ref{thm:B} (\ref{it:B2}).

\begin{corollary}\label{cor:MCGPowersKappa}
	For all $k \geq 1$ we have $H_{2k}(\Gamma_{3k,1}, \Gamma_{3k-1,1};\bQ) = \bQ\{\lambda^k\}$.
\end{corollary}

\begin{proof}
By the corollary above, each of the maps
\[\bQ\{\lambda\} = H_2(\Gamma_{3,1}, \Gamma_{2,1};\bQ) \overset{\lambda \cdot -}\lra H_4(\Gamma_{6,1}, \Gamma_{5,1};\bQ) \overset{\lambda \cdot -}\lra \cdots \overset{\lambda \cdot -}\lra H_{2k}(\Gamma_{3k,1}, \Gamma_{3k-1,1};\bQ)\]
is surjective, as $4 \cdot 4 \leq 3 \cdot 6-1$. In order to prove the claim it is therefore enough to show that $H_{2k}(\Gamma_{3k,1}, \Gamma_{3k-1,1};\bQ) \neq 0$ for $k \geq 1$. 
	
To show this we use the Miller--Morita--Mumford classes, which will be discussed in more detail in Section \ref{sec:proofthmh43}. Morita \cite[Theorem 1.1]{MoritaRels} has shown that the Miller--Morita--Mumford class $\kappa_{\lfloor g/3\rfloor+1}$ is decomposable as a polynomial in the classes $\kappa_1, \kappa_2, \ldots, \kappa_{\lfloor g/3\rfloor}$ in $H^*(\Gamma_{g};\bQ)$, and since the Miller--Morita--Mumford classes in $H^*(\Gamma_{g,1};\bQ)$ are pulled back from $H^*(\Gamma_g;\bQ)$ the same is true for $\Gamma_{g,1}$. In the infinite genus limit the $\kappa_i$ are algebraically independent by a theorem of Miller \cite[Theorem 1.1]{Miller} (in fact, they are free generators as a consequence of the Madsen--Weiss theorem \cite[Theorem 1.1.1]{MadsenWeiss}). By the homological stability result of Corollary \ref{cor:MCGStabRange} the map
	\[\lim_{g \to \infty} H^*(\Gamma_{g,1};\bQ) \lra H^*(\Gamma_{g,1};\bQ)\]
is injective in degrees $* \leq \tfrac{2g+1}{3}$, so they are also algebraically independent in $H^*(\Gamma_{g,1};\bQ)$ for $* \leq \tfrac{2g+1}{3}$. 
	
	In particular, $\kappa_{k} \in H^{2k}(\Gamma_{3k-1,1};\bQ)$ is decomposable as a polynomial in the classes $\kappa_1, \ldots, \kappa_{k-1}$ while $\kappa_{k} \in H^{2k}(\Gamma_{3k,1};\bQ)$ is not. Thus there is a non-zero polynomial in $\kappa_1, \ldots, \kappa_{k}$ which gives rise to a nontrivial element in the kernel of
	\[H^{2k}(\Gamma_{3k,1};\bQ) \lra H^{2k}(\Gamma_{3k-1,1};\bQ),\]
	so $H^{2k}(\Gamma_{3k,1},\Gamma_{3k-1,1};\bQ) \neq 0$, hence $H_{2k}(\Gamma_{3k,1}, \Gamma_{3k-1,1};\bQ) \neq 0$ as required.
\end{proof}

\begin{remark}
One can extract from Corollary \ref{cor:MCGSecStab} a statement about the kernels and cokernels
	\begin{align*}\mr{ker}_{g,d}(\sigma) &\coloneqq \mr{ker}\left(\sigma \colon H_d(\Gamma_{g,1};\bQ) \to H_d(\Gamma_{g+1,1};\bQ)\right)\\
	\mr{coker}_{g,d}(\sigma) &\coloneqq \mr{coker}\left(\sigma \colon H_d(\Gamma_{g-1,1};\bQ) \to H_d(\Gamma_{g,1};\bQ)\right)
\end{align*}
of the stabilisation map.

Since $\gR_\bQ$ is an $E_2$-algebra, and hence homotopy commutative, the maps $\phi(\sigma)$ and $\phi(\lambda)$ commute up to homotopy and hence give a map of cofibre sequences
\[\begin{tikzcd} 
S^{4,2} \otimes \overline{\gR}_\bQ \rar{S^{3,2} \otimes \phi(\sigma)} \dar{\phi(\lambda)} &[5pt] S^{3,2} \otimes \overline{\gR}_\bQ \dar{\phi(\lambda)} \rar & S^{3,2} \otimes \overline{\gR}_\bQ/\sigma \dar{\phi(\lambda)} \rar{\partial} & S^{4,3} \otimes \overline{\gR}_\bQ \dar{\phi(\lambda)} \\
S^{1,0} \otimes \overline{\gR}_\bQ \rar{\phi(\sigma)}  & \overline{\gR}_\bQ \rar & \overline{\gR}_\bQ/\sigma \rar{\partial} & S^{1,1} \otimes \overline{\gR}_\bQ.
\end{tikzcd}\]
From the induced map of long exact sequences we can therefore extract the following map of short exact sequences:
\[\begin{tikzcd}0 \rar  &[-10pt] \mr{coker}_{g-3,d-2}(\sigma) \rar \dar{\lambda \cdot -} & H_{g-3,d-2}(\gR_\bQ/\sigma) \dar{\lambda \cdot -} \rar & \mr{ker}_{g-4,d-3}(\sigma) \dar{\lambda \cdot -} \rar &[-10pt] 0 \\
0 \rar  & \mr{coker}_{g,d}(\sigma) \rar & H_{g,d}(\gR_\bQ/\sigma) \rar & \mr{ker}_{g-1,d-1}(\sigma) \rar & 0.\end{tikzcd}\]
When the middle map is an isomorphism (i.e.~for $4d \leq 3g-5$, or looking forward to Corollary \ref{cor:MCGSecStabImprov} when $5d \leq 4g-6$) it follows that the left vertical map is injective and the right vertical map is surjective. Furthermore, the snake lemma identifies kernel of the latter with the the cokernel of the former (via the Massey product $x \mapsto \langle \sigma, x, \lambda \rangle$), giving an exact sequence
	\[\begin{tikzcd}
	& 0 \rar & \mr{coker}_{g-3,d-2}(\sigma) \rar{\lambda \cdot -} \ar[draw=none]{d}[name=X, anchor=center]{} & \mr{coker}_{g,d}(\sigma)
	\ar[rounded corners,
	to path={ -- ([xshift=2ex]\tikztostart.east)
		|- (X.center) \tikztonodes
		-| ([xshift=-2ex]\tikztotarget.west)
		-- (\tikztotarget)}]{dll} \\
	& \mr{ker}_{g-4,d-3}(\sigma) \rar{\lambda \cdot -} & \mr{ker}_{g-1,d-1}(\sigma) \rar & 0.
\end{tikzcd}\]

Using Corollary \ref{cor:MCGSecStabIntegral} this argument goes through with $\bQ$ replaced by $\bZ[\frac{1}{10}]$, in a range of degrees.
\end{remark}

\subsection{Integral secondary homological stability}
\label{sec:integr-second-homol} 

We now want to show that the previous results hold with integral coefficients, and our first goal will be to define secondary stabilisation maps. Recall from Lemma \ref{lem:MCGFacts} (\ref{it:MCGFacts42}) that there is a class  $\lambda \in H_2(\Gamma_{3,1};\bZ) = H_{3,2}(\gR_\bZ)$, well-defined modulo classes in the image of the stabilisation map
\[\sigma \cdot - \colon H_2(\Gamma_{2,1};\bZ) \lra H_2(\Gamma_{3,1};\bZ).\]
As the domain of this map is $\bZ/2$, by Lemma \ref{lem:MCGFacts} (\ref{it:MCGFactsg2}), the class $\lambda$ is well-defined after inverting 2; after inverting $10$ it agrees with the class we called $\lambda$ in Lemma \ref{lem:MCGFacts} (\ref{it:MCGFacts32abs}).

%The difficulty with following the argument in Section \ref{sec:RatStab} is that $H_{3,2}(\overline{\gR}_\bZ/\sigma) = \bZ\{\mu\}$ by Lemma \ref{lem:MCGFacts} (\ref{it:MCGFacts32rel}) but the class $\mu$ itself does not come from the ring $H_{*,*}(\overline{\gR}_\bZ)$, only $10\mu$ does. Thus the map we would like to show is an isomorphism ought to be ``multiplication by $\mu$" acting on $H_{*,*}(\overline{\gR}_\bZ/\sigma)$, but this does not make sense as $\overline{\gR}_\bZ/\sigma$ is not a ring and $\mu$ does not lift to an element of  $H_{3,2}(\overline{\gR}_\bZ)$.

\subsubsection{Constructing the integral secondary stabilisation map}\label{sec:ConstSecStabMap}

Let us first explain why we cannot simply proceed as in Section \ref{sec:RatStab}. If we consider multiplication by the class $\lambda$ described above (or rather, a choice of this class up to the indicated $\bZ/2$-ambiguity) then the map $\phi(\lambda) \colon S^{3,2} \otimes \overline{\gR}_\bZ/\sigma \lra \overline{\gR}_\bZ/\sigma$ induces a map
\[\lambda \cdot - \colon \bZ\{1\} = H_{0,0}(\overline{\gR}_\bZ/\sigma) \lra H_{3,2}(\overline{\gR}_\bZ/\sigma) = \bZ\{\mu\}\]
sending $1$ to $\lambda \cdot 1 = 10 \mu$. This is not an epimorphism, and so \emph{the statement of Theorem \ref{thm:A} would be false with this choice of secondary stabilisation map}.

To solve this problem we would like to say that the secondary stabilisation map is ``multiplication by $\mu$" acting on $H_{*,*}(\overline{\gR}_\bZ/\sigma)$, but this does not make sense as $\overline{\gR}_\bZ/\sigma$ is not itself a ring (it is a quotient left $\overline{\gR}_\bZ$-module of $\overline{\gR}_\bZ$, so should be thought of as dividing out the left ideal generated by $\sigma$ but not the two-sided ideal) and $\mu$ does not lift to an element of $H_{3,2}(\overline{\gR}_\bZ)$ (only $\lambda = 10\mu$ does). To circumvent this difficulty, we will construct the secondary stabilisation map
\[\varphi_* \colon H_{d-2}(\Gamma_{g-3,1},\Gamma_{g-4,1};\bZ) \lra H_{d-2}(\Gamma_{g-1,1},\Gamma_{g-,1};\bZ)\] from an endomorphism of the left $\overline{\gR}_\bZ$-module $\overline{\gR}_\bZ/\sigma$ which does \emph{not} come from multiplying on the left by an element of $H_{3,2}(\overline{\gR}_\bZ)$. Consider the diagram
\begin{equation*}
\begin{tikzcd}
S^{3,2}_\bZ \otimes S^{1,0}_\bZ \otimes \overline{\gR}_\bZ \rar{S^{3,2}_\bZ \otimes \phi(\sigma) }&[5pt]  S^{3,2}_\bZ \otimes  \overline{\gR}_\bZ \arrow{rr} \dar{\mu \otimes \overline{\gR}_\bZ} &[5pt] &[-5pt] S^{3,2}_\bZ \otimes  \overline{\gR}_\bZ/\sigma  \dar[dashed] \\
&\overline{\gR}_\bZ/\sigma \otimes  \overline{\gR}_\bZ \rar{\cong}[swap]{\beta_{\overline{\gR}_\bZ/\sigma, \overline{\gR}_\bZ}} & \overline{\gR}_\bZ \otimes \overline{\gR}_\bZ/\sigma  \rar{a} & \overline{\gR}_\bZ/\sigma
\end{tikzcd}
\end{equation*}
in the category of left $\overline{\gR}_\bZ$-modules, where the top row is obtained by applying $S^{3,2}_\bZ \otimes -$ to the cofibration sequence \eqref{eq:RModSigmaCofSeq}, and the bottom map is (transposing the factors followed by) the left $\overline{\gR}_\bZ$-module structure map of $\overline{\gR}_\bZ/\sigma$. By the long exact sequence associated to a cofibre sequence, the dashed arrow in the diagram may be filled in with an $\overline{\gR}_\bZ$-module map making the square commute up to homotopy if and only if the composition 
\[a \circ \beta_{\overline{\gR}_\bZ/\sigma, \overline{\gR}_\bZ} \circ (\mu \otimes \overline{\gR}_\bZ) \circ (S^{3,2}_\bZ \otimes \phi(\sigma)) \colon  S^{3,2}_\bZ \otimes S^{1,0}_\bZ \otimes \overline{\gR}_\bZ \lra \overline{\gR}_\bZ/\sigma\]
is null as a left $\overline{\gR}_\bZ$-module map, which happens if and only if the restriction to $S^{3,2}_\bZ \otimes S^{1,0}_\bZ \subset  S^{3,2}_\bZ \otimes S^{1,0}_\bZ \otimes \overline{\gR}_\bZ $ is null as a map. But by Corollary \ref{cor:MCGStabRange} we have $H_{4,2}(\overline{\gR}_\bZ/\sigma)=H_2(\Gamma_{4,1},\Gamma_{3,1};\bZ) = 0$, so  all such maps are null and hence the dashed arrow exists. Let us write
\[\varphi \colon S^{3,2}_\bZ \otimes \overline{\gR}_\bZ/\sigma \lra \overline{\gR}_\bZ/\sigma\]
for a choice of such a dashed map.   

As always in obstruction theory, the set of possible extensions is a torsor, in this case for $H_{4,3}(\overline{\gR}_\bZ/\sigma) = H_3(\Gamma_{4,1}, \Gamma_{3,1};\bZ)$. This group is non-zero because $H_2(\Gamma_{3,1};\bZ) \to H_2(\Gamma_{4,1};\bZ)$ has kernel $\bZ/2\bZ$ by Remark \ref{rem:mcg-comp}. This failure of uniqueness for $\varphi$ is not a problem, as the argument we will give holds for \emph{any} choice of extension $\varphi$. 

The problem as we have presented it has to do with divisibility by 10, and so it is not surprising that it can be circumvented when working with $\bZ[\tfrac{1}{10}]$-coefficients. We record this observation for later use. See Sections $E_k$.12.2.2 and $E_k$.12.2.3 for adapters and their applications.

\begin{lemma}\label{lem:LambdaOverTen}
With $\bZ[\tfrac{1}{10}]$-coefficients, one choice of $\varphi$ is given by multiplication by $\tfrac{1}{10}\lambda \in H_{3,2}(\overline{\gR}_{\bZ[\frac{1}{10}]})$ using the adapter.
\end{lemma}

\begin{proof}
Multiplication by $\tfrac{1}{10}\lambda$ using the adapter gives the bottom row in the following homotopy commutative diagram in the category of left $\overline{\gR}_{\bZ[\frac{1}{10}]}$-modules
\begin{equation*}
\begin{tikzcd}
S^{3,2}_\bZ \otimes \overline{\gR}_{\bZ[\frac{1}{10}]} \dar \rar{\tfrac{1}{10}\lambda \otimes \mathrm{Id}} & \overline{\gR}_{\bZ[\frac{1}{10}]} \otimes \overline{\gR}_{\bZ[\frac{1}{10}]} \rar & \overline{\gR}_{\bZ[\frac{1}{10}]} \dar\\
S^{3,2}_\bZ \otimes \overline{\gR}_{\bZ[\frac{1}{10}]}/\sigma \rar{\tfrac{1}{10}\lambda \otimes \mathrm{Id}} & \overline{\gR}_{\bZ[\frac{1}{10}]} \otimes \overline{\gR}_{\bZ[\frac{1}{10}]}/\sigma \rar & \overline{\gR}_{\bZ[\frac{1}{10}]}/\sigma,
\end{tikzcd}
\end{equation*}
where the two rightmost horizontal maps are given by the special $\overline{\gR}_{\bZ[\frac{1}{10}]}$-module structure produced using the adapter. Thus the restriction of the bottom map to $\overline{\gR}_{\bZ[\frac{1}{10}]} \subset \overline{\gR}_{\bZ[\frac{1}{10}]}/\sigma$ is the unique homotopy class of $\overline{\gR}_{\bZ[\frac{1}{10}]}$-module map sending the unit to
\[S^{3,2}_\bZ \overset{\tfrac{1}{10}\lambda}\lra \overline{\gR}_{\bZ[\frac{1}{10}]} \lra \overline{\gR}_{\bZ[\frac{1}{10}]}/\sigma.\]
But this is the class $\mu$, so the lower map extends the map $a \circ b_{\overline{\gR}_\bZ/\sigma, \overline{\gR}_\bZ} \circ (\mu \otimes \overline{\gR}_\bZ)$ and hence is a choice of $\varphi$.
\end{proof}

\subsubsection{Proof of integral secondary homological stability}

The main result of this section is that each of the maps $\varphi$ has trivial relative homology in the same range of degrees as we showed $\lambda \cdot -$ does rationally in Section \ref{sec:RatStab}. The proof starts as in Section \ref{sec:RatStab} by building a cellular model for $\gR_\bZ$ in low degrees, similar to that in the previous section but more complex. Just as we had to a pick a representative of the Browder bracket, we need to pick a point in $\cC_2(2)$ to get a representative of the product operation and a map $S^1 \to \cC_2(2)/\fS_2$ to get a representative of $Q^1_\bZ$. Using these we obtain maps from $S^{2,1}_\bZ$, $S^{2,1}_\bZ$, and $S^{3,1}_\bZ$ into $\gE_2(S_\bZ^{1,0}\sigma \oplus S_\bZ^{1,1}\tau)$ representing $10 \sigma \tau$, $Q^1_\bZ(\sigma) - 3\sigma \tau$ and $\sigma^2 \tau$ respectively.  Let us define an $E_2$-algebra $\gA$ by
\begin{equation*}
\gA \coloneqq \gE_2(S_\bZ^{1,0}\sigma \oplus S_\bZ^{1,1}\tau) \cup^{E_2}_{10\sigma\tau} \gD_\bZ^{2,2}\rho_1 \cup^{E_2}_{Q^1_\bZ(\sigma)-3\sigma\tau} \gD_\bZ^{2,2}\rho_2 \cup^{E_2}_{\sigma^2\tau} \gD_\bZ^{3,2} \rho_3.
\end{equation*}

Let us next choose maps $\sigma \colon S^{1,0}_\bZ \to \gR_\bZ$ and $\tau \colon S^{1,1}_\bZ \to \gR_\bZ$ representing $\sigma \in H_{1,0}(\gR_\bZ)$ and $\tau \in H_{1,1}(\gR_\bZ)$. In $H_{*,*}(\gR_\bZ)$ we have equations $10\sigma \tau =0$, $Q^1_\bZ(\sigma) = 3\sigma \tau$, and $\sigma^2 \tau=0$, and we may pick homotopies $\rho_1$, $\rho_2$ and $\rho_3$ imposing these (which are not unique). Using this data we may define an $E_2$-map
\[c \colon \gA \lra \gR_\bZ,\]
sending the formal generator $\sigma$ of $\gA$ to $\gR_\bZ$ using $\sigma \colon S^{1,0}_\bZ \to \gR_\bZ$, etc. By Cohen's calculations we have
\[H_{g,1}(\gE_2(S_\bZ^{1,0}\sigma \oplus S_\bZ^{1,1}\tau)) = \begin{cases}
0 & \text{if $g=0$,} \\
\bZ\{\tau\} & \text{if $g=1$,}\\
\bZ\{\sigma^{g-1}  \tau, \sigma^{g-2}  Q^1_\bZ(\sigma) \} & \text{if $g \geq 2$,}
\end{cases}\]
and attaching the cells $\rho_i$ gives, by an easy computation with the skeletal spectral sequence of Corollary $E_k$.10.19,
\[H_{g,1}(\gA) = \begin{cases}
\bZ\{\tau\} & \text{if $g=1$,}\\
\bZ/10\{\sigma  \tau\} & \text{if $g = 2$,}\\
0 & \text{if $g\neq 1,2$.}
\end{cases}\]
%This is an easy computation with the skeletal filtration spectral sequence of Section $E_k$.10.3.3.

The map $c$ is not uniquely defined: we may change it on the cells $\rho_i$ by adding on any homology class of $\gR_\bZ$ of the appropriate bidegree. Considering the induced map of long exact sequences
\begin{equation*}
\begin{tikzcd}
H_{3,2}(\overline{\gA}) \dar{c_*} \rar & H_{3,2}(\overline{\gA}/\sigma) \dar{c_*} \rar{\partial} & H_{2,1}(\overline{\gA}) = \bZ/10\{\sigma \tau\} \dar{\simeq} \rar& 0\\
H_{3,2}(\overline{\gR}_\bZ) \rar & H_{3,2}(\overline{\gR}_\bZ/\sigma) = \bZ\{\mu\} \rar{\partial} & H_{2,1}(\overline{\gR}_\bZ) = \bZ/10\{\sigma \tau\} \rar& 0
\end{tikzcd}
\end{equation*}
the class $\rho_3 \in H_{3,2}(\overline{\gA}/\sigma)$ satisfies $\partial(\rho_3) = \sigma \tau$, so that $c_*(\rho_3) = (1 + 10 \cdot N) \mu \in H_{3,2}(\overline{\gR}_\bZ/\sigma)$. As $10 \mu$ lifts to the class $\lambda \in H_{3,2}(\overline{\gR}_\bZ)$, we have $c_*(\rho_3) = \mu + N \cdot \lambda$. Therefore after re-choosing the map $c$ by changing it by $-N \cdot \lambda \in H_{3,2}(\overline{\gR}_\bZ)$ on the cell $\rho_3$ we may suppose that the middle map satisfies $c_*(\rho_3)=\mu$ and so in particular is surjective. This will be used in the following lemma: it is at this point that we have used the rather subtle calculation of Lemma \ref{lem:MCGFacts} (\ref{it:MCGFacts32rel}).

\begin{lemma}\label{lem:IntVanishingRelE2HomologyMCG}$H^{E_2}_{g,d}(\gR_\bZ, \gA)=0$ for $4d \leq 3g-1$.\end{lemma}

\begin{proof}As in the proof of Lemma \ref{lem:VanishingRelE2HomologyMCG}, to show that $H^{E_2}_{g,d}(\gR_\bZ, \gA)=0$ for $4d \leq 3g-1$ it is enough to show vanishing in bidegrees $(1,0)$, $(2,1)$, and $(3,2)$. The first two cases go through as in the proof of that lemma, and it remains to show that $H^{E_2}_{3,2}(\gR_\bZ, \gA)=0$. We have that $H_{g,1}(\gR_\bZ, \gA)=0$ for all $g$, by Lemma \ref{lem:MCGFacts} (\ref{it:MCGFacts11}), (\ref{it:MCGFacts21}), and (\ref{it:MCGFactsg1}), so the map $c$ is $(2,2,2,\ldots)$-connective. The objects $\gA$ and $\gR_\bZ$ are 0-connective, so it follows from Proposition $E_k$.11.9 that the map of homotopy cofibres
	\[\gR_\bZ/\gA \lra Q_\bL^{E_2}(\gR_\bZ)/Q_\bL^{E_2}(\gA)\]
	is $(3,3,3,\ldots)$-connective. Thus the Hurewicz map $H_{3,2}(\gR_\bZ, \gA) \to H_{3,2}^{E_2}(\gR_\bZ, \gA)$ for $E_2$-homology is surjective. Note that the composition
	\[S^{1,0}_\bZ \otimes \gR_\bZ \overset{\sigma \cdot -}\lra \gR_\bZ \lra Q^{E_2}_\bL(\gR_\bZ)\]
	is canonically null-homotopic, as multiplying by $\sigma$ lands in the decomposables, by definition.  Therefore there is a factorisation $\gR_\bZ \to \gK \to Q^{E_2}_\bL(\gR_\bZ)$ of the Hurewicz map over the homotopy cofibre
        $\gK$ of $\sigma \cdot - \colon S^{1,0}_\bZ \otimes \gR_\bZ \to \gR_\bZ$. As there is a weak equivalence
	\[\gK(g) \simeq (\overline{\gR}_\bZ/\sigma)(g)\]
	for $g > 0$, in these degrees there is a factorisation of the induced map
	\[H_{g,d}(\gR_\bZ) \lra H_{g,d}(\overline{\gR}_\bZ/\sigma) \lra H^{E_2}_{g,d}(\gR_\bZ)\]
	on homology; similarly for $\gA$. The (factorised) Hurewicz maps give a diagram
	\begin{equation*}
	\begin{tikzcd}
	H_{3,2}(\gA) \dar \rar{c_*}&[-5pt] H_{3,2}(\gR_\bZ) \dar \rar[two heads]&[-8pt] H_{3,2}(\gR_\bZ, \gA) \arrow[two heads]{dd} \rar{0}&[-5pt] H_{2,1}(\gA) \arrow{dd} \rar{c_*}[swap]{\sim}&[-5pt] H_{2,1}(\gR_\bZ)\\
	H_{3,2}(\overline{\gA}/\sigma) \dar \rar[two heads]{c_*}& H_{3,2}(\overline{\gR}_\bZ/\sigma) \arrow[dashed]{rd} \dar \\
	H^{E_2}_{3,2}(\gA) \rar{c_*}& H^{E_2}_{3,2}(\gR_\bZ) \rar& H^{E_2}_{3,2}(\gR_\bZ, \gA) \rar& H^{E_2}_{2,1}(\gA) =0
	\end{tikzcd}
	\end{equation*}
	and so the dashed composition 
	\[\bZ\{\mu\} = H_{3,2}(\overline{\gR}_\bZ/\sigma) \lra H^{E_2}_{3,2}(\gR_\bZ) \lra H^{E_2}_{3,2}(\gR_\bZ, \gA)\]
	is surjective. But $c_* \colon H_{3,2}(\overline{\gA}/\sigma) \to H_{3,2}(\overline{\gR}_\bZ/\sigma)$ is surjective too, as we arranged that $c_*(\rho_3)=\mu$, and after a diagram chase one concludes that $\smash{H^{E_2}_{3,2}}(\gR_\bZ, \gA)=0$.\end{proof}

In Corollary \ref{cor:MCGSecStab} we proved that the homology of the cofibre of $\lambda \cdot -  \colon S^{3,2}_{\mathbb{Q}} \otimes \overline{\mathbf{R}}_\mathbb{Q}/\sigma \to \overline{\mathbf{R}}_\mathbb{Q}/\sigma$ vanishes in a range, and the analogous statement here is that the homology of the cofibre $\gC_\varphi$ of $\varphi \colon S^{3,2}_\bZ \otimes \overline{\gR}_\bZ/\sigma \to\overline{\gR}_\bZ/\sigma$ vanishes in a range. At this point our strategy deviates from the rational case because, although $\gA$ contains all the cells in bidegree $(d,g)$ satisfying $\frac{d}{g}<\frac{3}{4}$, this is not enough to have the same obstruction theory argument go through as for $\gR_\bZ$. To remedy this, we add enough cells to get a CW $E_2$-algebra equivalent to $\gR_\bZ$; this is the $E_2$-algebra $\gB$ we construct now. Our knowledge of the connectivity of $\gA \to \gR_\bZ$ on $E_2$-homology will allow us to constrain the bidegrees of cells added to $\gA$ to produce $\gB$. (We could have given an argument in the rational case along the same lines, which amounts to replacing the appeal to Theorem $E_k$.15.4 to the calculational proof of that theorem for discrete groupoids $\cat{G}$ referred to in Remark $E_k$.15.2.)

By Theorem $E_k$.11.21 we may find a factorisation 
\[c \colon \gA \lra \gB \overset{\sim}\lra \gR_{\bZ}\]
where $\gB$ is a CW $E_2$-algebra obtained by attaching cells $\{\gD^{g_\alpha, d_\alpha}x_\alpha\}_{\alpha \in I}$ to $\gA$ with $4d_\alpha \geq 3g_\alpha$. Replacing $\gR_\bZ$ by the weakly equivalent $\gB$, the construction of $\varphi$ goes through the same way, giving morphisms
\[\psi \colon S^{3,2}_\bZ \otimes \overline{\gB}/\sigma \lra \overline{\gB}/\sigma\]
of $\overline{\gB}$-modules which may perhaps be regarded as a sort of ``multiplication by $\rho_3$.''  Writing  $\gC_\psi$ for its cofibre, the weak equivalence $\gB \overset{\sim}\to \gR_\bZ$ gives without loss of generality an equivalence $\gC_\psi \simeq \gC_\varphi$. Thus to prove our integral secondary stability result, it will be enough to show that $H_{g,d}(\gC_\psi)=0$ for $4d \leq 3g-1$.

\begin{theorem}\label{thm:IntStabMCG}
Any choice of $\varphi$ has cofibre $\gC_\varphi$ satisfying $H_{g,d}(\gC_\varphi)=0$ for $4d \leq 3g-1$.
\end{theorem}
\begin{proof}

Let $\mr{sk}(\gB)$ denote the skeletal filtration of $\gB$, with associated graded $\mathrm{gr}(\mr{sk}(\gB))$.  This acquires an extra grading, and we shall denote tridegrees as $(g,d,r)$, where as before $g$ is the genus and $d$ is homological degree, and $r$ is the new grading arising from passing to associated graded in the skeletal filtration.  Then the associated graded $\mathrm{gr}(\mr{sk}(\gB))$ is given by
\[\gE_2\left(S^{1,0,0}_\bZ\sigma \oplus S^{1,1,1}_\bZ\tau \oplus S^{2,2,2}_\bZ\rho_1 \oplus S^{2,2,2}_\bZ\rho_2 \oplus S^{3,2,2}_\bZ\rho_3 \oplus \bigoplus_{\alpha \in I} S^{g_\alpha, d_\alpha,d_\alpha}_\bZ x_\alpha\right)\]
and similarly for $\mathrm{gr}(\overline{\mr{sk}(\gB)}/\sigma)$.  Below we will show that the morphism $\psi$ may be lifted to a morphism in the filtered category in a certain way. In principle there is an obstruction to doing this, which will be avoided using the following calculation.

%As we shall see below, the morphism $\psi$ may be lifted to a map in the filtered category.\mnote{orw: why is this here, it is not relevant until after the claim}

\vspace{1ex}

\noindent\textbf{Claim}: $H_{4,3, q}(\mathrm{gr}(\overline{\mr{sk}(\gB)}/\sigma))=0$ for $q\geq 4$.
\begin{proof}[Proof of claim]
  Since $\gB$, being equivalent to $\gR_\bZ$, is of finite type, it is enough to verify this for the base change $\gB_{\bF_\ell} \coloneqq \gB \otimes_\bZ \bF_\ell$ for any prime number $\ell$, and it follows from the construction that $\overline{\mr{sk}(\gB)_{\bF_\ell}} \simeq \overline{\mr{sk}(\gB)}_{\bF_\ell}$. By Cohen's calculations (see \cite[ p.207ff]{CLM} and the summary in Section $E_k$.16), $H_{*,*, *}(\mathrm{gr}(\overline{\mr{sk}(\gB)}_{\bF_\ell}))$ is a free graded commutative algebra on generators of the form $Q^I_\ell(y)$, where $y$ is a basic Lie word in $\{\sigma, \tau, \rho_i, x_\alpha\}$ and $Q^I_\ell$ runs through Dyer--Lashof monomials satisfying the usual admissibility and excess conditions.  Hence the homology of the cofibre can be identified, as a $H_{*,*, *}(\mathrm{gr}(\overline{\mr{sk}(\gB)}_{\bF_\ell}))$-module, as
\[H_{*,*, *}(\mathrm{gr}(\overline{\mr{sk}(\gB)}_{\bF_\ell}/\sigma)) \cong H_{*,*, *}(\mathrm{gr}(\overline{\mr{sk}(\gB)}_{\bF_\ell}))/(\sigma),\]
which admits the structure of a free graded-commutative algebra on the same set of generators except for $\sigma = Q^\varnothing_\ell \sigma$.

The generators $\{\sigma, \tau, \rho_i, x_\alpha\}$ each have homological dimension equal to their filtration grading, so applying the operations $[-,-]$ and $Q^I_\ell(-)$ gives classes which have dimension at least their grading. Thus $H_{4,3, q}(\mathrm{gr}(\overline{\mr{sk}(\gB)}_{\bF_\ell}/\sigma))$ can be non-zero only for $q \leq 3$.
\end{proof}

The map $\sigma \colon S^{1,0}_\bZ \to \gA \subset \gB$ lifts to the 0-skeleton $(\mr{sk}(\gB))(0)$, so by adjunction determines a filtered map $\sigma \colon 0_*(S^{1,0}_\bZ) \to \mr{sk}(\gB)$. Using this we can form $\overline{\mr{sk}(\gB)}/\sigma$. There is a homotopy cofibre sequence
\[(0_*(S^{1,0}_\bZ) \otimes \overline{\mr{sk}(\gB)})(2)\lra (\overline{\mr{sk}(\gB)})(2)\lra (\overline{\mr{sk}(\gB)}/\sigma)(2),\]
and the attaching map $\sigma^2 \tau$ for the $(3,2)$-cell $\rho_3$ lifts to $(0_*(S^{1,0}_\bZ) \otimes \overline{\mr{sk}(\gB)})(2)$, giving a map $\rho_3 \colon S^{3,2}_\bZ \to (\overline{\mr{sk}(\gB)}/\sigma)(2)$, so by adjunction a filtered map $\rho_3 \colon 2_*(S^{3,2}_\bZ) \to \overline{\mr{sk}(\gB)}/\sigma$. We may consider this as a filtered map $\rho_3 \colon 3_*(S^{3,2}_\bZ) \to \overline{\mr{sk}(\gB)}/\sigma$, by neglect of structure.

\vspace{1ex}

\noindent\textbf{Claim}: The composition
\[3_*(S^{4,2}_\bZ) = 0_*(S^{1,0}_\bZ) \otimes 3_*(S^{3,2}_\bZ)   \xrightarrow{\sigma \otimes \rho_3} \overline{\mr{sk}(\gB)} \otimes (\overline{\mr{sk}(\gB)}/\sigma) \lra \overline{\mr{sk}(\gB)}/\sigma\]
is null-homotopic.
\begin{proof}[Proof of claim]
This map is adjoint to a map $u \colon S^{4,2}_\bZ \to (\overline{\mr{sk}(\gB)}/\sigma)(3)$. There are homotopy cofibre sequences
\[(\overline{\mr{sk}(\gB)}/\sigma)(q-1) \lra (\overline{\mr{sk}(\gB)}/\sigma)(q) \lra \mathrm{gr}(\overline{\mr{sk}(\gB)}/\sigma)(q)\]
and as $H_{4,3, q}(\mathrm{gr}(\overline{\mr{sk}(\gB)}/\sigma))=0$ for $q\geq 4$ by our calculation above, the maps
\[(\overline{\mr{sk}(\gB)}/\sigma)(3) \lra (\overline{\mr{sk}(\gB)}/\sigma)(4) \lra (\overline{\mr{sk}(\gB)}/\sigma)(5) \lra (\overline{\mr{sk}(\gB)}/\sigma)(6) \lra \cdots\]
are all injective on $H_{4,2}(-)$. As the colimit $H_{4,2}(\gB/\sigma)\cong H_{4,2}(\gR_\bZ/\sigma)$ vanishes as a consequence of Corollary \ref{cor:MCGStabRange}, it follows that $H_{4,2}((\overline{\mr{sk}(\gB)}/\sigma)(3))=0$ and so the map $u$ is null-homotopic.
\end{proof}

Using this claim, the obstruction theory argument (in the category of $\overline{\mr{sk}(\gB)}$-modules) gives a $\overline{\mr{sk}(\gB)}$-module map in the filtered category
\[\mr{sk}(\psi) \colon 3_*(S^{3,2}_\bZ) \otimes (\overline{\mr{sk}(\gB)}/\sigma) \lra \overline{\mr{sk}(\gB)}/\sigma\]
which is again a ``multiplication by $\rho_3$'', but now in the filtered category.

Recall that our previous definition of $\psi$ only pinned it down up to a torsor for the unknown group $H_{4,3}(\gB/\sigma)$.  It is clear that the obstruction theoretic construction of $\mr{sk}(\psi)$ is a refinement of the one for $\psi$, so passing to underlying ungraded maps (i.e.\ taking colimit over the filtration) will send any $\mr{sk}(\psi)$ to some $\psi$.  In fact, any $\psi$ arises this way:

\vspace{1ex}

\noindent\textbf{Claim}: 
Any of the ($H_{4,3}(\gB/\sigma)$-torsor worth of) maps $\psi$ arise by taking underlying unfiltered map from some $\mr{sk}(\psi)$ arising as above.
\begin{proof}[Proof of claim]
 The claim is simply that the indeterminacy in the filtered obstruction problem is at least as large as that in the original obstruction problem: in particular, admitting a lift to a filtered map does not further pin down $\psi$.

The maps
\[H_{4,3}((\overline{\mr{sk}(\gB)}/\sigma)(3)) \lra H_{4,3}((\overline{\mr{sk}(\gB)}/\sigma)(4)) \lra H_{4,3}((\overline{\mr{sk}(\gB)}/\sigma)(5)) \lra \cdots\]
are all surjective by the first claim above, so $H_{4,3}((\overline{\mr{sk}(\gB)}/\sigma)(3)) \to H_{4,3}(\overline{\gB}/\sigma)$ is surjective.
\end{proof}

We have now lifted the previous constructions of $\psi \colon S^{3,2} \otimes \overline{\gB}/\sigma \to \overline{\gB}/\sigma$ to the filtered category.  Returning to the proof of Theorem~\ref{thm:IntStabMCG}, we must again do some low-dimensional calculations.

The cofibre $\gC_{\mr{sk}(\psi)}$ of $\mr{sk}(\psi)$ has underlying unfiltered object $\gC_\psi$. As the map $\rho_3 \colon 3_*(S^{3,2}_\bZ) \to \overline{\mr{sk}(\gB)}/\sigma$ used to construct $\mr{sk}(\psi)$ was obtained by neglect of structure from a map from $2_*(S^{3,2}_\bZ)$, the map $\mr{sk}(\psi)$ strictly decreases filtration and so it is trivial on associated graded. Since $\grr$ commutes with cofibres the associated graded of $\gC_{\mr{sk}(\psi)}$ is given by $(S^{0,0,0} \oplus S^{3,3,3}\rho_4) \otimes \grr(\overline{\mr{sk}(\gB)})/\sigma$, with $\rho_4$ denoting the map $S^{3,3,3} \to \grr(\gC_{\mr{sk}(\psi)})$ obtained from the trivial map 
\[S^{3,2,3} = S^{3,2,3} \otimes \bunit \lra \grr(3_*(S^{3,2}_\bZ) \otimes (\overline{\mr{sk}(\gB)}/\sigma)) \xrightarrow{\grr(\mr{sk}(\psi))} \grr(\overline{\mr{sk}(\gB)}/\sigma)\]
upon taking cofibre. We further have that $\grr(\overline{\mr{sk}(\gB)})/\sigma$ is 
\[\overline{\gE_2(S^{1,0,0}_\bZ\sigma \oplus S^{1,1,1}_\bZ\tau \oplus S^{2,2,2}_\bZ\rho_1 \oplus S^{2,2,2}_\bZ\rho_2 \oplus S^{3,2,2}_\bZ\rho_3 \oplus \bigoplus_{\alpha \in I} S^{g_\alpha, d_\alpha,d_\alpha}_\bZ )}/\sigma,\]
where as before we have $\frac{d_\alpha}{g_\alpha} \geq \frac{3}{4}$ for all $\alpha \in I$.

The associated spectral sequence 
\[E^1_{*,*,*} = H_{*,*,*}\left((S^{0,0,0}_\bZ \oplus S^{3,3,3}_\bZ\rho_4) \otimes \grr(\overline{\mr{sk}(\gB)}/\sigma)\right) \Longrightarrow H_{*,*}(\gC_\psi)\]
has $d^1(\rho_4) = \rho_3$, $d^1(\rho_1) = 0$, and $d^1(\rho_2) = Q^1_\bZ(\sigma)$ (the $d^1$-differential vanishes on $\rho_3$ since taking the quotient by $\sigma$ made its attaching map null-homotopic). This is a module spectral sequence over that for $\overline{\mr{sk}(\gB)}$, having $F^1_{*,*,*}$ given by
\[H_{*,*, *}\left(\overline{\gE_2(S^{1,0,0}_\bZ\sigma \oplus S^{1,1,1}_\bZ\tau \oplus S^{2,2,2}_\bZ\rho_1 \oplus S^{2,2,2}_\bZ\rho_2 \oplus S^{3,2,2}_\bZ\rho_3 \oplus \bigoplus_{\alpha \in I} S^{g_\alpha, d_\alpha,d_\alpha}_\bZ )}\right).\]
There is a similar spectral sequence after applying base change
$- \otimes_{\bZ} \bF_\ell$, and we will show a vanishing line for the
$E^2$-page of this spectral sequence for every prime number $\ell$:
\begin{equation}\label{eqn:ss-cpsi-ell} E^{1,\ell}_{*,*,*} = H_{*,*,*}\left((S^{0,0,0}_{\bF_\ell} \oplus S^{3,3,3}_{\bF_\ell}\rho_4) \otimes \grr(\overline{\mr{sk}(\gB)}_{\bF_\ell})/\sigma\right) \Longrightarrow H_{*,*}((\gC_\psi)_{\bF_\ell}).\end{equation}

As in the proof of the first claim above, the results of Cohen imply that the homology $H_{*,*,*}(\mathrm{gr}(\overline{\mr{sk}(\gB)}_{\bF_\ell})/\sigma)$ is a free graded commutative algebra on $L/\langle \sigma \rangle$ (so we omit $\sigma=Q^\varnothing_\ell(\sigma)$), where $L$ is the $\bF_\ell$-vector space with basis $Q^I_\ell(y)$, where $y$ is a basic Lie word in $\{\sigma, \tau, \rho_i, x_\alpha\}$, satisfying certain conditions. We now use Cohen's calculations to study the $E^2$-page of (\ref{eqn:ss-cpsi-ell}). 

\vspace{1ex}

\noindent\textbf{Claim.} $E^{2,\ell}_{g,p,q}$ vanishes for $\frac{p+q}{g}<\frac{3}{4}$.

\begin{proof}[Proof of claim] Due to slight differences in results of Cohen's calculations, we consider the cases $\ell = 2$ and $\ell$ odd separately.
	
\vspace{1ex}

\noindent\emph{The case $\ell=2$.} In this case $Q^1_{\ell}(\sigma) = Q^1_2(\sigma)$. The groups $E^{2,\ell}_{*,*,*}$ are given by the homology of the chain complex
\[E^{1,\ell}_{*,*,*} = \left((S^{0,0,0}_{\bF_2} \oplus S^{3,3,3}_{\bF_2}\rho_4) \otimes \Lambda_{\bF_2}(L/\langle \sigma \rangle),d^1\right),\]
where $\Lambda_{\bF_2}$ denotes the free \emph{commutative} algebra in this case.

As in the proof of Lemma \ref{lem:VanishA}, we introduce an additional ``computational'' filtration. Let $F^\bullet E^{1,\ell}_{*,*,*}$ be the filtration in which $Q^1_2(\sigma)$, $\rho_2$, $\rho_3$ and $\rho_4$ are given filtration $0$, the remaining basis elements are given filtration equal to homological degree, and the filtration is extended multiplicatively to all of $\Lambda_{\bF_2}(L/\langle \sigma \rangle)$. The differential $d^1$ preserves this filtration. Its associated graded $\grr(F^\bullet E^{1,\ell}_{\ast,\ast,\ast})$ can be identified with the tensor product
\[\left((S^{0,0,0}_{\bF_2} \oplus S^{3,3,3}_{\bF_2}\rho_4) \otimes \Lambda_{\bF_2}(\bF_2\{Q^1_2(\sigma),\rho_2,\rho_3\}),\delta\right) \otimes (\Lambda_{\bF_2}(L/\langle \sigma,Q^1_2(\sigma),\rho_2,\rho_3\rangle),0)\]
where the differential is determined by $\delta(Q^1_2(\sigma)) = 0$, $\delta(\rho_2) = Q^1_2(\sigma)$, and $\delta(\rho_4) = \rho_3$, and multiplicativity. The right factor has basis concentrated in degrees $(g,p,q,r)$ with $\frac{p+q}{g} \geq \frac{3}{4}$. The left factor is itself a tensor product
\[\left((S^{0,0,0}_{\bF_2} \oplus S^{3,3,3}_{\bF_2}\rho_4) \otimes \Lambda_{\bF_2}(\bF_2\{\rho_3\}),\delta \right) \otimes \left(\Lambda_{\bF_2}(\bF_2\{ Q^1_2(\sigma),\rho_2\}),\delta \right).\]
The left factor of this has homology $\bF_\ell$ in degree $(0,0,0,0)$. The right factor has homology $\Lambda_{\bF_2}(\bF_2\{\rho_2^2\})$, so consists of elements of slope 1. We conclude that the second page of the spectral sequence converging to $E^{2,\ell}_{g,p,q}$ vanishes for $\frac{p+q}{g} < \frac{3}{4}$, hence so does $E^{2,\ell}_{g,p,q}$.

\vspace{1ex}

\noindent\emph{The case $\ell$ is odd.} In this case $Q^1_{\ell}(\sigma) = -\tfrac{1}{2}[\sigma, \sigma]$. As in the previous case we compute the homology of the chain complex
\[E^{1,\ell}_{*,*,*} = \left((S^{0,0,0}_{\bF_\ell} \oplus S^{3,3,3}_{\bF_\ell}\rho_4) \otimes \Lambda_{\bF_\ell}(L/\langle \sigma \rangle),d^1\right)\]
by introducing an additional ``computational'' filtration preserved by $d^1$. We let $F^\bullet E^{1,\ell}_{\ast,\ast,\ast}$ be the filtration in which $[\sigma,\sigma]$, $\rho_2$, $\rho_3$ and $\rho_4$ are given filtration $0$, the remaining basis elements are given filtration equal to their homological degree, and the filtration is extended multiplicatively. Its associated graded $\grr(F^\bullet E^{1,\ell}_{\ast,\ast,\ast})$ is given by a chain complex with zero differential which is concentrated in degrees $(g,p,q,r)$ with $\frac{p+q}{g} \geq \frac{3}{4}$, tensored with
\[\left((S^{0,0,0}_{\bF_\ell} \oplus S^{3,3,3}_{\bF_\ell}\rho_4) \otimes \Lambda_{\bF_\ell}(\bF_\ell\{\rho_3\}),\delta \right) \otimes \left(\Lambda_{\bF_\ell}(\bF_\ell\{[\sigma,\sigma],\rho_2\}),\delta \right),\]
where the differential is determined by $\delta([\sigma,\sigma]) = 0$, $\delta(\rho_2) = -\tfrac{1}{2}[\sigma,\sigma]$, $\delta(\rho_4) = \rho_3$, and multiplicativity. The left factor has homology $\bF_\ell$ in degree $(0,0,0,0)$. The non-zero class of lowest slope in the homology of the right factor is $[\sigma, \sigma] \cdot \rho_2^{\ell-1}$, of bidegree $(2\ell, 2\ell-1)$ and so of slope $\frac{2\ell-1}{2\ell} \geq \frac{5}{6}$. As before, we conclude that the second page of the spectral sequence converging to $E^{2,\ell}_{g,p,q}$ vanishes for $\frac{p+q}{g} < \frac{3}{4}$, hence so does $E^{2,\ell}_{g,p,q}$.
\end{proof}

It then follows from the spectral sequence that $H_{g,d}((\gC_\psi )_{\bF_\ell})=0$ for $d < \tfrac{3}{4} g$. As $H_{*,*}(\gC_\psi)$ is an iterated cofibre of self-maps of a finite type simplicial abelian group weakly equivalent to $\gR_\bZ$, it is of finite type and hence it follows from the universal coefficient theorem that $H_{g,d}(\gC_\psi)=0$ for $d < \tfrac{3}{4} g$, as required.
\end{proof}

\begin{corollary}\label{cor:MCGSecStabIntegral}
The homomorphisms
\[\varphi_* \colon H_{d-2}(\Gamma_{g-3,1}, \Gamma_{g-4,1};\bZ) \lra H_{d}(\Gamma_{g,1}, \Gamma_{g-1,1};\bZ),\]
induced by any of the maps $\varphi \colon \gR/\sigma \to \gR/\sigma$ constructed above are surjective for $d \leq \frac{3g-1}{4}$, and isomorphisms for $d \leq \frac{3g-5}{4}$.
\end{corollary}

\begin{proof}
Letting $\varphi \colon S^{3,2}_\bZ \otimes \overline{\gR}_\bZ/\sigma \to \overline{\gR}_\bZ/\sigma$ be any of the maps constructed in Section \ref{sec:ConstSecStabMap} by obstruction theory, the induced map on homology is of the given form. As Theorem \ref{thm:IntStabMCG} shows that $H_{g,d}(\gC_\varphi)=0$ for $4d \leq 3g-1$, the claim follows.
\end{proof}

\begin{corollary} \label{cor:rel-integral}
	For each $k \geq 1$ we have $H_{2k}(\Gamma_{3k,1}, \Gamma_{3k-1,1};\bZ) \cong \bZ$.
\end{corollary}
\begin{proof}
	By Lemma \ref{lem:MCGFacts} (\ref{it:MCGFacts32rel}) we have $H_{2}(\Gamma_{3,1}, \Gamma_{3-1,1};\bZ) = \bZ\{\mu\}$, and, by the corollary above, as $4 \cdot 4 \leq 3 \cdot 6-1$ and $4 \leq 6-2$ there are surjective maps
	\[\bZ\{\mu\} = H_2(\Gamma_{3,1}, \Gamma_{2,1};\bZ) \overset{\varphi_*}\lra H_4(\Gamma_{6,1}, \Gamma_{5,1};\bZ) \overset{\varphi_*}\lra \cdots \overset{\varphi_*}\lra H_{2k}(\Gamma_{3k,1}, \Gamma_{3k-1,1};\bZ).\]
	Thus $H_{2k}(\Gamma_{3k,1}, \Gamma_{3k-1,1};\bZ)$ is a cyclic abelian group, but it has rank 1 by Corollary \ref{cor:MCGPowersKappa}.
\end{proof}

\subsection{Mapping class groups with punctures and boundaries}

The previous ideas can also be used to study the homology of mapping class groups $\Gamma_{g, 1+B}^P$ of a genus $g$ surface with punctures in bijection with a finite set $P$, a distinguished boundary component, and further boundary components in bijection with a finite set $B$.  The main observation is that the chains on such mapping class groups can be made into a \emph{module} over the algebra $\gR_\bZ$ studied before.  We go into this both as an illustration of module structures and coefficient systems, and because the results we obtain will be used in the proof of Theorem~\ref{thm:H43} below.

More precisely, fix finite sets $P$ and $B$, and an embedding $(P \sqcup B) \times D^2 \hookrightarrow [0,1]^2 \times \{0\}$. Then we let the model surface $\Sigma_{0,1+B}^P \subset [0,1]^3$ be given by $[0,1]^2 \times \{0\} \setminus (P \sqcup B) \times \mr{int}(D^2)$. This has genus $0$ and is equal to $[0,1]^2 \times \{0\}$ near $\partial [0,1]^3$. So $\partial [0,1]^2 \times \{0\}$ forms part of the boundary of $\Sigma_{0,1+B}^P$ and is called the \emph{distinguished boundary component}. It has further boundary components which are identified with $S^1$ and are in bijection with $B$, called \emph{additional boundary components}, and boundary components which are in bijection with $P$ but which we do \emph{not} identify with $S^1$, called \emph{punctures}. We then write
\[\Sigma_{g,1+B}^P \coloneqq \Sigma_{g,1} \cup (\Sigma_{0,1+B}^P + g \cdot e_1) \subset [0,g+1] \times [0,1]^2,\]
and call $\partial_0 \Sigma_{g,1+B}^P \coloneqq \partial [0,g+1]^2 \times \{0\}$ the distinguished boundary component. 

Let $\Diff(\Sigma_{g,1+[B]}^{[P]})$ denote the diffeomorphisms of $\Sigma_{g,1+B}^P$ which are equal to the identity near the distinguished boundary, may permute the additional boundaries but must preserve their identification with $S^1$, and may permute the punctures. As always, this is a topological group with the $C^\infty$-topology. We set 
\[\Gamma_{g, 1+[B]}^{[P]} \coloneqq \pi_0\left(\Diff(\Sigma_{g,1+[B]}^{[P]})\right),\]
and write $\Gamma_{g, 1+B}^{P} \lhd \Gamma_{g, 1+[B]}^{[P]}$ for the normal subgroup of those isotopy classes of diffeomorphisms inducing the trivial permutation of the additional boundaries and punctures. 

Let the groupoid $\cat{MCG}_B^P$ have objects given by the natural numbers, and morphisms given by
\[\cat{MCG}_B^P(g,h) = \begin{cases} \Gamma_{g, 1+[B]}^{[P]} & \text{if $g=h$,} \\
\varnothing & \text{otherwise}.\end{cases}\]

We use $\fS_T$ to denote the group of permutations of a set $T$, and write $r' \colon \cat{MCG}_B^P \to \bN \times \fS_B \times \fS_P$ for the functor which associates to the object $g$ the object $(g, *, *)$, and associates to an isotopy class the permutations of additional boundaries and of punctures it induces. There is a (left) action of the strict monoidal category $(\cat{MCG}, \oplus, 0)$ given by
\begin{align*}
\cat{MCG} \times \cat{MCG}_B^P  &\lra \cat{MCG}_B^P\\
(g, h) &\longmapsto g+h\\
(\varphi \in \Gamma_{g, 1}, \psi \in \Gamma_{g, 1+[B]}^{[P]}) &\longmapsto \varphi \cup (\psi + g \cdot e_1) \in \Gamma_{g+h, 1+[B]}^{[P]}
\end{align*}
where $\varphi \cup (\psi + g \cdot e_1)$ is considered as a diffeomorphism of
\[\Sigma_{g, 1} \cup (\Sigma_{h,1+B}^P+ g \cdot e_1) = \Sigma_{g+h, 1+B}^P.\]

The terminal functor $\underline{\ast}_{\geq 0}$ in $\cat{sSet}^{\cat{MCG}^P_B}$ sending everything to $\ast$ is canonically a module over the terminal non-unital associative algebra $\underline{\ast}_{> 0}$ in $\cat{sSet}^\cat{MCG}$ and hence also over its cofibrant replacement $\gT' \in \cat{Alg}_{\mr{Ass}} (\cat{sSet}^\cat{MCG})$ (as in \cite{e2cellsI}, $\mr{Ass}$ denotes the non-unital associative operad). Let $\gU \in \cat{sSet}^\cat{MCG_B^P}$ denote a cofibrant replacement of the terminal functor $\underline{\ast}_{\geq 0}$ considered as a $\gT'$-module. 
Since the following diagram commutes
\[\begin{tikzcd} \cat{MCG} \times \cat{MCG}_B^P \rar \dar{r \times r'} &[15pt] \cat{MCG}_B^P \dar{r'} \\
\bN \times \bN \times \fS_B \times \fS_P \rar{+ \times \mr{id} \times \mr{id}} & \bN  \times \fS_B \times \fS_P,\end{tikzcd}\]
with horizontal map given by addition in the $\bN$ entries, the pushforward  $(r')_*(\gU) \in \cat{sSet}^{\bN \times \fS_B \times \fS_P}$ is a module over $r_*(\gT') \in \Alg_{\mr{Ass}}(\cat{sSet}^\bN)$. We have that 
\[r_*(\gT')(g) = \bL r_*(\underline{\ast}_{> 0}) \simeq \begin{cases} \varnothing & \text{if $g=0$,}\\
 B\Gamma_{g,1} & \text{otherwise,}\end{cases}\]
 \[\text{and} \qquad (r')_*(\gU)(g) = \bL (r')_*(\underline{\ast}_{\geq 0}) \simeq B\Gamma_{g, 1+[B]}^{[P]}.\]

We claim that as a non-unital $E_1$-algebra $r_*(\gT')$ is weakly equivalent to $\gR$. Recall that $\gT$ denotes the cofibrant replacement of the terminal non-unital $E_2$-algebra $\underline{\ast}_{> 0}$ in $\Alg_{E_2}(\cat{sSet}^\cat{MCG})$. We may regard $\gT'$ as a non-unital $E_1$-algebra and as such it has a unique map to $\underline{\ast}_{> 0}$. Since trivial fibrations of non-unital $E_1$-algebras are those of the underlying objects, this is a trivial fibration of non-unital $E_1$-algebras and thus there is a lift $\gT \to \gT'$ of non-unital $E_1$-algebras. This induces a weak equivalence $\gR = r_*(\gT) \to r_*(\gT')$ of non-unital $E_1$-algebras, and hence on unitalisations.

We denote $r_*(\gT')$ by $\widetilde{\gR}$, and use $\widetilde{\gR}^+$ for its unitalisation. We also denote $(r')_*(\gU)$ by $\gM$, whose $\smash{\widetilde{\gR}}$-module structure induces an $\smash{\widetilde{\gR}}^+$-module structure. To save space, we shall write $g$ for the object $(g,*,*) \in \bN \times \fS_B \times \fS_P$. 

In Section $E_k$.9.4.2 we have explained that---like $E_2$-indecomposables---we can form $\widetilde{\gR}^+$-module indecomposables $Q^{\widetilde{\gR}^+}_\bL(\gM)$ and hence define $\widetilde{\gR}^+$-homology groups \[H^{\widetilde{\gR}^+}_{g,d}(\gM) \coloneqq \tilde{H}_{g,d}(Q^{\widetilde{\gR}^+}_\bL(\gM)).\] By Corollary $E_k$.9.17 these can be computed by a bar construction: assuming that $\widetilde{\gR}^+$ and $\gM$ are cofibrant in $\cat{sSet}^\bN$ then we have $Q^{\widetilde{\gR}^+}_\bL(\gM) \simeq B(1,\widetilde{\gR}^+,\gM)$.

\begin{theorem}\label{thm:MCGModuleCells}
	We have $H^{\smash{\widetilde{\gR}}^+}_{g,d}(\gM)=0$ for $d < g$.
\end{theorem}
\begin{proof} We have that $\smash{\widetilde{\gR}}^+$ is cofibrant in $\cat{sSet}^\bN$, and $\gM$ is cofibrant in $\cat{sSet}^{\bN \times \fS_B \times \fS_P}$ so in particular gives a cofibrant object in $\cat{sSet}^\bN$ with an action of $\fS_B \times \fS_P$. Thus by Corollary~$E_k$.9.17 we have $Q^{\smash{\widetilde{\gR}}^+}_\bL(\gM) \simeq B(\mathbbm{1}, \smash{\widetilde{\gR}}^+, \gM)$, where the bar construction is taken in $\cat{sSet}^\bN$ (for this argument the $\fS_B \times \fS_P$-action on $\gM$ may be ignored). Thus applying $B_\bullet(-,\smash{\widetilde{\gR}}^+,\gM)$ to anything which is cofibrant in $\cat{sSet}^\bN$ gives rise to a Reedy-cofibrant semi-simplicial object and hence $B(-,\smash{\widetilde{\gR}}^+,\gM)$ preserves cofibration sequences between cofibrant objects. In particular there is a cofibration sequence
	\[B(\smash{\widetilde{\gR}}, \smash{\widetilde{\gR}}^+, \gM)\lra B(\smash{\widetilde{\gR}}^+, \smash{\widetilde{\gR}}^+, \gM)\lra B(\mathbbm{1}, \smash{\widetilde{\gR}}^+, \gM)\]
	and the augmentation $\epsilon \colon B(\smash{\widetilde{\gR}}^+,\smash{\widetilde{\gR}}^+, \gM) \to \gM$ is homotopy equivalence, as this augmented semi-simplicial object has an extra degeneracy. Therefore it is enough to show that the homotopy fibre of the augmentation
	\[\epsilon \colon B(\smash{\widetilde{\gR}}, \smash{\widetilde{\gR}}^+, \gM) \lra \gM\]
	has a certain connectivity. At the object $g \in \bN$, on $p$-simplices this is the inclusion
	\begin{equation}\label{eqn:young-type-pb}\bigsqcup_{\substack{g_{0} > 0, g_1, \ldots, g_{p+1} \geq 0\\ \sum g_i=g}}
	B\Gamma_{(g_0,\ldots,g_{p+1}),1+B}^P \lra B\Gamma_{g, 1+B}^P,\end{equation}
	where $\Gamma_{(g_0,\ldots,g_{p+1}),1+B}^P$ is the Young-type subgroup of $\Gamma_{g,1+B}^P$ associated to the decomposition 
	\[\Sigma_{g,1+B}^P = \bigsqcup_{i=0}^P (\Sigma_{g_i,1}+(g_0+\cdots+g_{i-1})\cdot e_1) \sqcup (\Sigma_{g_p,1+B}^P+(g_0+\cdots+g_{p-1})\cdot e_1),\] 
	which is isomorphic to $\Gamma_{g_0, 1} \times \Gamma_{g_1, 1} \times \cdots \times \Gamma_{g_p, 1} \times \Gamma_{g_{p+1}, 1+B}^P$ by a mild generalisation of Lemma \ref{lem:MonoidalInj}. The map (\ref{eqn:young-type-pb}) therefore has homotopy fibre homotopy equivalent to the set
	\[\bigsqcup_{\substack{g_{0} > 0, g_1, \ldots, g_{p+1} \geq 0\\ \sum g_i=g}} \frac{\Gamma_{g, 1+B}^P}{\Gamma_{(g_0,\ldots,g_{p+1}),1+B}^P},\]
	and these form the $p$-simplices of a semi-simplicial set whose thick geometric realisation is therefore equivalent to the homotopy fibre of $\epsilon$ at $g \in \bN$.
	
	If $B=P=\varnothing$ then $\gM \cong \smash{\widetilde{\gR}}^+$ so $Q^{\smash{\widetilde{\gR}}^+}_\bL(\gM) \simeq S^{0,0}$ and there is nothing to show.  If $B$ or $P$ is non-empty, then just as in the proof of Proposition \ref{prop:SplittingModel} we see that there is a homotopy equivalence
	\[\mathrm{hofib}(B(\smash{\widetilde{\gR}}, \smash{\widetilde{\gR}}^+, \gM)(g) \to \gM(g)) \simeq S(\Sigma_{g, 1+B}^P, \partial_0 \Sigma_{g, 1+B}^P, (0,0,0), (0,1,0))_\bullet,\]
	where the latter is as in Definition \ref{defn:SplittingCx2}. By Theorem \ref{thm:HighConnSCx} (ii) this is $(g-2)$-connected, so
	\[Q^{\smash{\widetilde{\gR}}^+}_\bL(\gM)(g) \simeq B(\mathbbm{1}, \smash{\widetilde{\gR}}^+, \gM)(g) \simeq \mathrm{hocofib}(B(\smash{\widetilde{\gR}}, \smash{\widetilde{\gR}}^+, \gM)(g) \lra \gM(g))\]
	is $(g-1)$-connected, as required.
\end{proof}

This theorem allows us to promote the results developed so far, of homological stability and of secondary homological stability, from surfaces with a single boundary to surfaces with multiple boundaries and with punctures.

\begin{corollary}\label{cor:punctures-secondary-stability}
	For any finite sets $B$ and $P$ we have $H_d(\Gamma_{g, 1+B}^P, \Gamma_{g-1, 1+B}^P;\bZ)=0$ for $3d \leq 2g-1$. Furthermore, there is a $\fS_B \times \fS_P$-equivariant map
	\[\varphi'_* \colon H_{d-2}(\Gamma_{g-3, 1+B}^P, \Gamma_{g-4, 1+B}^P;\bZ) \lra H_d(\Gamma_{g, 1+B}^P, \Gamma_{g-1, 1+B}^P;\bZ)\]
	which is an epimorphism for $4d \leq 3g-1$ and an isomorphism for $4d \leq 3g-5$.
\end{corollary}

\begin{proof}
	For studying homology we may as well pass from simplicial sets to simplicial $\bZ$-modules, resulting in a non-unital associative algebra $\gR_\bZ \in \cat{sMod}_\bZ^{\bN}$ and an $\smash{\widetilde{\gR}}^+_\bZ$-module $\gM_\bZ \in \cat{sMod}_\bZ^{\bN}$ with an action of $\fS_B \times \fS_P$. Using the equivalence $\overline{\gR} \overset{\sim}\to \smash{\widetilde{\gR}}^+$ we may consider $\gM_\bZ$ as a $\overline{\gR}_\bZ$-module. 
	
	The first part concerns the vanishing of the homology of $\gM_\Z/\sigma$. For the second part, we may form the map $\varphi' \coloneqq B(\varphi, \overline{\gR}_\bZ, \gM_\Z)$, considered as a map
	\[S^{3,2}_\bZ \otimes \gM_\Z/\sigma \simeq  B(S^{3,2}_\bZ \otimes\overline{\gR}_\bZ/\sigma , \overline{\gR}_\bZ , \gM_\Z) \lra B(\overline{\gR}_\bZ/\sigma , \overline{\gR}_\bZ , \gM_\Z) \simeq \gM_\Z/\sigma.\]
	This is $\fS_B \times \fS_P$-equivariant, by functoriality, and the second part will follow if we show that the homology of its cofibre vanishes in bidegrees $(g,d)$ with $d < \frac{3}{4}g$. When proving either part, we may ignore equivariance.
	
	\vspace{1em}
	
	For the first part, we use that $\overline{\mathbf{R}}$, $\widetilde{\mathbf{R}}^+$ and $\mathbf{M}$ are cofibrant and the inclusion of the unit into $\overline{\gR}$ is a cofibration, to obtain equivalences 
	\[Q^{\smash{\widetilde{\gR}}^+}_\bL(\gM) \simeq B(\mathbbm{1}, \smash{\widetilde{\gR}}^+, \gM) \simeq  B(\mathbbm{1}, \overline{\gR}, \gM) \simeq Q^{\overline{\gR}}_\bL(\gM)\]
	in $\cat{sMod}_\bZ^\bN$, and so by Theorem \ref{thm:MCGModuleCells} we have $\smash{H^{\overline{\gR}}}_{g,d}(\gM)=0$ for $d < g$. By Theorem $E_k$.11.21 there is a CW approximation $\gZ \overset{\sim}\to \gM_\bZ$ in the category of $\overline{\gR}_\bZ$-modules which only has cells of bidegree $(g, d)$ with $d \geq g$. 
	
	Letting $\mr{sk}(\gZ)$ denote the skeletal filtration of this $\overline{\gR}_\bZ$-module, we have $\grr(\mr{sk}(\gZ)) = \overline{\gR}_\bZ \otimes (\bigoplus_{\alpha \in I} S^{g_\alpha, d_\alpha, d_\alpha}_\bZ)$ with $d_\alpha \geq g_\alpha$ for each $\alpha$.  Then the skeletal filtration of $\gZ$ induces a filtration of $\gZ/\sigma$, giving a spectral sequence
	\[E^1_{g, p, q} = H_{g, p+q, q}\left(\overline{\gR}_\bZ/\sigma \otimes \left(\bigoplus_{\alpha \in I} S^{g_\alpha, d_\alpha, d_\alpha}_\bZ\right)\right) \Longrightarrow H_{g, p+q}(\gZ/\sigma),\]
	and as $H_{g,d}(\overline{\gR}_\bZ/\sigma)=0$ for $d < \tfrac{2}{3}g$ it follows that $H_{g,d}(\gZ/\sigma)$ vanishes in this range too. As $H_{g,d}(\gZ/\sigma) \cong H_{g,d}(\gM_\bZ/\sigma) = H_d(\Gamma_{g, 1+B}^P, \Gamma_{g-1, 1+B}^P;\bZ)$ this proves the first part.
	
	\vspace{1em}
	
	Secondly, use $\gZ/\sigma \simeq B(\overline{\gR}_\bZ/\sigma , \overline{\gR}_\bZ , \gZ)$ and the right $\overline{\gR}_\bZ$-module map $\varphi \colon S^{3,2}_\bZ \otimes \overline{\gR}_\bZ/\sigma \to \gR^+_\bZ/\sigma$ constructed in the proof of Theorem \ref{thm:IntStabMCG} to form $\psi' \coloneqq B(\varphi , \overline{\gR}_\bZ , \gZ)$ as a map
	\[S^{3,2}_\bZ \otimes \gZ/\sigma \simeq  B(S^{3,2}_\bZ \otimes\overline{\gR}_\bZ/\sigma , \overline{\gR}_\bZ , \gZ) \lra B(\overline{\gR}_\bZ/\sigma , \overline{\gR}_\bZ , \gZ) \simeq \gZ/\sigma\]
	with cofibre $B(\gC_\varphi , \overline{\gR}_\bZ , \gZ)$. The spectral sequence associated to the filtered object $B(0_*\gC_\varphi , 0_*\overline{\gR}_\bZ , \mr{sk}(\gZ))$ coming from the skeletal filtration of $\gZ$ takes the form
	\[E^1_{g, p, q} = H_{g, p+q, q}\left(\gC_\varphi\otimes \left(\bigoplus_{\alpha \in I} S^{g_\alpha, d_\alpha,d_\alpha}_\bZ\right)\right) \Longrightarrow H_{g, p+q}(B(\gC_\varphi , \overline{\gR}_\bZ , \gZ)),\]
	and as $H_{g,d}(\gC_\varphi)=0$ for for $d < \tfrac{3}{4}g$ it follows that $H_{g, d}(B(\gC_\varphi , \overline{\gR}_\bZ , \gZ))=0$ in this range too. Under the equivalence $\gZ \simeq \gM_\bZ$ the map induced by $\varphi'$ on homology has the claimed property.
\end{proof}

\begin{remark}
For normal subgroups $H_B \lhd \fS_B$ and $H_P \lhd \fS_P$, pushing forward along the evident functor $q \colon \bN \times \fS_B \times \fS_P \to \bN \times \frac{\fS_B}{H_B} \times \frac{\fS_P}{H_P}$ gives an object $q_\ast(\gM)$ with
\[q_\ast(\gM)(g, *, *) \simeq B\Gamma_{g,1+[B]_{H_B}}^{[P]_{H_P}},\]
the classifying space of the normal subgroup $\Gamma_{g,1+[B]_{H_B}}^{[P]_{H_P}} \leq \Gamma_{g,1+[B]}^{[P]}$ consisting of those isotopy classes of diffeomorphisms of $\Sigma_{g, 1+B}^P$ which induce a permutation in $H_B$ of the additional boundaries and a permutation in $H_P$ of the punctures. By applying $q_\ast(-)$ to the argument above, one gets the same secondary stability theorem for the homology of these groups.
\end{remark}

\begin{remark}The construction used in this section generalises to the situation of a braided monoidal groupoid $\cat{G}$ acting on another category $\cat{M}$, in this case $\cat{G} = \cat{MCG}$ acts on $\cat{M} = \cat{MCG}^P_B$. Such a setup naturally produces an $E_1$-module over an $E_2$-algebra, though above we rigidified this to a $\widetilde{\gR}^+$-module structure on $\gM$. Apart from the connectivity result for separating arc complexes used in the proof of Theorem \ref{thm:MCGModuleCells}, the arguments in this section are formal and can be repeated to give generic homological stability results akin to Theorem $E_k$.18.1 for braided monoidal groupoids $\cat{G}$ acting on a category $\cat{M}$. This is analogous to the generalisation of the homological stability arguments of \cite{RWW} given in \cite{Krannich}.
\end{remark}

\subsection{Coefficient systems}

In Section $E_k$.19.1 we discussed a general set-up for coefficient systems which applies here.

\subsubsection{Twisted homological stability for mapping class groups}

We defined a coefficient system to be a functor $\gA \colon \cat{MCG} \to \cat{sMod}_\mathbbm{k}$ with the structure of a left module over the commutative algebra object given by the constant functor $\underline{\mathbbm{k}} \in \cat{sMod}_\mathbbm{k}^\cat{MCG}$. This yields for each $g$ a (simplicial) $\mathbbm{k}[\Gamma_{g,1}]$-module $\gA(g)$, with (hyper)homology $\bH_*(\Gamma_{g,1} ; \gA(g))$, and the $\underline{\mathbbm{k}}$-module structure yields stabilisation maps between such (hyper)homology groups. This of course includes the case of coefficient systems of $\bk$-modules, via the inclusion $\cat{Mod}_\mathbbm{k} \subset \cat{sMod}_\mathbbm{k}$, in which case hyperhomology reduces to ordinary homology. We gave a general homological stability result (Theorem $E_k$.19.2) for such stabilisation maps, in terms of the $\overline{\gR}_\bk$-module homology of $\gR_\gA \coloneqq \bL r_*(\gA)$, which in this case is as follows.

\begin{theorem}\label{thm:TwistedStab}
Let $\gA \in \underline{\mathbbm{k}}\text{-}\mathrm{Mod}$ be a coefficient system and $\lambda \leq \tfrac{2}{3}$ be such that $H_{g,d}^{\overline{\gR}_\bk}(\gR_\gA)=0$ for $d < \lambda(g - c)$. Then the map
\[\sigma \cdot - \colon \bH_d(\Gamma_{g-1,1} ; \gA(g-1)) \lra \bH_d(\Gamma_{g,1} ; \gA(g))\]
is an epimorphism for $d < \lambda(g - c)$ and an isomorphism for $d < \lambda(g -c) -1$.
\end{theorem}

Moreover, the work we have done so far immediately proves secondary homological stability with coefficient systems satisfying a hypothesis similar to that of Theorem \ref{thm:TwistedStab}. The main difficulty is stating the result in classical terms: for a (simplicial) $\mathbbm{k}[H]$-module $N$ and a (simplicial) $\mathbbm{k}[G]$-module $M$, we define the relative (hyper)homology associated to a homomorphism $\varphi \colon H \to G$ of groups and a $\varphi$-equivariant homomorphism $\psi \colon N \to M$ of (simplicial) $\mathbbm{k}$-modules as the homology of the mapping cone of
\[E \varphi \otimes_{\varphi} \psi \colon \mathbbm{k}E H \otimes_{\mathbbm{k}[H]} N \lra \mathbbm{k}E G \otimes_{\mathbbm{k}[G]} M.\]
Abusing notation, we write such homology groups as $\bH_{*}(G, H ; M, N)$.

The relative hyperhomology groups $\bH_{d}(\Gamma_{g,1}, \Gamma_{g-1,1} ; \gA(g), \gA(g-1))$ of the stabilisation map $s \colon \Gamma_{g-1,1} \to \Gamma_{g,1}$ may be identified in terms of $\gR_\gA$ as follows. The object $\gR_\bA = \bL r_*(\gA)$ can be explicitly computed by taking the cofibrant replacement $\gT$ of $\underline{\ast}_{>0}$ in $\Alg_{E_2}(\cat{sSet}^\cat{MCG})$, taking its levelwise tensor with $\gA$ and applying $r_*$. Because $\gT$ is cofibrant in $\cat{sSet}^\cat{MCG}$, it is contractible with free $\Gamma_{g,1}$-action and we have that $\gR_\gA(g)$ is weakly equivalent to $\bk E\Gamma_{g,1} \otimes_{\bk \Gamma_{g,1}} \gA(g)$. In terms of this description, the map $\sigma \colon \gR_\gA(g-1) \to \gR_\gA(g)$ is given by 
\[Es \otimes \gA(s) \colon \bk E\Gamma_{g-1,1} \otimes_{\bk \Gamma_{g-1,1}}  \gA(g-1) \lra \bk E\Gamma_{g,1} \otimes_{\bk \Gamma_{g,1}}\gA(g),\] whose mapping cone has homology equal to the relative hyperhomology.

\begin{theorem}\label{thm:TwistedSecondStab}
Let $\gA \in \underline{\mathbbm{k}}\text{-}\mathrm{Mod}$ be a coefficient system and $\lambda \leq \tfrac{3}{4}$ be such that $H_{g,d}^{\overline{\gR}_\bk}(\gR_\gA)=0$ for $d < \lambda(g - c)$. There is a map
\[\varphi_* \colon \bH_{d-2}(\Gamma_{g-3,1}, \Gamma_{g-4,1} ; \gA(g-3), \gA(g-4)) \lra \bH_{d}(\Gamma_{g,1}, \Gamma_{g-1,1} ; \gA(g), \gA(g-1))\]
which is an epimorphism for $d < \lambda(g - c)$ and an isomorphism for $d < \lambda(g-c)-1$.
\end{theorem}

\begin{proof}
We consider the defining cofibration sequence of right $\overline{\gR}_\bk$-modules
\[S^{3,2} \otimes \overline{\gR}_\bk/\sigma \overset{\varphi}\lra \overline{\gR}_\bk/\sigma \lra \gC_\varphi\]
and apply $B(- ; \overline{\gR}_\bk ; \gR_\gA)$, which preserves cofibration sequences as before. We have $B(\overline{\gR}_\bk/\sigma; \overline{\gR}_\bk; \gR_\gA) \simeq \gR_\gA/\sigma$ so there is a homotopy cofibre sequence
\[S^{3,2}_\bk \otimes \gR_\gA/\sigma \overset{\varphi}\lra \gR_\gA/\sigma \lra B(\gC_\varphi; \overline{\gR}_\bk;  \gR_\gA)\]
and we must show that the right-hand term has vanishing homology in bidegrees $(g,d)$ with $d < \lambda (g-c)$.

As in the proof of Theorem $E_k$.19.2, up to weak equivalence the left $\overline{\gR}_\bk$-module $\gR_\gA$ admits a filtration with associated graded $\overline{\gR}_\bk \otimes (\bigoplus_{\alpha \in I} S^{g_\alpha, d_\alpha,d_\alpha}_\bk)$ with $d_\alpha \geq \lambda(g_\alpha -c)$, and so up to weak equivalence $B(\gC_\varphi , \overline{\gR}_\bk , \gR_\gA)$ admits a filtration with associated graded $\gC_\varphi \otimes (\bigoplus_{\alpha \in I} S^{g_\alpha, d_\alpha,d_\alpha}_\bk )$ with $d_\alpha \geq\lambda(g_\alpha - c)$. By Theorem \ref{thm:IntStabMCG} we have $H_{g,d}(\gC_\varphi)=0$ for $4d < 3g$, and so for $d < \lambda g$. As in the proof of Theorem $E_k$.19.2 the $E^1$-page of the associated spectral sequence vanishes in bidegrees $(g,d)$ satisfying
\[d < \lambda g + \min_{\alpha \in I}(d_\alpha-\lambda g_\alpha),\]
and so in particular in bidegrees satisfying $d < \lambda g -\lambda c$ as required.
\end{proof}

\begin{remark}
If $\gA$ is a coefficient system which is discrete (i.e.\ $H_{g,d}(\gA)=0$ for all $d \neq 0$) then in Lemma $E_k$.19.4 we described a sufficient condition for the hypothesis ``$\smash{H_{g,d}^{\overline{\gR}_\bk}}(\gR_\gA)=0$ for $d < \lambda(g - c)$" to hold, namely that $\smash{\mathrm{Tor}^{\underline{\bk}}_d}(\bk, \gA)(g)=0$ for $d < \lambda(g - c)$ with $\mathrm{Tor}$-groups formed in the functor category $\mathsf{Mod}_\bk^\mathsf{MCG}$. This hypothesis should perhaps be considered as a derived form of representation stability. Theorem A of \cite{patztcentral} shows how this condition is related to \emph{central stability homology}, which in turn by Theorem D of \cite{patztcentral} implies homological stability with such coefficients.
\end{remark}

Let us give the main explicit examples of coefficient systems to which the above results may be applied: the objectwise tensor power of linear coefficient systems (as in Definition $E_k$.19.7).

\begin{example}\label{ex:H1MCG} Consider the functor $\gH_1 \colon \cat{MCG} \to \cat{Mod}_\mathbbm{k}$ given by $\gH_1(g) \coloneqq H_1(\Sigma_{g,1};\mathbbm{k})$ with the evident $\Gamma_{g,1}$-action. The decomposition $\Sigma_{k+g, 1} = \Sigma_{k,1} \cup (\Sigma_{g,1} + k \cdot e_1)$ and Mayer--Vietoris yields isomorphisms
\[s_{k,g} \colon \gH_1(k) \oplus \gH_1(g) \lra \gH_1(k+g)\]
defining a strong monoidality on $\gH_1$, hence making it a linear coefficient system.

Writing $\gH_1^{\boxtimes s}$ for the $s$th objectwise tensor power of $\gH_1$, which is a coefficient system, Corollary $E_k$.19.9 and the change-of-diagram-category spectral sequence of Theorem $E_k$.10.13 show that
\begin{equation}\label{eq:h1-coeff-module} H_{g,d}^{\overline{\gR}_\bk}(\gR_{\gH_1^{\boxtimes s}})=0 \text{ for } d < g - s.\end{equation}
In particular, this vanishes for $3d < 2(g-s)$ and for $4d < 3(g-s)$. %Because hyperhomology groups reduce to ordinary homology groups when the differential on the coefficients is trivial, 
Theorem \ref{thm:TwistedStab} applies to show that
\[\sigma \cdot - \colon H_d(\Gamma_{g-1,1} ; H_1(\Sigma_{g-1,1};\mathbbm{k})^{\otimes s}) \lra H_d(\Gamma_{g,1} ; H_1(\Sigma_{g,1};\mathbbm{k})^{\otimes s})\]
is an epimorphism for $3d \leq 2g -2s-1$ and an isomorphism for $3d \leq 2g -2s -4$. This improves on the range of homological stability for $\Gamma_{g,1}$ with coefficients in tensor powers of $H_1(\Sigma_{g,1})$ which may be deduced from \cite{Ivanov, Boldsen, RWW}. Similarly Theorem \ref{thm:TwistedSecondStab} applies to show that the associated secondary stabilisation map is an epimorphism for $4d \leq 3g -3s-1$ and an isomorphism for $4d \leq 3g -3s-5$. 
\end{example}

\subsubsection{Improvements for the coefficient system $\gH_1$}
The constant in the range of Theorem \ref{thm:TwistedSecondStab} can occasionally be improved. We shall explain this for the coefficient system $\gH_1$; a consequence of this improved range which shall be used to prove Theorem \ref{thm:H43}.

\begin{lemma}\label{lem:H1CoinvVanish}
We have that $H_{g,0}(\gR_{\gH_1}) = H_0(\Gamma_{g,1} ; H_1(\Sigma_{g,1} ; \mathbbm{k}))=0$ for all $g \geq 0$.
\end{lemma}

\begin{proof}
For $g=0$ there is nothing to show, as $H_1(\Sigma_{0,1} ; \mathbbm{k})=0$. For $g \geq 1$, as a $\bk[\Gamma_{1,1} \times \cdots \times \Gamma_{1,1}]$-module we have $H_1(\Sigma_{g,1} ; \mathbbm{k}) = H_1(\Sigma_{1,1} ; \mathbbm{k})^{\oplus g}$, so there is a surjection $H_0(\Gamma_{1,1};H_1(\Sigma_{1,1} ; \mathbbm{k}))^{\oplus g} \to H_0(\Gamma_{g,1} ; H_1(\Sigma_{g,1} ; \mathbbm{k}))$ and thus it suffices to establish this for $g=1$. Furthermore, by the universal coefficient theorem it suffices to establish it for $\mathbbm{k}=\bZ$. 

A Dehn twist around a simple closed curve in the class of $a$ acts on $H_1(\Sigma_{1,1} ; \bZ) = \bZ\{a, b\}$  as the matrix $\bigl( \begin{smallmatrix}1 & 1\\ 0 & 1\end{smallmatrix}\bigr)$ and a Dehn twist around a simple closed curve in the class of $b$ acts as the matrix $\bigl( \begin{smallmatrix}1 & 0\\ -1 & 1\end{smallmatrix}\bigr)$. In the coinvariants we thus have $b \sim b + a$ and $a \sim a-b$, so $a=b=0$.
\end{proof}

\begin{proposition}\label{prop:StabH1}
The map
\[\sigma \cdot - \colon H_d(\Gamma_{g-1,1} ; H_1(\Sigma_{g-1,1};\bZ)) \lra H_d(\Gamma_{g,1} ; H_1(\Sigma_{g,1};\bZ))\]
is an epimorphism for $3d \leq 2g -2$ and an isomorphism for $3d \leq 2g -5$, and there is a map
\[\begin{tikzcd} H_{d-2}(\Gamma_{g-3,1}, \Gamma_{g-4,1} ; H_1(\Sigma_{g-3,1};\bZ), H_1(\Sigma_{g-4,1};\bZ)) \dar{\varphi_*} \\  H_{d}(\Gamma_{g,1}, \Gamma_{g-1,1} ; H_1(\Sigma_{g,1};\bZ), H_1(\Sigma_{g-1,1};\bZ))\end{tikzcd}\]
which is an epimorphism for $4d \leq 3g -3$ and an isomorphism for $4d \leq 3g -7$.
\end{proposition}
\begin{proof}
It follows from \eqref{eq:h1-coeff-module} in the case $s=1$ and Lemma \ref{lem:H1CoinvVanish} that $\gR_{\gH_1}$ up to equivalence may be constructed as a $\overline{\gR}_\bk$-module using cells of bidegrees $(g,d)$ with $d \geq g-1$ and $d>0$. In particular, we may find a CW approximation $\gZ \overset{\sim}\to \gR_{\gH_1}$ in the category of $\overline{\gR}_\bk$-modules, whose skeletal filtration $\mr{sk}(\gZ)$ has associated graded $\mathrm{gr}(\mr{sk}(\gZ)) = \overline{\gR}_\bk \otimes (\bigoplus_{\alpha \in I} S^{g_\alpha, d_\alpha, d_\alpha}_\bk )$ with $d_\alpha \geq g_\alpha-1$ and $d_\alpha>0$. Now follow the arguments of Theorem $E_k$.19.2 and Theorem \ref{thm:TwistedSecondStab}, using that $\gZ$ has no $(1,0)$-cell.
\end{proof}

\begin{corollary}\label{cor:H1Vanish42}
For $g \geq 4$, $H_2(\Gamma_{g,1} ; H_1(\Sigma_{g,1};\bZ[\tfrac{1}{10}]))=0$.
\end{corollary}
\begin{proof}
We take $\mathbbm{k} = \bZ[\tfrac{1}{10}]$. As discussed in Lemma \ref{lem:LambdaOverTen}, with these coefficients the secondary stability map $\varphi$ coincides with multiplication by the class $\tfrac{1}{10}\lambda \in H_{3,2}(\overline{\gR}_\bk)$ using the $\overline{\gR}_\bk$-module structure on $\gR_{\gH_1}/\sigma$. We may therefore form the commutative diagram
\begin{equation*}
\begin{tikzcd}
H_{2,1}(\gR_{\gH_1}/\sigma) \rar{\partial} \dar{\tfrac{1}{10}\lambda \cdot -}& H_{1,0}(\gR_{\gH_1}) \rar{\sigma \cdot -} \dar{\tfrac{1}{10}\lambda \cdot -}& H_{2,0}(\gR_{\gH_1}) \dar{\tfrac{1}{10}\lambda \cdot -}\\
H_{5,3}(\gR_{\gH_1}/\sigma) \rar{\partial}& H_{4,2}(\gR_{\gH_1}) \rar{\sigma \cdot -}& H_{5,2}(\gR_{\gH_1}).
\end{tikzcd}
\end{equation*}
Now by the second part of Proposition \ref{prop:StabH1} the left-hand vertical map is surjective, but by Lemma \ref{lem:H1CoinvVanish} we have $H_{1,0}(\gR_{\gH_1})=0$, and so the top left-hand map is zero. Thus the bottom left-hand map is zero, and so
\[\sigma \cdot - \colon H_{4,2}(\gR_{\gH_1}) \lra H_{5,2}(\gR_{\gH_1})\]
is injective. Further stabilisation maps are isomorphisms, by the first part of Proposition \ref{prop:StabH1}, so the claim now follows from Morita's calculation \cite[Theorem 2]{MoritaTwisted} that $H_2(\Gamma_{g,1} ; H_1(\Sigma_{g,1};\bZ))=0$ for $g \geq 12$.
\end{proof}

\begin{remark}The groups $H_d(\Gamma_{g,1};H_1(\Sigma_{g,1};\bQ))$ are known for low $d$ and $g$. Using the results in \cite{ABE}, one may deduce that $H_2(\Gamma_{g,1};H_1(\Sigma_{g,1};\bQ)) = 0$ for $g \leq 2$, and that $H_3(\Gamma_{g,1};H_1(\Sigma_{g,1};\bQ)) = 0$ for $g \leq 2$.\end{remark}

\section{Relations in the tautological ring and their consequences} In this section we shall use results on the tautological ring to obtain a strengthened rational secondary stability range and compute a number of non-vanishing relative groups.
	
\subsection{A strengthening of the rational secondary stability range}\label{sec:RatStrengthening} 
Using the methods described in the previous section, we can extend the range of bidegrees enjoying secondary homological stability with rational coefficients once we establish the following result:

\begin{theorem}\label{thm:H43}
$H_{4,3}(\gR_\bQ) = H_3(\Gamma_{4,1};\bQ)=0$.
\end{theorem}

Before proving this theorem, we use it to improve the range for rational secondary stability that was obtained in Corollary \ref{cor:MCGSecStab}. We return to the argument of Section \ref{sec:RatStab}, and in particular the map
\[c \colon \gA = \gE_2(S^{1,0}_\bQ\sigma  \oplus S^{3,2}_\bQ\lambda) \cup^{E_2}_{[\sigma,\sigma]} \gD^{2,2}_\bQ \rho \lra \gR_\bQ\]
of $E_2$-algebras. We have the following strengthening of Lemma \ref{lem:VanishingRelE2HomologyMCG}.

\begin{lemma}
We have $H_{g,d}^{E_2}(\gR_\bQ, \gA)=0$ for $5d \leq 4g-1$.
\end{lemma}
\begin{proof}
In addition to what was shown in Lemma \ref{lem:VanishingRelE2HomologyMCG}, we must show that we also have $H_{4,3}^{E_2}(\gR_\bQ, \gA)=0$. As the map \eqref{eq:RelHurewicz} is $(4,4,4,4,3,3,\ldots)$-connective, the relative Hurewicz map
\[H_{4,3}(\gR_\bQ, \gA) \lra H_{4,3}^{E_2}(\gR_\bQ, \gA)\]
is surjective. The domain of this map fits into a long exact sequence
\[H_{4,3}(\gR_\bQ) \lra H_{4,3}(\gR_\bQ, \gA) \overset{\partial}\lra H_{4,2}(\gA) \overset{c_*}\lra H_{4,2}(\gR_\bQ),\]
where the leftmost term is zero by Theorem \ref{thm:H43}. Using the cell attachment spectral sequence as in the proof of Lemma \ref{lem:VanishA}, we have
\[H_{4,2}(\gA) = \bQ\{\sigma \lambda\}.\]
By Lemma \ref{lem:MCGFacts}  (\ref{it:MCGFacts42}) we have $H_{4,2}(\gR_\bQ) = \bQ$, and as $H_{3,2}(\gR_\bQ) = \bQ\{\lambda\}$ by Lemma \ref{lem:MCGFacts} (\ref{it:MCGFacts32abs}) and $H_{4,2}(\gR_\bQ/\sigma)=0$ by Corollary \ref{cor:MCGStabRange}, it follows that $H_{4,2}(\gR_\bQ) = \bQ\{\sigma \lambda\}$ too. Thus the leftmost map is injective, and so $H_{4,3}(\gR_\bQ, \gA)=0$ as required.
\end{proof}

Now, the map $[\sigma, \lambda] \colon S^{4,3}\to \gR_\bQ$ (depending on previous choices of representatives of $[-,-]$, $\sigma$ and $\lambda$) represents a class in the group $H_{4,3}(\gR_\bQ)$, which vanishes by Theorem \ref{thm:H43}. Choosing a null-homotopy of this map, we can extend $c$ to a map
\[d \colon \gB \coloneqq \gE_2(S^{1,0}_\bQ\sigma  \oplus S^{3,2}_\bQ\lambda) \cup^{E_2}_{[\sigma,\sigma]} \gD^{2,2}_\bQ \rho \cup^{E_2}_{[\sigma, \lambda]} \gD^{4,4}_\bQ\rho' \lra \gR_\bQ.\]
The $E_2$-homology of $\gA$ and $\gB$ are identical below slope 1, so we have
\begin{equation}\label{eq:RelB}
H_{g,d}^{E_2}(\gR_\bQ, \gB)=0 \text{ for } 5d \leq 4g-1.
\end{equation}
We have the following analogue of Lemma \ref{lem:VanishA}.

\begin{lemma}\label{lem:VanishB}
$H_{g,d}(\overline{\gB}/(\sigma, \lambda))=0$ for $d < \tfrac{4}{5}g$.
\end{lemma}
\begin{proof}
This argument is similar to the proof of Lemma \ref{lem:VanishA}. As in that proof the skeletal filtration gives rise to two (tri-graded) spectral sequences converging to the homology over $\overline{\gB}$ and $\overline{\gB}/(\sigma,\lambda)$ respectively. The former is multiplicative and the latter is a module over the former. 

The spectral sequence converging to $H_{g,d}(\overline{\gB}/(\sigma, \lambda))$ has $E^1$-page given by 
\[E^1_{g,p,q} = H_{g,p+q,q}\left(\overline{\gE_2(S^{1,0,0}_\bQ \sigma \oplus S^{3,2,2}_\bQ \lambda \oplus S^{2,2,2}_\bQ \rho \oplus S^{4,4,4}_\bQ \rho')}/(\sigma, \lambda)\right).\]
This is the free graded commutative algebra on $L/\langle \sigma,\lambda \rangle$, where $L$ is the free $1$-Lie algebra on $\sigma$, $\lambda$, $\rho$ and $\rho'$. This is a module spectral sequence over that for $H_{g,d}(\overline{\gB})$, whose first page is the free graded commutative algebra on $L$ and on which the $d^1$-differential satisfies $d^1(\sigma) = 0$, $d^1(\lambda) = 0$, $d^1(\rho) = [\sigma,\sigma]$ and $d^1(\rho') = [\sigma,\lambda]$. 

To compute $E^2_{\ast,\ast,\ast}$, we add an additional ``computational'' filtration on $E^1_{\ast,\ast,\ast}$. The chain complex $F^\bullet E^1_{\ast,\ast,\ast}$ is given by giving $[\sigma,\sigma]$, $[\sigma,\lambda]$, $\rho$ and $\rho'$ filtration degree $0$, giving the remaining homogeneous basis elements filtration degree equal to their homological degree, and extending the grading multiplicatively. The $d^1$-differential preserves this filtration.

We obtain a spectral sequence converging to $E^2_{\ast,\ast,\ast}$ with first page given by the associated graded $\grr(F^\bullet E^1_{\ast,\ast,\ast})$. This can be identified with the tensor product of chain complexes
\[\left(\Lambda_\bQ(\bQ \{[\sigma,\sigma],[\sigma,\lambda],\rho,\rho'\}),\delta\right) \otimes \left(\Lambda_\bQ(L/\langle \sigma,\lambda,[\sigma,\sigma],[\sigma,\lambda],\rho,\rho' \rangle),0\right)\]
where $\delta$ is determined by $\delta([\sigma,\sigma]) = 0$, $\delta([\sigma,\lambda]) = 0$, $\delta(\rho) = [\sigma,\sigma]$ and $\delta(\rho') = [\sigma,\lambda]$. The right hand term is generated by elements satisfying $\frac{d}{g} \geq \frac{4}{5}$. The left hand term has homology given by $\bQ$ in degree $(0,0,0,0)$.

Thus $E^2_{g,p,q}$ vanishes for $\frac{p+q}{g} < \frac{4}{5}$, and hence $H_{g,d}(\overline{\gB}/(\sigma, \lambda))=0$ for $\frac{d}{g}<\frac{4}{5}$.
\end{proof}

We now conclude the argument as in Section \ref{sec:RatStab}: apply Theorem~$E_k$.15.4 to $B(\overline{\gR}_\bQ;\overline{\gB};\overline{\gB}/(\sigma, \lambda)) \simeq \overline{\gR}_\bQ/(\sigma, \lambda)$, using that $H^{E_2}_{g,d}(\gB)=0$ for $d < g-1$, \eqref{eq:RelB}, and Lemma \ref{lem:VanishB}. The conclusion is that $H_{g,d}(\overline{\gR}_\bQ/(\sigma, \lambda))=0$ for $d < \tfrac{4}{5}g$, so the long exact sequence on homology associated to the homotopy cofibre sequence \eqref{eq:MultKappa} gives the following.

\begin{corollary}\label{cor:MCGSecStabImprov}
The map
\[\lambda \cdot - \colon H_{d-2}(\Gamma_{g-3,1}, \Gamma_{g-4,1};\bQ) \lra H_{d}(\Gamma_{g,1}, \Gamma_{g-1,1};\bQ)\]
is an epimorphism for $5d \leq 4g-1$ and an isomorphism for $5d \leq 4g-6$.
\end{corollary}

Upon obtaining this strengthened rational secondary stability range, one can feed it back into our discussion of mapping class groups of surfaces with punctures and boundary, and with local coefficients. Upon working rationally we obtain improvements to the ranges of Corollary \ref{cor:punctures-secondary-stability} and Theorem \ref{thm:TwistedSecondStab}: in Corollary \ref{cor:punctures-secondary-stability} the secondary stability map is an epimorphism for $5d \leq 4g-1$ and an isomorphism for $5d \leq 4g-6$, and in the assumptions of Theorem \ref{thm:TwistedSecondStab} one may replace the assumption $\lambda \leq \tfrac{3}{4}$ by $\lambda \leq \tfrac{4}{5}$. Similarly the rational secondary stability range in Proposition \ref{prop:StabH1} improves to an epimorphism for $5d \leq 4g -4$ and an isomorphism for $5d \leq 4g -9$. Using this we can do a new computation.

\begin{corollary}\label{cor:H3twistedH1}
	For $g \geq 5$, $H_3(\Gamma_{g,1} ; H_1(\Sigma_{g,1};\bQ))\cong\bQ^2$.
\end{corollary}

\begin{proof}
	The proof is similar to that of Corollary \ref{cor:H1Vanish42}. Taking $\mathbbm{k} = \bQ$, $\varphi$ coincides with multiplication by $\tfrac{1}{10}\lambda \in H_{3,2}(\overline{\gR}_\bk)$. We can form the commutative diagram
	\begin{equation*}
	\begin{tikzcd}
	H_{3,2}(\gR_{\gH_1}/\sigma) \rar{\partial} \dar{\tfrac{1}{10}\lambda \cdot -}& H_{2,1}(\gR_{\gH_1}) \rar{\sigma \cdot -} \dar{\tfrac{1}{10}\lambda \cdot -}& H_{3,1}(\gR_{\gH_1}) \dar{\tfrac{1}{10}\lambda \cdot -}\\
	H_{6,4}(\gR_{\gH_1}/\sigma) \rar{\partial}& H_{5,3}(\gR_{\gH_1}) \rar{\sigma \cdot -}& H_{6,3}(\gR_{\gH_1}),
	\end{tikzcd}
	\end{equation*}
	and by the improved local coefficient secondary stability range the left vertical map is surjective, as $5 \cdot 4 \leq 4\cdot 6 -4$. By \cite{MoritaTwisted} and \cite[Section 6]{MoritaJacobian}, $H_{2,1}(\gR_{\gH_1}) \to H_{3,1}(\gR_{\gH_1})$ is an isomorphism. Hence the top left-hand map and thus also the bottom left-hand map are zero, and so
	\[\sigma \cdot - \colon H_{5,3}(\gR_{\gH_1}) \lra H_{6,3}(\gR_{\gH_1})\]
	is injective. It also surjective, using Proposition \ref{prop:StabH1} and $3 \cdot 3 \leq 2 \cdot 6 -2$. For $g \geq 7$, Proposition \ref{prop:StabH1} says that stabilisation $\sigma \cdot - \colon H_{g-1,3}(\gR_{\gH_1}) \to H_{g,3}(\gR_{\gH_1})$ is an isomorphism and the stable homology is $\bQ^2$ by \cite[Theorem 1.B]{KawazumiTwisted}.
\end{proof}

\subsection{Proof of Theorem \ref{thm:H43}} \label{sec:proofthmh43} It remains to supply the proof of Theorem \ref{thm:H43}, which uses several classical techniques for computing in the homology of mapping class groups in addition to Corollary \ref{cor:H1Vanish42}. We refer to \cite{Morita} for details on many of the constructions below.
	
	Write $\Gamma_g = \pi_0(\Diff^+(\Sigma_g))$ for the mapping class group of a \emph{closed} genus $g$ surface, i.e.\ the group of isotopy classes of orientation-preserving diffeomorphisms. Write $\Gamma_g^1 = \pi_0(\Diff^+(\Sigma_g, *))$ for the mapping class group of a {closed} genus $g$ surface with a marked point, i.e.\ the group of isotopy classes of orientation-preserving diffeomorphisms of $\Sigma_g$ which fix a distinguished point $* \in \Sigma_g$. There are group extensions
	\begin{align}
	\pi_1(\Sigma_g, *) & \lra \Gamma_{g}^1 \overset{\pi}\lra \Gamma_g\label{eq:Ext1}\\
	\bZ & \lra \Gamma_{g,1} \overset{p}\lra \Gamma_g^1,\label{eq:Ext2}
	\end{align}
	where the second is a central extension. Let $e \in H^2(\Gamma_{g}^1 ; \bZ)$ denote the class of the central extension \eqref{eq:Ext2}. As the group homology of $\pi_1(\Sigma_g, *)$ is the same as the homology of $\Sigma_g$ (as long as $g>0$), the Lyndon--Hochschild--Serre spectral sequence for the extension \eqref{eq:Ext1} has $E^2$-page given by
	\[E_2^{p,q} = \begin{cases}
	H^p(\Gamma_g ; \bZ) & \text{if $q=2$,}\\
	0 & \text{if $q > 2$,}
	\end{cases}
	\]
	so there is an (upper) edge homomorphism, giving rise to a \emph{Gysin map}
	\[\pi_! \colon H^n(\Gamma_g^1 ; \bZ) \lra E_\infty^{n-2,2} \subset E_2^{n-2,2} = H^{n-2}(\Gamma_g ; \bZ).\]
	This satisfies $\pi_!(\pi^*(x) \cdot y) = x \cdot \pi_!(y)$, using the multiplicative structure of the spectral sequence. It also satisfies $\pi_!(e) = \chi(\Sigma_g)=2-2g \in H^0(\Gamma_g;\bZ)$.
	
	The \emph{Miller--Morita--Mumford classes} are $\kappa_i \coloneqq \pi_!(e^{i+1}) \in H^{2i}(\Gamma_g ; \bZ)$. We also write $\kappa_i$ for $\pi^*(\kappa_i) \in H^{2i}(\Gamma_g^1 ; \bZ)$.

\begin{proof}[Proof of Theorem \ref{thm:H43}] Let us start by collecting some properties of Miller--Morita--Mumford classes from the literature. We shall freely use that the moduli space of genus $g$ Riemann surfaces has the same rational cohomology as $\Gamma_g$ for $g>1$; some results in the literature are stated in these terms.

\vspace{1ex}

\noindent\textbf{Claim 1}: $3 \kappa_1^2 + 32 \kappa_2 = 0 \in H^4(\Gamma_4 ; \bQ)$.

This relation may be found in \cite[p.\ 10]{Faber} (where the $\kappa_i$ there correspond to $(-1)^{i+1} \kappa_i$ in our notation), or in \cite[Example 2.4]{RWRels} or \cite[page 810]{MoritaRels}.

\vspace{1ex}

\noindent\textbf{Claim 2}: $\kappa_2 \neq 0 \in H^4(\Gamma_4 ; \bQ)$.

This is an instance of \cite[Theorem 2]{Faber}.

\vspace{1ex}

\noindent\textbf{Claim 3}: $H^2(\Gamma_4 ; \bQ) = \bQ\{\kappa_1\}$.

By Claims 1 and 2, $\kappa_1^2 \neq 0$ and so $\kappa_1 \neq 0$, so it is enough to show that this cohomology group has dimension (at most) 1. This follows from \cite[Theorem 1.4]{TomassiM4}, but can also be deduced from either of the stability theorems \cite{Boldsen, R-WResolution}, which show that the maps
\[H^2(\Gamma_4 ; \bQ) \lra H^2(\Gamma_{4,1} ; \bQ) \longleftarrow H^2(\Gamma_{5,1} ; \bQ)\]
are isomorphisms, and Harer's theorem \cite{HarerH2} that the rightmost group is $\bQ$.

\vspace{1ex}

\noindent\textbf{Claim 4}: $H^2(\Gamma_4^1 ; \bQ) = \bQ\{\kappa_1, e\}$.

Using the transfer, Kawazumi and Morita have shown \cite[Proposition 5.2]{KM} \cite[Proposition 3.1]{Morita} that the Lyndon--Hochschild--Serre spectral sequence for the extension \eqref{eq:Ext1} gives, for any $\Gamma_g$-module $M$, a splitting
\begin{equation}\label{eqn:morita-splitting}H^p(\Gamma_g^1 ; M) \cong H^p(\Gamma_4 ; M) \oplus H^{p-1}(\Gamma_g ; H_1(\Sigma_g)\otimes M ) \oplus e \cdot H^{p-2}(\Gamma_g ; M).\end{equation}
Apply this with $g=4$, $M=\bQ$, and $p=2$, and use that Morita has also shown \cite[Proposition 5.1]{MoritaTwisted} that $H^1(\Gamma_4 ; H_1(\Sigma_4;\bQ))=0$: interpreting this splitting and using Claim 3 proves Claim 4.

\vspace{1ex}

\noindent\textbf{Claim 5}: $H^3(\Gamma_4^1 ; \bQ) = 0$.

It follows from \cite[Theorem 1.4]{TomassiM4} that $H^3(\Gamma_4 ; \bQ)=0=H^1(\Gamma_4 ; \bQ)$, so (\ref{eqn:morita-splitting}) gives $H^3(\Gamma_4^1 ; \bQ) \cong H^2(\Gamma_4 ; H_1(\Sigma_4;\bQ))$. Applying (\ref{eqn:morita-splitting}) again with $M=H_1(\Sigma_4;\bQ)$ shows that this group is a summand of $H^2(\Gamma_4^1 ; H_1(\Sigma_4;\bQ))$. The Gysin sequence for the central extension \eqref{eq:Ext2} has the form
\[H^0(\Gamma_4^1 ; H_1(\Sigma_4;\bQ)) \overset{e \cdot -}\lra H^2(\Gamma_4^1 ; H_1(\Sigma_4;\bQ)) \lra H^2(\Gamma_{4,1} ; H_1(\Sigma_4;\bQ)),\]
where the leftmost term vanishes by Lemma \ref{lem:H1CoinvVanish} and the rightmost term vanishes by Corollary \ref{cor:H1Vanish42}, both dualised using the isomorphism $H_*(G,M)^\vee \cong H^*(G,M^\vee)$ given by the universal coefficient theorem, together the identification $H_1(\Sigma_4;\bQ)^\vee \cong H_1(\Sigma_4;\bQ)$ as a representation of $\Gamma_4$ using the symplectic form. This proves the claim.

\vspace{1ex}

We now begin the proof of the theorem; by the universal coefficient theorem it is enough to prove the vanishing of $H^3(\Gamma_{4,1};\bQ)$. The Gysin sequence for the central extension \eqref{eq:Ext2} is
\[H^3(\Gamma_{4}^1;\bQ) \overset{p^*}\lra H^3(\Gamma_{4,1};\bQ) \overset{p_!}\lra H^2(\Gamma_{4}^1;\bQ) \overset{e \cdot -}\lra H^4(\Gamma_{4}^1;\bQ).\]
The leftmost term is 0 by Claim 5, so to prove the theorem we must show that the rightmost map is injective: suppose therefore that $Ae + B \kappa_1 \in \bQ\{\kappa_1, e\} = H^2(\Gamma_4^1 ; \bQ)$ lies in the kernel of the rightmost map. We shall apply $\pi_!$ for $\pi \colon B\Gamma^1_4 \to B\Gamma_4$, the universal surface bundle, to expressions built from $Ae+B\kappa_1$. As $e \cdot (Ae + B \kappa_1) = 0$, applying $\pi_!(-)$ gives
\begin{align*}
0 &= \pi_!(A e^2 + B e \kappa_1)\\
 &= A \kappa_1 + B \kappa_1 \cdot \pi_!(e)\\
 &= A \kappa_1 + B (2-2 \cdot 4) \kappa_1\\
&= (A - 6 B) \kappa_1.
\end{align*}
As $\kappa_1 \neq 0$, we must have $A=6B$. On the other hand, we also have $e^2 \cdot (Ae + B \kappa_1)=0$ and so applying $\pi_!(-)$ gives
\begin{align*}
0 &= \pi_!(A e^3 + B e^2 \kappa_1)\\
 &= A \kappa_2 + B \kappa_1^2\\
 &= B(6 \kappa_2 + \kappa_1^2).
\end{align*}
If $B \neq 0$ then $6 \kappa_2 = -\kappa_1^2$, which combined with the relation $3 \kappa_1^2 =- 32 \kappa_2$ shows that $\kappa_2=0$, contradicting Claim 2. Thus $B=A=0$, as required.
\end{proof}

\subsection{Some non-vanishing relative groups} Using the techniques of Section \ref{sec:proofthmh43}, we prove that some relative homology groups do not vanish. 

The homology group $H_4(\Gamma_{6,1};\bQ)$ is in the stable range and it follows from the Madsen--Weiss theorem \cite{MadsenWeiss} that $H^4(\Gamma_{6,1};\bQ) = \bQ\{\kappa_2, \kappa_1^2\}$. Corollary \ref{cor:MCGSecStabImprov} implies that the map $H_4(\Gamma_{5,1};\bQ) \to H_4(\Gamma_{6,1};\bQ)$ has 1-dimensional cokernel, which means that $H_4(\Gamma_{5,1};\bQ)$ contains a rational homology class $x$ which pairs non-trivially with some linear combination of $\kappa_2$ and $\kappa_1^2$. (This may also be deduced from the facts that $H_4(\Gamma_{5,1};\bQ)$ surjects onto $H_4(\Gamma_5;\bQ)$ by \cite{R-WResolution} and that the tautological ring does not vanish in $H^4(\Gamma_5;\bQ)$ \cite{Faber4}.) By the relation $5 \kappa_1^2 + 72  \kappa_2=0$ in $H^4(\Gamma_{5};\bQ)$ \cite[Example 9.2]{MoritaRels}, we may assume that $\langle x,\kappa_2 \rangle \neq 0$. An element of $H_d(\Gamma_{g,1};\bQ)$ is said to \emph{destabilise} if it is in the image of the stabilisation map.

\begin{lemma}The element $\lambda^3 \cdot x^2 \in H_{14}(\Gamma_{19,1};\bQ)$ does not destabilise (so in particular is non-zero).
\end{lemma}

\begin{proof}
Morita's third relation for $k=7$ from \cite{MoritaRels} gives a polynomial $R = R(\kappa_1, \ldots, \kappa_7)$, of cohomological degree 14, which vanishes in the cohomology of $\Gamma_{18}$, hence also in the cohomology of $\Gamma_{18,1}$. If the class $\lambda^3 \cdot x^2$ destabilised, i.e.~ $\lambda^3 \cdot x^2 = \sigma \cdot y$ for some $y \in H_{14}(\Gamma_{18,1};\bQ)$, then as $\langle y, R \rangle=0$ we must also have $\langle \lambda^3 \cdot x^2, R \rangle=0$.
	
	This pairing may be computed as the pairing of $\lambda \otimes \lambda \otimes \lambda \otimes x \otimes x$ against the 5-fold coproduct of $R$, where the comultiplication on the cohomology of $\gR_\bQ$ is induced by the multiplication on $\gR_\bQ$. Since the Miller--Morita--Mumford classes $\kappa_i$ for $i>0$ are primitive with respect to this comultiplication, and the classes $\lambda$ and $x$ pair non-trivially only with $\kappa_1$ and $\{\kappa_2, \kappa_1^2\}$ respectively, we may discard terms involving $\kappa_i$ for $i>2$ in $R$ when computing this evaluation. Upon discarding these terms, $R$ is given (up to a non-zero scalar) by
	\[80435 \cdot \kappa_1^7 + 21719880 \cdot \kappa_1^5 \kappa_2 + 1387036224 \cdot \kappa_1^3 \kappa_2^2 + 17581100544 \cdot \kappa_1 \kappa_2^3.\]
	When evaluating the 5-fold coproduct of this against $\lambda \otimes \lambda \otimes \lambda \otimes x \otimes x$, only terms of the form $\kappa_1 \otimes \kappa_1 \otimes \kappa_1 \otimes w_1 \otimes w_2$ with $w_1,w_2 \in \{\kappa_1^2,\kappa_2\}$ contribute. These terms are
	\begin{align*}
101348100 \cdot \kappa_1 \otimes \kappa_1 \otimes \kappa_1 \otimes \kappa_1^2 \otimes \kappa_1^2 &+ 1303192800 \cdot \kappa_1 \otimes \kappa_1 \otimes \kappa_1 \otimes \kappa_1^2 \otimes \kappa_2 \\
+ 1303192800 \cdot \kappa_1 \otimes \kappa_1 \otimes \kappa_1 \otimes \kappa_2 \otimes \kappa_1^2 
	&+ 16644434688 \cdot \kappa_1 \otimes \kappa_1 \otimes \kappa_1 \otimes \kappa_2 \otimes \kappa_2.\end{align*}
	Because $\langle x,\kappa_1^2 \rangle = -\frac{72}{5} \langle x,\kappa_2\rangle$, the pairing is therefore $128024064 \cdot \langle \lambda,\kappa_1 \rangle^3 \langle x,\kappa_2 \rangle^2$. This does not vanish since $\langle \lambda,\kappa_1 \rangle \neq 0$ and $\langle x,\kappa_2 \rangle \neq 0$.
\end{proof}

This observation allows us to find the first non-vanishing relative rational homology groups in each genus $g \equiv 1,2 \pmod 3$. Let us denote the image of $x^i \in H_{4i}(\Gamma_{5i,1};\bQ)$ in $H_{4i}(\Gamma_{5i,1},\Gamma_{5i-1,1};\bQ)$ by $z_i$.

\begin{corollary}\label{cor:non-vanishing-rel-groups} We have that $\lambda^r \cdot z_1$ is non-zero in $H_{5+3r,4+2r}(\gR_\bQ/\sigma)$ for all $r \geq 0$. We have that $\lambda^r \cdot z_2$ is non-zero in $H_{10+3r,8+2r}(\gR_\bQ/\sigma)$ for all $r \geq 0$.\end{corollary}

\begin{proof}
By the improved secondary stability range of Corollary \ref{cor:MCGSecStabImprov}, for the first claim it is enough to show that $\lambda^2 \cdot z_1$ is non-zero in $H_{11,8}(\gR_\bQ/\sigma)$, because $5 \cdot 10 \leq 4 \cdot 14 - 6$. This is a consequence of $\lambda^2 \cdot x$ not destabilising (otherwise $\lambda^3 \cdot x^2$ would).
	
Similarly, for the second claim it is enough to show that $\lambda^3 \cdot z_2$ is non-zero in $H_{11,8}(\gR_\bQ/\sigma)$, because $5 \cdot 14 \leq 4 \cdot 19 - 6$. This is a consequence of $\lambda^3 \cdot x^2$ not destabilising.
\end{proof}

\begin{remark}
Further computer calculations, involving expressions too large to print here, show that $\lambda^3 \cdot x^6$ does not destabilise. The above argument then implies that $\lambda^r \cdot x^s$ does not destabilise for $r \geq 0$ and $s \leq 6$, resulting in additional infinite families of non-vanishing relative homology groups.
\end{remark}

\subsection{A summary of low-genus rational homology}\label{sec:table} We start by proving Theorem \ref{thm:B}, which we recall says that
\begin{enumerate}[(i)]
	\item $H_d(\Gamma_{g,1}, \Gamma_{g-1,1};\bZ)=0$ for $d \leq \tfrac{2g-1}{3}$.
	
	\item $H_{2k}(\Gamma_{3k,1}, \Gamma_{3k-1,1};\bZ) \cong \bZ$ for $k \geq 0$.
	
	\item $H_{2k+1}(\Gamma_{3k+1,1}, \Gamma_{3k,1};\bQ)=0$ for $k \geq 1$ (but $H_1(\Gamma_{1,1},\Gamma_{0,1};\bQ) = \bQ$).
	
	\item $H_{2k+2}(\Gamma_{3k+2,1}, \Gamma_{3k+1,1};\bQ) \neq 0$ for $k \geq 0$.
	
	\item $H_{2k+2}(\Gamma_{3k+1,1}, \Gamma_{3k,1};\bQ) \neq 0$ for $k \geq 3$ (but we have $H_{2}(\Gamma_{1,1}, \Gamma_{0,1};\bQ) = 0$ and $H_4(\Gamma_{4,1},\Gamma_{3,1};\bQ) \neq 0$, the case $k=2$ remaining unknown for now).
\end{enumerate}

\begin{proof}[Proof of Theorem \ref{thm:B}] Part (\ref{it:B1}) is Corollary \ref{cor:MCGStabRange} and part (\ref{it:B2}) is Corollary \ref{cor:rel-integral}. 
	
For part (\ref{it:B3}), we use that in the exact sequence
%	\[H_3(\Gamma_{4,1};\bQ) \lra H_3(\Gamma_{4,1}, \Gamma_{3,1};\bQ) \lra H_2(\Gamma_{3,1};\bQ) \lra H_2(\Gamma_{4,1};\bQ) \lra H_2(\Gamma_{4,1}, \Gamma_{3,1};\bQ)\]
	\begin{equation*}
	\begin{tikzcd}
	& & H_3(\Gamma_{4,1};\bQ) \rar \ar[draw=none]{d}[name=X, anchor=center]{} & H_3(\Gamma_{4,1}, \Gamma_{3,1};\bQ)
	\ar[rounded corners,
	to path={ -- ([xshift=2ex]\tikztostart.east)
		|- (X.center) \tikztonodes
		-| ([xshift=-2ex]\tikztotarget.west)
		-- (\tikztotarget)}]{dll} \\
	& H_2(\Gamma_{3,1};\bQ) \rar & H_2(\Gamma_{4,1};\bQ) \rar & H_2(\Gamma_{4,1}, \Gamma_{3,1};\bQ)
	\end{tikzcd}
	\end{equation*}
the first term vanishes by Theorem \ref{thm:H43}, the last term vanishes by Corollary \ref{cor:MCGStabRange}, and $H_2(\Gamma_{3,1};\bQ) \cong \bQ \cong H_2(\Gamma_{4,1};\bQ)$ by Lemma \ref{lem:MCGFacts} (\ref{it:MCGFacts32abs}) and (\ref{it:MCGFacts42}). It follows that $H_3(\Gamma_{4,1}, \Gamma_{3,1};\bQ)=0$, and so by Corollary \ref{cor:MCGSecStab} we have
	\[H_{2k+1}(\Gamma_{3k+1,1}, \Gamma_{3k,1};\bQ)=0 \text{ for all } k \geq 1.\]

Parts (\ref{it:B4}) and (\ref{it:B5}) follow from Corollary \ref{cor:non-vanishing-rel-groups}. The parenthesised statements follow from a variety of computations, collected in Figure \ref{fig:rat}.
\end{proof}

\begin{figure}[h]
	\begin{tikzpicture}
	\begin{scope}[scale=.9]
	\clip (-.75,-.75) rectangle (13.5,8.5);
	\draw (-.5,0)--(11.5,0);
	\draw (0,-1) -- (0,8.5);
	\fill [pattern=north west lines, pattern color=black!20!white] ({3.5*1.25},8.5) -- ({3.5*1.25},3.5) -- ({1.25*18.5/4},3.5) -- ({9.5*1.25},{1/5*(4*9.5-1)}) -- ({9.5*1.25},8.5)-- cycle;
	\fill [fill=Mahogany!5!white] (-1.25,-1) -- ({9.5*1.25},{1/3*(2*9.5-1)}) -- ({9.5*1.25},{1/5*(4*9.5-1)}) -- cycle;
	\begin{scope}
	\foreach \s in {0,...,8}
	{
		\draw [dotted] (-.5,\s)--(11.5,\s);
		\node [fill=white] at (-.25,\s) [left] {\tiny $\s$};
	}

	\foreach \s in {1,...,9}
	{
		\draw [dotted] ({1.25*\s},-0.5)--({1.25*\s},8.5);
		\node [fill=white] at ({1.25*\s},-.5) {\tiny $\s$};
		\node [fill=white] at ({1.25*\s},0) {$\bQ$};
		
	}
	\node [fill=Mahogany!5!white] at ({1.25*3},2) {\hspace{.5em}};	
	\node [fill=Mahogany!5!white] at ({1.25*6},4) {\hspace{1em}};
	\node [fill=Mahogany!5!white] at ({1.25*9},6) {\hspace{1em}};
	\end{scope}
	\node [fill=white] at (0,0) {$\bQ$};
	\node at (0,-.5) {\tiny 0};

	\draw [->>,thick] (.25,0) --++ (.75,0);
	\foreach \s in {2,...,9}
	{
		\draw [->,thick] ({1.25*\s-1},0) -- node[above] {\tiny $\cong$} ++ (.75,0);
	}
	
	\node [fill=white] at (1.25,1) {$\bQ$};
	\draw [->,thick] (2.75,3) --  node[above] {\tiny ?} ++ (.75,0);
	
	\foreach \s in {3,5,6}
	{
		\node [fill=white] at (3.75,\s) {$\bQ$};
	}
	\node [fill=white] at (2.5,3) {$\bQ$};
	
	\foreach \s in {4,...,9}
	{
		\node [fill=white] at ({1.25*\s},2) {$\bQ$};
	}
	\node at ({1.25*3},2) {$\bQ$};
	\foreach \s in {7,8,9}
	{
		\node at ({1.25*\s},4) {$\bQ^2$};
	}
	\node  at ({1.25*6},4) {$\bQ^2$};
	\draw [->>,thick] ({1.25*6+0.25},4) --++ (.75,0);
	\draw [->>,thick] ({1.25*6+0.25},5) --++ (.75,0);
	\draw [->,thick] ({1.25*5+0.25},4) -- node[above] {\tiny rank 1} ++(.75,0);
	\node  at ({1.25*9},6) {$\bQ^3$};
	\draw [->,thick] ({1.25*8+0.25},6) -- node[above] {\tiny rank 2} ++(.75,0);
	
	\foreach \s in {4,...,8}
	{
		\node at ({1.25*4},\s) {?};
		\node at ({1.25*5},\s) {?};
	}
	\foreach \s in {5,...,8}
	{
		\node at ({1.25*6},\s) {?};
		\node at ({1.25*7},\s) {?};
	}
	\foreach \s in {6,...,8}
	{
		\node at ({1.25*8},\s) {?};
	}
	\foreach \s in {7,8}
	{
		\node at ({1.25*9},\s) {?};
	}
	
	\foreach \s in {8,9}
	{
		\draw [->,thick] ({1.25*\s-1},4) -- node[above] {\tiny $\cong$} ++ (.75,0);
	}
	\draw [->>,thick] (4,2) --++ (.75,0);
	\foreach \s in {5,...,9}
	{
		\draw [->,thick] ({1.25*\s-1},2) -- node[above] {\tiny $\cong$} ++ (.75,0);
	}
	
	\draw [very thick,Mahogany,densely dotted] (-1.25,-1) -- ({9.5*1.25},{1/3*(2*9.5-1)}) node [right] {\tiny $d = \frac{2g-1}{3}$};
	\draw [very thick,Mahogany,densely dotted] (-1.25,-1) -- ({9.5*1.25},{1/5*(4*9.5-1)}) node [right] {\tiny $d = \frac{4g-1}{5}$};
	
	\node [fill=white] at (-.5,-.5) {$\nicefrac{d}{g}$};
	\end{scope}
	\end{tikzpicture}
	\caption{A summary of the low-degree low-genus rational homology of $\Gamma_{g,1}$ and the stabilisation maps.}
	\label{fig:rat}
\end{figure}
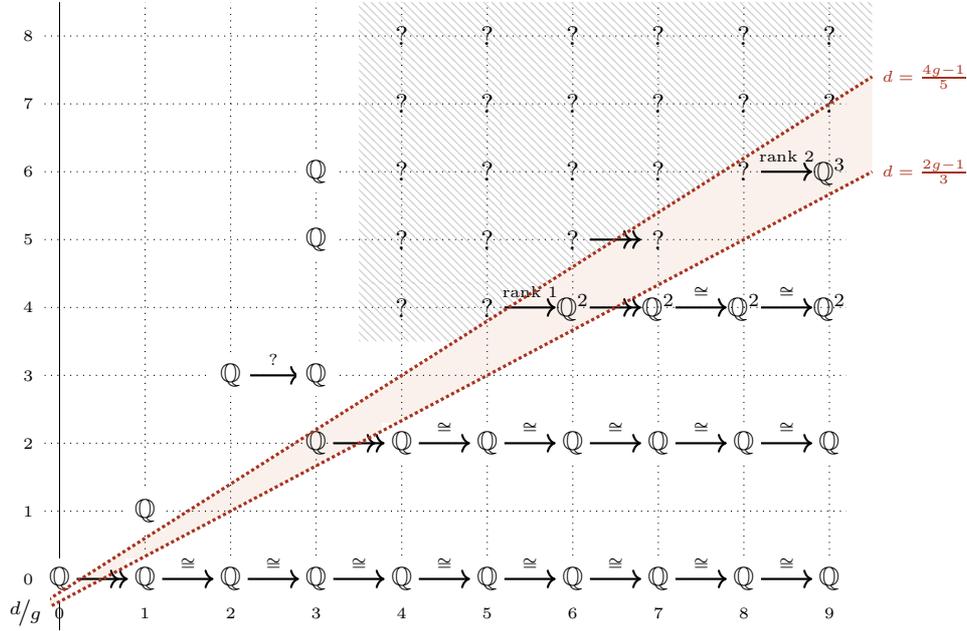

We used this to obtain Figure \ref{fig:rat}, which contains our current knowledge of the low-genus rational homology of the groups $\Gamma_{g,1}$ and the stabilisation maps. The shaded part is our rational secondary homological stability range; everything below it is in our improved homological stability range and given by Madsen--Weiss' solution of the Mumford conjecture \cite{MadsenWeiss}. We have also included low genus computations, obtained from \cite{ABE,GodinMCG,WangThesis,BoesHermannThesis} in addition to Lemma \ref{lem:MCGFacts} and Theorem \ref{thm:H43}. The dashed part consists of unknown groups outside our secondary stability range, empty entries denote zero groups, and question marks denote unknown groups or stabilisation maps.

It remains to explain how we obtained the groups and stabilisation maps in the shaded range. By ordinary homological stability the stabilisation map $H_4(\Gamma_{6,1};\bQ) \to H_4(\Gamma_{7,1};\bQ)$ is surjective, and the target is in the stable range so is isomorphic to $\bQ^2$ by the Madsen--Weiss theorem. By Theorem \ref{thm:B} (\ref{it:B3}), $H_5(\Gamma_{7,1},\Gamma_{6,1};\bQ) = 0$, so in the long exact sequence
\[H_5(\Gamma_{7,1},\Gamma_{6,1};\bQ) \lra H_4(\Gamma_{6,1};\bQ) \lra H_4(\Gamma_{7,1};\bQ) \lra 0\]
the right map has to be injective as well and $H_4(\Gamma_{6,1};\bQ) \cong \bQ^2$. To see that the map $H_4(\Gamma_{5,1};\bQ) \to H_4(\Gamma_{6,1};\bQ)$ has rank 1, we use that $H_3(\Gamma_{5,1};\bQ) = 0$ and Theorem \ref{thm:B} (ii) says $H_4(\Gamma_{6,1},\Gamma_{5,1};\bQ) \cong \bQ$. Analogous arguments two degrees up imply $H_6(\Gamma_{9,1};\bQ) \cong \bQ^3$, and that the stabilisation map $H_6(\Gamma_{8,1};\bQ) \to H_6(\Gamma_{9,1};\bQ)$ has $1$-dimensional cokernel. Similarly, the map $H_3(\Gamma_{4,1},\Gamma_{3,1};\bQ) \to H_5(\Gamma_{7,1},\Gamma_{6,1};\bQ)$ is surjective with domain $0$, so that $H_5(\Gamma_{6,1};\bQ) \to H_5(\Gamma_{7,1};\bQ)$ is surjective.

\begin{remark}It seems unknown whether stabilisation $H_3(\Gamma_{2,1};\bQ) \to H_3(\Gamma_{3,1};\bQ)$ is an isomorphism. The class in $H_3(\Gamma_{2,1};\bQ)$ is \emph{not} the Browder bracket of the generator $\tau$ of $H_1(\Gamma_{1,1};\bQ)$ with itself as this vanishes by $[\tau,\tau] = (-1)^{|\tau||\tau|} [\tau,\tau]$. Finally, the iterated stabilisation map $H_6(\Gamma_{3,1};\bQ) \to H_6(\Gamma_{9,1};\bQ)$ vanishes. If it did not then $H_6(\Gamma_{3,1};\bQ)$ would be detected by some linear combination of $\kappa_1^3$, $\kappa_1 \kappa_2$, and $\kappa_3$. Then the tautological ring in $H^6(\Gamma_3;\bQ)$ would be non-zero, but it vanishes in that degree by \cite[Theorem 1.2]{LooijengaRel} or using \cite[Theorem 4.10]{Looijenga} and weight considerations. An interpretation of a generator of $H^6(\Gamma_3;\bQ)$ was recently given in \cite[Theorem 1.1]{ChanGalatiusPayne}.\end{remark}

\subsection*{Acknowledgements} 

AK and SG were supported by the European Research Council (ERC) under the European Union’s Horizon 2020 research and innovation programme (grant agreement No 682922), and by the Danish National Research Foundation through the Centre for Symmetry and Deformation (DNRF92). SG was also supported by NSF grant DMS-1405001. AK was also supported by NSF grant DMS-1803766. ORW was partially supported by EPSRC grant EP/M027783/1.

SG acknowledges conversations long ago with Mike Hopkins and Nathalie Wahl, about the possibility of secondary stability phenomena in mapping class groups and other groups.

%\printbibliography
\bibliographystyle{amsalpha}
\bibliography{biblio}

\vspace{1em}

\end{document}